\documentclass[oneside, a4paper, 11pt]{amsart}

\pdfoutput=1

\usepackage[utf8]{inputenc}
\usepackage[T1]{fontenc}

\usepackage{amsmath,amsthm,amssymb,bbm,comment,mathrsfs}

\usepackage[ttscale=.85]{libertine}
\usepackage{libertinust1math}
\DeclareSymbolFont{cmarrows}{OMS}{cmsy}{m}{n}
\DeclareMathSymbol{\mapstochar}{\mathrel}{cmarrows}{"37}

\usepackage[tracking=true]{microtype}
\SetTracking
  {encoding=T1, shape=sc}
  {10}
\SetProtrusion
  {encoding=T1,size={7,8}}
  {1={ ,750},2={ ,500},3={ ,500},4={ ,500},5={ ,500},
    6={ ,500},7={ ,600},8={ ,500},9={ ,500},0={ ,500}}

\usepackage[backend=bibtex, maxnames=6, maxalphanames=6, style=alphabetic, citestyle=alphabetic]{biblatex}
\DeclareFieldFormat*{title}{#1}
\bibliography{TheBigWrap}

\DeclareMathOperator{\colim}{colim}
\newcommand*{\cat}[1]{\ensuremath{\mathcal{#1}}} 
\newcommand{\mathscale}[2]{\scalebox{#1}{\mbox{\ensuremath{\displaystyle #2}}}}
\newcommand*{\id}{\mathrm{id}}
\renewcommand{\colon}{\!\nobreak\mskip2mu\mathpunct{}\nonscript%
\mkern-\thinmuskip{:}\mskip6muplus1mu\relax%
}
\newcommand{\from}{\colon}
\newcommand*{\defeq}{\mathrel{\vcenter{\baselineskip0.5ex \lineskiplimit0pt
    \hbox{\scriptsize.}\hbox{\scriptsize.}}}%
=}
\newcommand*{\inj}[1]{\ensuremath{#1\text{\rm -}\mathrm{inj}}}
\newcommand*{\proj}[1]{\ensuremath{#1\text{\rm -}\mathrm{proj}}}
\newcommand*{\lmod}[1]{\ensuremath{#1\text{\rm -}\mathrm{mod}_{\mathrm{fd}}}}
\newcommand*{\modc}[2][\cat{C}]{\ensuremath{\mathrm{mod}_{#1}(#2)}}
\newcommand*{\fdcomodc}[2][\cat{C}]{\ensuremath{\mathrm{fdcomod}_{#1}(#2)}}
\newcommand*{\comodc}[2][\cat{C}]{\ensuremath{\mathrm{comod}_{#1}(#2)}}
\newcommand*{\rmod}[1]{\ensuremath{\mathrm{mod}\text{\rm -}#1}}
\newcommand*{\lMod}[1]{\ensuremath{#1\text{\rm -}\mathrm{Mod}}}

\newcommand*{\lComod}[1]{\ensuremath{#1\text{\rm -}\mathrm{Comod}}}
\newcommand*{\rComod}[1]{\ensuremath{\mathrm{Comod}\text{\rm -}#1}}
\newcommand*{\biComod}[2]{\ensuremath{#1\text{\rm -}\mathrm{Bicomod}\text{\rm -}#2}}

\newcommand*{\trimod}[1][B]{\tetramodfd[#1][#1][][#1]}
\newcommand*{\Trimod}[1][B]{\tetramod[#1][#1][][#1]}
\newcommand*{\termod}[1][B]{\tetramodfd[#1][#1][#1][]}
\newcommand*{\Termod}[1][B]{\tetramod[#1][#1][#1][]}

\newcommand*{\bcomod}[1][B]{\tetramodfd[#1]}
\newcommand*{\bbcomod}[1][B]{\tetramod[#1]}
\newcommand*{\bmodule}[1][B]{\tetramodfd[][#1][][]}

\NewDocumentCommand{\tetramod}{O{} O{} O{} O{}}{%
  \leftidx{^{#1}_{#2}}{\mathbf{Vec}}{_{#3}^{#4}}%
}
\NewDocumentCommand{\tetramodfd}{O{} O{} O{} O{}}{%
  \leftidx{^{#1}_{#2}}{\mathbf{vecfd}}{_{#3}^{#4}}%
}
\newcommand*{\blank}{{-}}
\newcommand*{\bblank}{{=}}
\newcommand*{\Id}{\mathrm{Id}}
\newcommand*{\kVect}{\tetramod{}}
\newcommand*{\kvect}{\tetramodfd{}}
\newcommand*{\op}{\ensuremath{\mathrm{op}}}
\newcommand*{\adj}[4]{\ensuremath{#1 \colon #3 \rightleftarrows #4 \cocolon #2}}
\newcommand*{\stdadj}{\adj{F}{U}{\cat{C}}{\cat{D}}}
\newcommand*{\ld}[1]{\leftidx{^\vee}{\!#1}{}}     
\newcommand*{\rd}[1]{{#1}^\vee}                   
\newcommand*{\ev}{\mathrm{ev}}
\newcommand*{\coev}{\mathrm{coev}}
\newcommand*{\Cocts}[1]{\ensuremath{#1\text{\rm -}\mathbf{Cocts}}}
\newcommand*{\nt}{\Longrightarrow}

\renewcommand*{\to}{\longrightarrow}
\renewcommand*{\mapsto}{\longmapsto}
\newcommand*{\cocolon}{%
  \nobreak\mskip6mu plus1mu%
  \mathpunct{}\nonscript\mkern-\thinmuskip{:\!}%
  \mskip2mu\relax%
}
\renewcommand*{\hom}[1]{\ensuremath{\lfloor#1\rfloor}}
\newcommand*{\cohom}[1]{\ensuremath{\lceil#1\rceil}}
\usepackage{scalerel}   
\makeatletter
\let\@tl\triangleleft
\let\@tr\triangleright
\newcommand*{\@smallTriangle}[2]{\vcenter{\hbox{\scalebox{0.75}{\ensuremath{#1#2}}}}}
\newcommand*{\@medTriangle}[2]{\vcenter{\hbox{\scalebox{0.90}{\ensuremath{#1#2}}}}}
\newcommand*{\lact}{\mathbin{\mathpalette\@smallTriangle\@tr}}
\newcommand*{\ract}{\mathbin{\mathpalette\@smallTriangle\@tl}}
\let\@btl\blacktriangleleft
\let\@btr\blacktriangleright
\newcommand*{\blact}{\mathbin{\mathpalette\@medTriangle\@btr}}
\newcommand*{\bract}{\mathbin{\mathpalette\@medTriangle\@btl}}
\renewcommand*{\triangleright}{\lact}
\renewcommand*{\triangleleft}{\ract}
\renewcommand*{\blacktriangleright}{\blact}
\renewcommand*{\blacktriangleleft}{\bract}
\newcommand*{\cotens}{\mathrel{\raisebox{0.07em}{\ensuremath{\Box}}}}
\renewcommand*{\square}{\cotens}
\makeatother
\newcommand{\biact}[2]{#1\text{\rm-}\mathbf{Biact}\text{\rm-}#2}
\newcommand{\biactC}[3][\cat{C}]{#2\text{\rm-}\mathbf{Biact}_{#1}\text{\rm-}#3}

\DeclareRobustCommand{\SkipTocEntry}[5]{}

\numberwithin{equation}{section}

\numberwithin{figure}{section}

\usepackage{thmtools,mathtools}
\usepackage{graphicx,enumitem,stmaryrd}
\usepackage{multicol}
\usepackage{xparse}
\usepackage{leftidx}    

\usepackage{anyfontsize}
\allowdisplaybreaks[4]
\usepackage[all]{xy}
\usepackage[table]{xcolor}
\usepackage[normalem]{ulem}
\usepackage{bm, tensor}
\usepackage{tikz,tikz-cd}
\usetikzlibrary{matrix}
\usetikzlibrary{fit}
\usetikzlibrary{calc}
\tikzcdset{
  arrow style=tikz,
  diagrams={>={Straight Barb}}
}

\definecolor{blue}{rgb}{0.38, 0.51, 0.71}
\definecolor{red}{RGB}{175, 49, 39}
\definecolor{green}{RGB}{146, 227, 95}
\usepackage[colorlinks, linktoc=page, linkcolor={red}, citecolor={green!65!black}, urlcolor={blue}]{hyperref}

\newcommand{\xtwoheadrightarrow}[2][]{%
  \xrightarrow[#1]{#2}\mathrel{\mkern-14mu}\rightarrow%
}

\makeatletter
\def\slashedarrowfill@#1#2#3#4#5{%
$\m@th\thickmuskip0mu\medmuskip\thickmuskip\thinmuskip\thickmuskip
  \relax#5#1\mkern-7mu%
  \cleaders\hbox{$#5\mkern-2mu#2\mkern-2mu$}\hfill
  \mathclap{#3}\mathclap{#2}%
  \cleaders\hbox{$#5\mkern-2mu#2\mkern-2mu$}\hfill
  \mkern-7mu#4$%
}
\def\rightslashedarrowfill@{%
\slashedarrowfill@\relbar\relbar\mapstochar\rightarrow}
\newcommand\xslashedrightarrow[2][]{%
\ext@arrow 0055{\rightslashedarrowfill@}{#1}{#2}}
\makeatother

\makeatletter
\newcommand{\ostar}{\mathbin{\mathpalette\make@circled\star}}
\newcommand{\make@circled}[2]{%
\ooalign{$\m@th#1\smallbigcirc{#1}$\cr\hidewidth$\m@th#1#2$\hidewidth\cr}%
}
\newcommand{\smallbigcirc}[1]{%
\vcenter{\hbox{\scalebox{0.77778}{$\m@th#1\bigcirc$}}}%
}
\makeatother

\usepackage{adjustbox}

\usepackage{CJKutf8}
\newcommand*{\yo}{\text{\begin{CJK}{UTF8}{min}よ\end{CJK}}}

\tikzcdset{scale cd/.style={every label/.append style={scale=#1},
  cells={nodes={scale=#1}}}}

\usetikzlibrary{calc}
\usetikzlibrary{decorations.pathmorphing}
\tikzset{curve/.style={settings={#1},to path={(\tikztostart)
  .. controls ($(\tikztostart)!\pv{pos}!(\tikztotarget)!\pv{height}!270:(\tikztotarget)$)
  and ($(\tikztostart)!1-\pv{pos}!(\tikztotarget)!\pv{height}!270:(\tikztotarget)$)
  .. (\tikztotarget)\tikztonodes}},
  settings/.code={\tikzset{quiver/.cd,#1}
      \def\pv##1{\pgfkeysvalueof{/tikz/quiver/##1}}},
  quiver/.cd,pos/.initial=0.35,height/.initial=0}

\tikzset{tail reversed/.code={\pgfsetarrowsstart{tikzcd to}}}
\tikzset{2tail/.code={\pgfsetarrowsstart{Implies[reversed]}}}
\tikzset{2tail reversed/.code={\pgfsetarrowsstart{Implies}}}
\tikzset{no body/.style={/tikz/dash pattern=on 0 off 1mm}}

\definecolor{blue(pigment)}{rgb}{0.2, 0.2, 0.6}
\definecolor{americanrose}{rgb}{1.0, 0.01, 0.24}
\definecolor{nicegreen}{rgb}{0.0, 0.5, 0.0}
\definecolor{deepmagenta}{rgb}{0.8, 0.0, 0.8}
\definecolor{deepcarrotorange}{rgb}{0.91, 0.41, 0.17}
\definecolor{cadetgrey}{rgb}{0.57, 0.64, 0.69}

\newtheorem{theoremm}{Theorem}[section]
\newtheorem{theoremmm}{Theorem}
\newtheorem{theoremmmm}{Theorem}

\declaretheorem[style=plain,name=Theorem,numberlike=theoremmm]{theoremat}
\declaretheorem[style=plain,name=Theorem,numberlike=theoremmmm]{theorematic}

\declaretheorem[style=plain,name=Theorem,numberlike=theoremm]{theorem}
\declaretheorem[style=plain,name=Theorem,numbered=no]{theorem*}

\declaretheorem[style=plain,name=Lemma,numberlike=theoremm]{lemma}
\declaretheorem[style=plain,name=Proposition,numberlike=theoremm]{proposition}
\declaretheorem[style=plain,name=Corollary,numberlike=theoremm]{corollary}

\declaretheorem[style=definition,name=Definition,numberlike=theorem]{definition}

\declaretheorem[style=remark,name=Example,numberlike=theorem]{example}
\declaretheorem[style=remark,name=Remark,numberlike=theorem]{remark}

\declaretheorem[style=plain,name=Porism,numberlike=theorem]{porism}

\usepackage[cal=boondoxo, calscaled=0.96, bb=dsserif]{mathalfa}

\newcommand{\on}[1]{\operatorname{#1}}
\newcommand{\setj}[1]{\left\{ #1 \right\}}
\newcommand{\hcomp}{\circ_{h}}

\newcommand*{\xiso}{%
\overset{\smash{\raisebox{-0.65ex}{\ensuremath{\scriptstyle\sim}}}%
                \,}%
        {\to}%
}
\newcommand*{\xIso}{%
\overset{\smash{\raisebox{-0.65ex}{\ensuremath{\scriptstyle\sim}}}%
                \,}%
        {\nt}%
}

\DeclareFontFamily{U}{DSSerif}{\skewchar \font =45}
\DeclareFontShape{U}{DSSerif}{m}{n}{<-> s*[1]  DSSerif}{}
\DeclareFontSubstitution{U}{DSSerif}{m}{n}
\DeclareMathAlphabet{\mathbbbb}{U}{DSSerif}{m}{n}

\newcommand{\kotimes}{\otimes_{\Bbbk}}

\newcommand{\leftmod}[2]{(#1, #2)\text{\rm-Mod}}
\newcommand{\modma}{\leftmod{A}{\mathcal{M}}}
\newcommand{\leftcomod}[2]{(#1, #2)\text{\rm-Comod}}
\newcommand{\comodmc}{\leftcomod{C}{\mathcal{M}}}
\newcommand{\freeleftmod}[2]{(#1, #2)\text{\rm-Free}}
\newcommand{\freeleftcomod}[2]{{(#1, #2)}\text{\rm-Cofree}}
\newcommand{\forgetleftmod}[2]{{(#1, #2)}\text{\rm-Forget}}
\newcommand{\forgetleftcomod}[2]{{(#1, #2)}\text{\rm-Coforget}}
\newcommand{\fincolimit}[1]{\mathbf{Fin}_{\mathbf{co}}(#1)}

\usetikzlibrary{backgrounds}
\usetikzlibrary{arrows}
\usetikzlibrary{shapes,shapes.geometric,shapes.misc}
\tikzstyle{tikzfig}=[baseline=-0.25em,scale=0.5]
\pgfkeys{/tikz/tikzit fill/.initial=0}
\pgfkeys{/tikz/tikzit draw/.initial=0}
\pgfkeys{/tikz/tikzit shape/.initial=0}
\pgfkeys{/tikz/tikzit category/.initial=0}
\pgfdeclarelayer{edgelayer}
\pgfdeclarelayer{nodelayer}
\pgfsetlayers{background,edgelayer,nodelayer,main}
\tikzstyle{none}=[inner sep=0mm]
\tikzstyle{every loop}=[]
\tikzstyle{whitedot}=[fill=white, draw, shape=circle, scale=0.3, tikzit draw=black, tikzit shape=circle, tikzit fill=white]
\tikzstyle{blackdot}=[fill=black, draw, shape=circle, scale=0.3, tikzit draw=black, tikzit shape=circle, tikzit fill=black]
\tikzstyle{box}=[fill=white, draw=black, shape=rectangle, tikzit fill=white]
\tikzstyle{BL}=[draw=black, shape=circle, fill=black, scale=0.3]
\tikzstyle{PP}=[draw={rgb,255:red,102;green,41;blue,163}, shape=circle, fill={rgb,255:red,102;green,41;blue,163}, scale=0.3]
\tikzstyle{morphism-edge}=[-, draw=black, thick]
\tikzstyle{cotensor}=[-, draw=gray]
\tikzstyle{braid-over}=[-, draw=white, thick, double=black, double distance=0.8pt, tikzit draw={rgb,255: red,128; green,0; blue,128}]
\tikzstyle{purple-over}=[-, draw=white, thick, double={rgb,255:red,102;green,41;blue,163}, double distance=0.8pt, tikzit draw={rgb,255:red,102;green,41;blue,163}]
\tikzstyle{purple}=[-, draw={rgb,255:red,102;green,41;blue,163}, thick]
\tikzstyle{blue-under}=[-, draw={rgb,255:red,0;green,128;blue,128}, thick]
\tikzstyle{ddd}=[-, draw=black, dash dot dot]
\tikzstyle{unit}=[-, draw=black, densely dotted]
\tikzstyle{Front}=[-, draw=black, fill={rgb,255; red,255; green,255; blue,255}, opacity=0.8]
\tikzstyle{Hidden}=[-, draw=black, fill={rgb,255; red,255; green,255; blue,255}, opacity=0.2]
\tikzstyle{directed}=[-, thick, black, decoration={markings, mark=at position 0.5 with {\arrow{>}}}, postaction=decorate]

\usepackage[a4paper, margin=3cm]{geometry}
\usetikzlibrary{external}
\makeatletter
\usepackage{environ}
\def\temp{&} \catcode`&=\active \let&=\temp
\newcommand{\mytikzcdcontext}[2]{
  \begin{tikzpicture}[baseline=(maintikzcdnode.base)]
      \node (maintikzcdnode) [inner sep=0, outer sep=0] {\begin{tikzcd}[#2]
            #1
        \end{tikzcd}};
  \end{tikzpicture}}
\NewEnviron{mytikzcd}[1][]{%
  \def\myargs{#1}%
  \edef\mydiagram{
    \noexpand\mytikzcdcontext{\expandonce\BODY}{\expandonce\myargs}
  }%
  \mydiagram%
}
\makeatother

\begin{document}

\title{Reconstruction of module categories in the infinite and non-rigid settings}

\author{Mateusz Stroiński}
\address{M.S., Department of Mathematics, Uppsala University, Box. 480, SE-75106, Uppsala, Sweden}
\email{mateusz.stroinski@math.uu.se}

\author{Tony Zorman}
\address{T.Z., TU Dresden, Institut für Geometrie, Zellescher Weg 12--14, 01062 Dresden, Germany}
\email{tony.zorman@tu-dresden.de}

\begin{abstract}
  By building on the notions of internal projective and injective objects in a module category introduced by Douglas, Schommer-Pries, and Snyder, we extend the reconstruction theory for module categories of Etingof and Ostrik.
  More explicitly, instead of algebra objects in finite tensor categories,
  we consider quasi-finite coalgebra objects in locally finite tensor categories.
  Moreover, we show that module categories over non-rigid monoidal categories can be reconstructed via lax module monads, which generalize algebra objects.
  For the monoidal category of finite-dimensional comodules over a (non-Hopf) bialgebra, we give this result a more concrete form, realizing module categories as categories of contramodules over Hopf trimodule algebras---this specializes to our tensor-categorical results in the Hopf case.
  In this context, we also give a precise Morita theorem, as well as an analogue of the Eilenberg--Watts theorem for lax module monads and, as a consequence, for Hopf trimodule algebras.
  Using lax module functors we give a categorical proof of the variant of the fundamental theorem of Hopf modules which applies to Hopf trimodules. We also give a characterization of fusion operators for a Hopf monad as coherence cells for a module functor structure, using which we similarly reinterpret and reprove the Hopf-monadic fundamental theorem of Hopf modules due to Bruguières, Lack, and Virelizier.
\end{abstract}

\maketitle

\vspace{-0.5cm}
\tableofcontents

\setlength{\parskip}{0.25em}
\section{Introduction}

Monoidal categories are abundant in various areas of mathematics.
Module categories---categories with an action of a monoidal category---also feature in a wide variety of settings.
Viewing monoidal and module categories as natural categorifications of algebras and modules, this is analogous to one of the motivating observations in classical representation theory: every algebra admits many interesting representations, and every explicit realization of an algebra gives rise to a representation, and possibly meaningful realization, of other representations.

The areas in which the language and theory of module categories has been applied include
the theory of subfactors~\cite{We,EP,EY},
field theories~\cite{FRS},
categorification in representation theory~\cite{BSW,MMMTZ},
symplectic and algebraic geometry~\cite{BZFN,BZBJ,Pas},
and, particularly prominently, many aspects of the theory of Hopf algebras and tensor categories.
The latter includes classical results,
such as those describing the relation between the Drinfeld center, Hopf modules, and Yetter--Drinfeld modules---%
see~\cite[Chapter~7]{EGNO} for a textbook account---%
as well as emerging theories, such as that of partial modules~\cite{Alvez15:partial,batista22:como-part}
and of Serre functors for comodule algebras~\cite{fuchs22:spher-morit-serre,shimizu23:relat-serre}.
For a survey on the role module categories in applied category theory and computer science,
better known as \emph{actegories} in these contexts,
see~\cite{CG}.

In the category-theoretic setting, a deep connection between $\mathcal{C}$-module categories, for a monoidal category $\mathcal{C}$, and enrichments in $\mathcal{C}$ was described in~\cite{BCSW,GoPo}.
Closely related, and fundamentally important in representation theory, is the process of recovering an algebra object from a module category.
More precisely, given an algebra object $A \in \mathcal{C}$, the category $\modc{A}$ is naturally a left $\mathcal{C}$-module category. Conversely, when considering a given module category $\mathcal{M}$, it is often very useful to find and study an algebra object $A$ in $\mathcal{C}$ such that there is an equivalence $\mathcal{M} \simeq \modc{A}$ of\, $\mathcal{C}$-module categories.

An early reconstruction result of this kind is~\cite[Theorem~1]{Os},
and many generalizations and variants have appeared in the literature since;
for instance~\cite[Theorem~7.10.1]{EGNO},~\cite[Theorem~2.24]{DSPS},~\cite[Theorem~4.7]{MMMT}, and~\cite[Theorem~4.6]{BZBJ}.
The following theorem is a slight reformulation of~\cite[Theorem~7.10.1]{EGNO}, which exemplifies results of this kind.

\begin{theoremat}[{\cite[Theorem~1]{Os},~\cite[Theorem~7.10.1]{EGNO}}]\label{finiterecon}
  Let $\mathcal{C}$ be a finite tensor category and let $\mathcal{M}$ be a finite abelian $\mathcal{C}$-module category such that, for any $M \in \mathcal{M}$, the evaluation functor $\blank \triangleright M\from \mathcal{C} \to \mathcal{M}$ is exact.
  Then there exists an algebra object $A \in \mathcal{C}$ such that there is an equivalence of\, \(\cat{C}\)-module categories $\modc{A} \simeq \mathcal{M}$.
\end{theoremat}

Even this very general result makes many assumptions on the monoidal category $\mathcal{C}$, the module category $\mathcal{M}$, and on the action functor $\blank \triangleright \bblank$.
Three different kinds of assumptions can be identified in Theorem~\ref{finiterecon}, and similarly in many other results in the literature:
\begin{enumerate}
  \item Finiteness assumptions on the categories $\mathcal{C}$ and \(\cat{M}\)---in the case of Theorem~\ref{finiterecon}, both categories are finite abelian, and thus have enough projectives and injectives, finite-dimensional $\on{Hom}$-spaces, and finitely many isomorphism classes of simple objects.
  \item Exactness assumptions on the monoidal structure $\blank \otimes \bblank$ of\, $\mathcal{C}$ and on the action functor $\blank \triangleright \bblank\from \mathcal{C} \kotimes \mathcal{M} \to \mathcal{M}$. In Theorem~\ref{finiterecon}, these are exact in both variables.
  \item Rigidity assumptions on $\mathcal{C}$; i.e., that all of its objects admit left and right duals with respect to its monoidal structure.
\end{enumerate}

The aim of the first part of this article is to generalize these reconstruction results in a way that greatly relaxes the first two kinds of assumptions, and, crucially, removes the third one altogether.

In~\cite[Example~2.20]{DSPS} it is shown that, in the absence of rigidity, it may be impossible to reconstruct a given module category using algebra objects.
As such, we will have to use an appropriate generalization thereof.

In order to illustrate the role of the above assumptions, we briefly sketch two different proofs of Theorem~\ref{finiterecon}.

\begin{enumerate}
  \item For $M \in \mathcal{M}$, the exactness of\, $\blank \triangleright M$, together with the finiteness  of\, $\mathcal{C}$ and $\mathcal{M}$, ensure that the presheaf $\on{Hom}_{\mathcal{M}}(\blank \triangleright M,M)\from \mathcal{C}^{\on{op}} \to \tetramod$ is representable.
  It is easy to show that this functor is lax monoidal, so that the representing object, denoted by $\hom{M,M}$, is an algebra object in $\mathcal{C}$.
  \item Similarly, for any $N \in \mathcal{M}$, the presheaf $\on{Hom}_{\mathcal{M}}(\blank \triangleright M,N)$ is a module under the functor $\on{Hom}_{\mathcal{M}}(\blank \triangleright M,M)$ in the sense of~\cite[Definition~39]{Ye};
  it is represented by an $\hom{M,M}$-module object $\hom{M,N}$.
  This defines a left exact functor $\hom{M,\blank}\from \mathcal{M} \to \modc{\hom{M,M}}$.
  \item Finiteness of $\mathcal{M}$ and rigidity of\, $\mathcal{C}$ imply the existence of $X \in \mathcal{M}$ such that $\hom{X,\blank}$ is also right exact and faithful.
  \item Finiteness of $\mathcal{M}$ and rigidity of\, $\mathcal{C}$ imply that for such $X$, the functor $\hom{X,\blank}$ is an equivalence of\, $\mathcal{C}$-module categories.
\end{enumerate}

However, as we show in Example~\ref{coreps}, if\, $\mathcal{C}$ is not rigid,
then even in the presence of an object $X$ such that a faithful and exact $\hom{X,\blank}$ exists, we may have $\mathcal{M} \not\simeq \modc{\hom{X,X}}$.
A proof strategy that relies more on categorical tools, such as Beck's monadicity theorem, gives a clearer understanding of this failure.
The following sketch mirrors the proof of a generalization of Theorem~\ref{finiterecon} given in~\cite{BZBJ}:

\begin{enumerate}
  \item For $M \in \mathcal{M}$, the exactness of\, $\blank \triangleright M$, together with finiteness of\, $\mathcal{C}$ and $\mathcal{M}$, ensure that $\blank \triangleright M$ has a right adjoint, which we denote by $\hom{M,\blank}$.
  \item The rigidity of\, $\mathcal{C}$ implies that $\hom{M,\blank}$ is a $\mathcal{C}$-module functor, and so the composite monad $\hom{M,\blank \triangleright M}$ is a $\mathcal{C}$-module endofunctor of\, $\mathcal{C}$.
  \item The Yoneda lemma provides a monoidal equivalence $\mathcal{C} \simeq \lMod{\mathcal{C}}(\mathcal{C,C})$, under which $\hom{M,\blank \triangleright M}$ corresponds to an object $\hom{M,M} \in \mathcal{C}$, the monad structure on $\hom{M,\blank \triangleright M}$ corresponds to an algebra object structure on $\hom{M,M}$, and the Eilenberg--Moore category $\mathbf{EM}(\hom{M,\blank \triangleright M})$ is canonically equivalent to $\modc{\hom{M,M}}$ as a $\mathcal{C}$-module category.
  \item The finiteness of $\mathcal{M}$ and $\mathcal{C}$ imply the existence of $X \in \mathcal{M}$ such that $\hom{X,\blank}$ is exact and conservative, and thus, by Beck's monadicity theorem, the comparison functor $\modc{\hom{X,X}} \simeq \mathbf{EM}(\hom{X,\blank \triangleright X}) \leftarrow \mathcal{M}$ is an equivalence.
\end{enumerate}

From this sketch it is clear that the finiteness and exactness assumptions are used to establish the monadicity of $\mathcal{M}$ over $\mathcal{C}$, and the rigidity of\, $\mathcal{C}$ is used to identify the obtained monad with one given by an algebra object in $\mathcal{C}$.

Given an object $M \in \mathcal{M}$ such that the adjoint $\hom{M,\blank}$ exists, in Section~\ref{sec:internal-projectives-and-injectives} we determine sufficient---and in many cases necessary---conditions for $\hom{X,\blank}$ to be exact and conservative in the presence of projective objects.
This allows us to sharpen the finiteness and exactness conditions above to the following, significantly more general form.

\begin{theorematic}[Theorems~\ref{thm:main-result-rigid-projective-case} and~\ref{thm:main-result-rigid-injective-case}]\label{rigged}
  Assume that \(\mathcal{C}\) is rigid, that \(\mathcal{C}\) and \(\cat{M}\) have enough projectives (injectives), and that there is an object \(X \in \mathcal{M}\) such that:
  \begin{enumerate}
    \item there is a right adjoint \(\hom{X,\blank}\) (left adjoint \(\cohom{X,\blank}\)) to \(\blank \triangleright X\);
    \item for \(P \in \mathcal{C}\) projective (injective), the object \(P \triangleright X\) is projective (injective);
    \item any projective (injective) object \(N\) of \(\mathcal{M}\) is a direct summand of an object of the form \(P \triangleright X\), for \(P\) projective (injective).
  \end{enumerate}
  Then \(\mathcal{M}\) is, as a \(\cat{C}\)-module category, equivalent to \(\modc{\hom{X,X}}\) \((\comodc{\cohom{X,X}})\).
\end{theorematic}

In the case of tensor categories, this statement may on the surface seem inapplicable, since a tensor category which is not finite will generally not have projectives or injectives.
However, since the category underlying a tensor category is equivalent to the category of finite-dimensional comodules for a coalgebra, the $\mathbf{Ind}$-completion of a tensor category has enough injectives.
Further, since multitensor categories can be realized as categories of compact objects in locally finitely presentable categories, we may use an appropriate variant of the special adjoint functor theorem, so that we may characterize when the internal cohom $\cohom{M,\blank}$ exists for an object \(M \in \mathbf{Ind}(\mathcal{M})\).
We obtain the following result.

\begin{theorematic}\label{multitensorintroduction}
  Let $\mathcal{C}$ be a multitensor category, and let $\mathcal{M}$ be a locally finite abelian $\mathcal{C}$-module category, such that there is an object $X \in \mathbf{Ind}(\mathcal{M})$ satisfying the following properties:
  \begin{enumerate}
    \item $\blank \triangleright X$ is left exact;
    \item for an injective $I \in \mathbf{Ind}(\mathcal{C})$, the object $I \triangleright X$ is injective and any injective in $\mathbf{Ind}(\mathcal{M})$ is a direct summand of an object of this form;
    \item the space $\on{Hom}_{\mathbf{Ind}(\mathcal{M})}(M, I \triangleright X)$ is finite-dimensional for any $M \in \mathcal{M}$ and finitely cogenerated injective $I \in \mathbf{Ind}(\mathcal{M})$.
  \end{enumerate}

  Then $\cohom{X,\blank}$ exists and $\mathbf{Ind}(\mathcal{M}) \simeq \comodc[\mathbf{Ind}(\mathcal{C})]{\cohom{X,X}}$ as $\mathcal{C}$-module categories. This restricts to a $\mathcal{C}$-module equivalence $\mathcal{M} \simeq \comodc{X,X}$.
\end{theorematic}

Consider the special case where a tensor category $\mathcal{C}$ admits a fiber functor to the category $\tetramodfd$ of finite-dimensional vector spaces, and is thus monoidally equivalent to the category $\tetramodfd[H]$ of finite-dimensional comodules over a finite-dimensional Hopf algebra $H$.
Theorem~\ref{finiterecon} realizes a finite $\mathcal{C}$-module category $\mathcal{M}$ as the category of finite-dimensional modules over a finite-dimensional $H$-comodule algebra.
Comodule algebras arise naturally in many settings, and so Theorem~\ref{finiterecon} captures an important class of module categories.

Since $H$ and $A$ are finite-dimensional, the category of modules for an $H$-comodule algebra $A$ is equivalent, as an $\tetramodfd[H]$-module category, to the category of comodules for the $H$-comodule coalgebra $A^{\ast}$. In this sense, the following Hopf-theoretic corollary of Theorem~\ref{multitensorintroduction} is an immediate generalization of Theorem~\ref{finiterecon}.

\begin{theorematic}\label{hopfuming}
  Let $H$ be a Hopf algebra, and let $\mathcal{M}$ be a $\tetramodfd[H]$-module category satisfying the assumptions of Theorem~\ref{finiterecon} for $\mathcal{C} = \tetramodfd[H]$.
  Then there is an $H$-comodule coalgebra $C$, which is possibly infinite-dimensional, such that $\mathcal{M} \simeq \fdcomodc[H]{C}$.
\end{theorematic}

As previously stated,~\cite[Example~2.20]{DSPS} shows that (co)algebra objects in $\mathcal{C}$ do not suffice for an exhaustive reconstruction process in the non-rigid case.
This is in fact also the case in some settings where rigidity is assumed, and one instead considers (co)algebra objects in a suitable (co)completion, as is the case in e.g.~\cite[Theorem~4.7]{MMMT}, or Theorem~\ref{multitensorintroduction} above.
In the absence of rigidity, we invoke the bicategorical Yoneda lemma to realize the objects of\, $\mathcal{C}$ as the $\mathcal{C}$-module endofunctors of\, $\mathcal{C}$ itself, and view the \emph{lax} $\mathcal{C}$-module endofunctors of\, $\mathcal{C}$ as the correct generalization of objects in $\mathcal{C}$, and lax $\mathcal{C}$-module monads as the correct generalization of algebra objects in $\mathcal{C}$.
Indeed, by doctrinal adjunction (see~\cite{Ke1} and~\cite{HZ2} for the particular case of module categories), the right adjoint of a (strong) $\mathcal{C}$-module functor is canonically lax, and thus the (co)monads we study are canonically (op)lax $\mathcal{C}$-module functors.

An additional difficulty we encounter is that while the Kleisli category of a lax $\mathcal{C}$-module monad is naturally a $\mathcal{C}$-module category, the same is not true for the Eilenberg--Moore category.
However, we show that in general there is (up to equivalence) at most one $\mathcal{C}$-module structure on the Eilenberg--Moore category \emph{extending} that on the Kleisli category, and in Theorem~\ref{extendingalways} we show its existence under the assumption of right exactness (left exactness for comonads) of the (co)monad, using so-called Linton coequalizers, and multicategorical techniques similar to those in~\cite{aguiar18:monad}.

More precisely, we show that the structure defined by Linton coequalizers on the Eilenberg--Moore category can be described as a multiactegory, in the sense of~\cite{arkor24}, and, in Proposition~\ref{hillbility}, we prove that this multiactegory is \emph{representable}, in a similar sense to representable multicategories considered in~\cite{hermida00:repres,leinster04:higher}. The module category structure is then well-defined, since we also show that, analogously to multicategories and monoidal categories, representable multiactegories yield module categories.

The following is our most general reconstruction result.

\begin{theorematic}[{Theorems~\ref{mainnonsenseprojective},~\ref{mainnonsenseinjective}, and~\ref{projcorrespondence}}]\label{jennydeathwhen}
  Assume that $\mathcal{C}$ and \(\cat{M}\) have enough projectives (injectives) and that there is an object $X \in \mathcal{M}$ such that:
  \begin{enumerate}
    \item there is a right adjoint $\hom{X,\blank}$ (left adjoint $\cohom{X,\blank}$) to $\blank \triangleright X$;
    \item for $P \in \mathcal{C}$ projective (injective), the object $P \triangleright X$ is projective (injective);
    \item any projective (injective) object $N$ of $\mathcal{M}$ is a direct summand of an object of the form $P \triangleright X$, for $P$ projective (injective).
  \end{enumerate}
  Then $\mathcal{M} \simeq \mathbf{EM}(\hom{X,\blank \triangleright X})$, and the $\mathcal{C}$-module structure of the category of \(T\)-modules is extended from the Kleisli category. Furthermore, this extends to a bijection
   \begin{equation}\label{Level4OnWednesday}
 \begin{aligned}
  \setj{\,(\mathcal{M},X) \text{ as above}\,}/(\mathcal{M} \simeq \mathcal{N}) &\xleftrightarrow{1:1} \setj{\,\text{Lax }\mathcal{C}\text{-module monads on } \mathcal{C}\,}/(\mathbf{EM}(T) \simeq \mathbf{EM}(S)). \\
  (\mathcal{M},X) &\longmapsto\, \hom{X,-\triangleright X} \\
  (\mathbf{EM}(T), T(\mathbb{1})) &\longmapsfrom\, T
 \end{aligned}
   \end{equation}
\end{theorematic}

The formulation of the bijection~\eqref{Level4OnWednesday} indicates a Morita aspect to the reconstruction theory for module categories, as the equivalence relation imposed on monads strongly resembles, and can be specialized to, Morita equivalence of algebras. In Theorem~\ref{TurboEilenbergWatts}, we characterize the right exact lax $\mathcal{C}$-module functors between Eilenberg--Moore categories for lax $\mathcal{C}$-module monads in terms of suitable bimodule objects in the category of endofunctors of $\mathcal{C}$ - this gives a precise meaning to the notion of Morita equivalence of monads, hence the bijection~\eqref{Level4OnWednesday} can be rephrased as
\[
 \setj{\,(\mathcal{M},X) \text{ as above}\,}/(\mathcal{M} \simeq \mathcal{N}) \xleftrightarrow{1:1} \setj{\,\text{Lax }\mathcal{C}\text{-module monads on } \mathcal{C}\,}/\simeq_{\text{Morita}},
\]
see Theorem~\ref{injectivebijection}.

While Theorem~\ref{jennydeathwhen} gives a general reconstruction result in the non-rigid setting,
which has recently seen increased interest in multiple areas
(see e.g.~\cite{DSPS,allen21:dualit-feigin-fuchs} or~\cite[Corollary~10.11]{stroinski24:modul-tambar}),
lax $\mathcal{C}$-module monads on $\mathcal{C}$ are less accessible than mere objects of\, $\mathcal{C}$, and may indeed appear not very accessible in general.
To show that this need not be the case, we again turn to the Hopf-algebraic setting.
Recall that a ring category (the non-rigid counterpart of a tensor category, see~\cite[Sections~4.2 and~5.4]{EGNO}) which admits a fiber functor to $\tetramodfd$ is monoidally equivalent to the category $\tetramodfd[B]$ of finite-dimensional left $B$-comodules over a (not necessarily Hopf) bialgebra $B$.

In Section~\ref{sec:Hopf-trimodules}, we show that the category of left exact, finitary lax $\tetramod[B][][][]$-module endofunctors of\, $\tetramod[B][][][]$ is monoidally equivalent to the category $\tetramod[B][B][][B]$ of \emph{Hopf trimodules}; i.e., $B$-$B$-bicomodules with an additional left $B$-action that is a $B$-$B$-bicomodule morphism.
Such structures feature prominently in the quasi-bialgebraic generalization of the fundamental theorem of Hopf modules, see~\cite{hausser99:integ-theor-quasi-hopf-algeb,saracco17:hopf}.
This equivalence matches a lax $\tetramod[B][][][]$-module monad, such as that featuring in Theorem~\ref{jennydeathwhen}, with a Hopf trimodule algebra:
a Hopf trimodule $A$ together with maps $A \cotens A \to A$ and $B \to A$ satisfying the usual associativity and unitality axioms for an algebra object.

If $B$ is infinite-dimensional, then Theorem~\ref{jennydeathwhen} should yield reconstruction in terms of an \emph{oplax} $\tetramod[B][][][]$-module comonad $\cohom{X,\blank \triangleright X}$, using the existence of injective objects in $\tetramod[B][][][]$.
Since our result gives an algebraic realization only for the \emph{lax} $\tetramod[B][][][]$-module functors, we restrict ourselves to the case where $\blank \triangleright X$ admits both a left and a right adjoint, by assuming $\blank \triangleright X$ to be exact.
In that case, we obtain an adjunction $\cohom{X,\blank \triangleright X} \dashv \hom{X,\blank \triangleright X}$.

However, since the monad is right, and not left, adjoint, the category $\mathbf{EM}(\cohom{X,\blank \triangleright X})$ is not equivalent to the category $\mathbf{EM}(\hom{X,\blank \triangleright X})$ of modules over the associated Hopf trimodule algebra $\hom{X,X}$, but rather to the category of its \emph{contramodules}---structures extensively studied in the setting of so-called semi-infinite homological algebra, see~\cite{Pos}.
On the other hand, while the Eilenberg--Moore categories are not equivalent, the Kleisli categories are, and it is again the Kleisli categories that control the $\tetramod[B][][][]$-module structure, which can thus be read off directly from the Hopf trimodule algebra.
We obtain the following result.

\begin{theorematic}\label{parkingtickets}
  Let $\mathcal{M}$ be a locally finite abelian $\tetramodfd[B][][][]$-module category satisfying the assumptions of Theorem~\ref{jennydeathwhen}.
  Then there is a Hopf trimodule algebra $A \in \tetramodfd[B][B][][B]$ such that there is an equivalence $\mathbf{Ind}(\mathcal{M}) \simeq A\text{\rm-Contramod}$ of\, $\tetramodfd[B]$-module categories, where the $\tetramodfd[B]$-module structure on the right-hand side is extended from the category of free $A$-module.
  This equivalence restricts to $\mathcal{M} \simeq A\text{\rm-contramod}_{\on{f.d.}}$.
\end{theorematic}

Since this algebraic realization of our reconstruction results may at first glance seem difficult to apply in calculations, we give two explicit examples in which we determine a trimodule algebra for a given module category. The first is Example~\ref{coreps}, where $B$ is the semigroup algebra for the unique two-element monoid which is not a group. This example is intended to be very similar to~\cite[Example~2.20]{DSPS}, in that a simple object of\, $\mathcal{C}$ acts as zero on an indecomposable object.

We also show that applying the ordinary, ``rigid'' reconstruction procedure on this example yields the same algebra object in $\tetramodfd[B][][][]$ as that corresponding to the regular action of $\tetramodfd[B][][][]$ on itself, and thus fails to classify module categories.
For our second example, we let $B$ be arbitrary and determine the Hopf trimodule algebra underlying the module functor given by the fiber functor $\tetramodfd[B] \to \tetramodfd$.

Finally, the Morita-theoretic results of Theorem~\ref{injectivebijection} give us a precise notion of Morita equivalence for Hopf trimodule algebras with respect to their contramodules in Definition~\ref{contramoritadef}, resulting in a Morita Theorem for $\tetramodfd[B][][][]$-module categories, Theorem~\ref{contramodulereuters}.

As a more theoretic application, we use the trimodule reconstruction to give a categorical interpretation of a variant of the Hausser--Nill theorem.

\begin{theorematic}\label{illmatic}
  The functor $\tetramod[B] \to \tetramod[B][B][][B]$ corresponds to the inclusion of strong $\tetramod[B]$-module endofunctors in the category of lax $\tetramod[B]$-module endofunctors, under the Yoneda lemma and Theorem~\ref{thm:hopf-trimodules-are-lax-monoidal-functors}.
  The latter---and hence also the former---functor is an equivalence if and only if\, $\tetramodfd[B]$ is left rigid, which is the case if and only if $B$ has a twisted antipode.
\end{theorematic}

Finally, in Proposition~\ref{diffusion}, we show that the fusion operators of a Hopf monad $T$ on a monoidal category $\mathcal{W}$ can be interpreted as coherence cells for the canonical oplax $\mathbf{EM}(T)$-module structure on $T$. This gives a strong converse to the fact that (op)lax module functors over a rigid category are automatically strong, by providing a distinguished module functor which is strong if and only if the category is rigid. Theorem~\ref{illmatic} can also be viewed as such a converse. This result, interpreting additional structure of a monad on a monoidal category as morphisms for an (op)lax module functor structure on said monad, is similar to the results of~\cite[Theorems~3.17 and~3.18]{flake24:froben-drinf}, which investigates Frobenius monoidal functors rather than Hopf monads.

The article is structured as follows.
In Section~\ref{callingallcars}, we establish notation and collect the necessary preliminaries.
In Section~\ref{sec:module-monads} we describe the canonical module category structures on the Kleisli and Eilenberg--Moore categories associated to (op)lax module (co)monads,
the module category structures on Eilenberg Moore categories defined by Linton coequalizers for (left) right exact (co)monads,
the relations between Kleisli respectively Eilenberg--Moore categories for adjunctions between lax module monads and oplax module comonads, and finally we give a brief account of Beck's monadicity theorem.
In Section~\ref{sec:internal-projectives-and-injectives}, we describe sufficient conditions for exactness and conservativity of internal Homs, coHoms and internal projectives and injectives.
In Section~\ref{sec:lax-module-reconstruction}, we state and show our main results regarding reconstruction in terms of (op)lax module (co)monads: Theorems~\ref{rigged},~\ref{multitensorintroduction},~\ref{hopfuming}, and~\ref{jennydeathwhen}.
Section~\ref{Takyon} is devoted to the proof of the generalization of the Eilenberg--Watts theorem to the setting of lax module monads, as well as its Morita-theoretic implications for the results of Section~\ref{sec:lax-module-reconstruction}.
In Section~\ref{sec:Hopf-trimodules}, we establish the equivalence between lax module endofunctors of\, $\tetramodfd[B]$ and Hopf trimodules, state and prove the Hopf trimodule reconstruction of Theorem~\ref{parkingtickets}, and show a simple application of this result in Example~\ref{coreps}.
In Section~\ref{fiberfunctorthings}, we determine the Hopf trimodule algebra underlying the fiber functor.
In Section~\ref{sec:Hausser-Nill}, we state our categorical variant of the Hausser--Nill theorem; i.e., Theorem~\ref{illmatic}.

\addtocontents{toc}{\SkipTocEntry}
\subsection*{Acknowledgements}\label{sec:acknowledgements}

The authors would like to thank Ulrich Krähmer, Volodymyr Mazorchuk, and Victor Ostrik for many useful comments and discussions, as well as Marcelo Aguiar for explaining the use of Linton coequalizers and multicategories in the study of monoidal structures on Eilenberg--Moore categories for lax monoidal monads in~\cite{aguiar18:monad}.
T.Z.~is supported by \textsc{dfg} grant \textsc{kr} \oldstylenums{5036/2--1}.

\section{Preliminaries}\label{callingallcars}

\subsection{Monoidal and module categories}

For a more extensive account of the notions of monoidal categories and strong monoidal functors (therein referred to simply as \emph{monoidal functors}), we refer the reader to~\cite[Chapter~2]{EGNO}.

\begin{definition}\label{def:lax-monoidal-monoidal-functor}
  A \emph{lax monoidal functor} between two monoidal categories \((\cat{C}, \otimes_{\cat{C}}, \mathbb{1}_{\cat{C}})\) and \((\cat{D}, \otimes_{\cat{D}}, \mathbb{1}_{\cat{D}})\)
  comprises a functor \(F \from \cat{C} \to \cat{D}\),
  together with a morphism $F_{\mathsf{u}} \from \mathbb{1}_{\cat{D}} \to F(\mathbb{1}_{\cat{C}})$ and a natural transformation
  \[
    {(F_{\mathsf{m}})}_{V,W} \from F(V) \otimes_{\cat{D}} F(W) \to F(V \otimes_{\cat{C}} W), \qquad \text{for all \(V, W \in \cat{C}\)},
  \]
  satisfying \emph{associativity} and \emph{unitality} conditions specified in~\cite[Definition~2.4.1]{EGNO}.

  Analogously, an \emph{oplax monoidal functor} is a functor $G\from \cat{C} \to \cat{D}$ together with morphisms $G_{\mathsf{u}}\from G(\mathbb{1}_{\cat{C}}) \to \mathbb{1}_{\cat{D}}$ and
  \[
    {(G_{\mathsf{m}})}_{V,W} \from G(V \otimes_{\cat{C}} W) \to G(V) \otimes_{\cat{D}} G(W), \qquad \text{for all \(V, W \in \cat{C}\)},
  \]
  making \(G^{\op}\from \cat{C}^{\op} \to \cat{D}^{\op}\) lax monoidal.

  We call \(F\) \emph{strong monoidal}
  if \(F_{\mathsf{m}}\) and \(F_{\mathsf{u}}\) and invertible,
  and \emph{strict monoidal} if they are identities.
\end{definition}

Observe that an oplax monoidal functor with invertible coherence morphisms canonically defines a strong monoidal functor.

\begin{definition}\label{def:monoidal-nat-trafo}
  Let \(F, F' \from (\cat{C}, \otimes, \mathbb{1}) \to (\cat{D}, \otimes, \mathbb{1})\) be monoidal functors.
  A \emph{monoidal natural transformation} is a natural transformation \(\varphi\from F \nt F'\) from \(F\) to \(F'\),
  such that for all \(V, W \in \cat{C}\)
  \[
    {(F_{\mathsf{m}})}_{V, W} \circ \varphi_{V \otimes W} = (\varphi_{V} \otimes \varphi_{W}) \circ {(F'_{\mathsf{m}})}_{V, W},
  \]
  and such that $\varphi_{\mathbb{1}}$ is an isomorphism.
\end{definition}

We denote the category of strong monoidal functors from \(\mathcal{C}\) to \(\mathcal{D}\) by \(\mathbf{StrMonCat}(\mathcal{C,D})\),
the category of lax monoidal functors from \(\mathcal{C}\) to \(\mathcal{D}\) by \(\mathbf{LaxMonCat}(\mathcal{C,D})\), and
the category of oplax monoidal functors from \(\mathcal{C}\) to \(\mathcal{D}\) by \(\mathbf{OplaxMonCat}(\mathcal{C,D})\).
More generally, composing lax, oplax, or strong monoidal functors again yields a monoidal functor of the same kind,
so we obtain \(2\)-categories \(\mathbf{StrMonCat}\), \(\mathbf{LaxMonCat}\), and \(\mathbf{OplaxMonCat}\).

\begin{definition}\label{def:right-module-category}
  Let \(\cat{C}\) be a monoidal category.
  A \emph{left \(\cat{C}\)-module category} comprises a category \(\cat{M}\)
  together with a functor \(\lact \from \cat{C} \times \cat{M} \to \cat{M}\)
  and isomorphisms
  \[
    {(\cat{M}_{\mathsf{a}})}_{V, W, M} \from (V \otimes W) \lact M \xIso V \lact (W \lact M),
  \]
  natural in \(M \in \cat{M}\) and \(V,W \in \cat{C}\),
  satisfying \emph{coherence} axioms similar to those of a monoidal category;
  see~\cite[Definition~7.1.1]{EGNO}.
\end{definition}

In an analogous way to Definition~\ref{def:right-module-category},
one can define \emph{right} module categories over \(\cat{C}\),
which involve a functor \(\ract \from \cat{M} \times \cat{C} \to \cat{M}\) and isomorphisms
\(M \ract (V \otimes W) \xIso (V \ract W) \ract M\).

\begin{definition}
  Let \(\cat{C}\) and \(\cat{D}\) be monoidal categories.
  A \emph{\((\cat{C}, \cat{D})\)-bimodule category} comprises a left \(\cat{C}\)-module category \(\cat{M}\) that
  is simultaneously a right \(\cat{D}\)-module category
  and which, for all \(V, W \in \cat{C}\) and \(M \in \cat{M}\), comes equipped with a natural isomorphism
  \[
    (V \lact M) \ract W \xIso V \lact (M \ract W),
  \]
  called the \emph{middle interchange},
  and satisfying \emph{associativity} constraints;
  see~\cite[Definition~7.1.7]{EGNO}.
\end{definition}

\begin{example}\label{ex:copower}
  Let \(\Bbbk\) be a field.
  If \(\cat{A}\) is an additive \(\Bbbk\)-linear category,
  then the category \(\kvect\) of finite-dimensional \(\Bbbk\)-vector spaces
  acts on \(\mathcal{A}\) via \emph{copowers} (also known as \emph{tensoring}): given a finite-dimensional vector space $V$ over $\Bbbk$ and an object $A \in \mathcal{A}$, the object $V \triangleright A$ is defined by the universal property
  \[
    \mathcal{A}(V \triangleright A, A') \simeq \kvect(V, \mathcal{A}(A,A')).
  \]
  In particular, $V \triangleright A \cong A^{\oplus \dim{V}}$.

  A particularly well-behaved example of this is the category of right \(A\)-modules for a \(\Bbbk\)-algebra \(A\): here, given a module $M \in \rmod{A}$, we have $V \triangleright M = V \kotimes M$, with the right $A$-module structure inherited from $M$.
\end{example}

\begin{definition}\label{def:lax-module-functor}
  Let \(\cat{M}\) and \(\cat{N}\) be left module categories over a monoidal category \(\cat{C}\).
  A \emph{lax \(\cat{C}\)-module functor} from \(\cat{M}\) to \(\cat{N}\) is a functor
  \(F \from \cat{M} \to \cat{N}\), together with a collection of morphisms
  \[
    {(F_{\mathsf{a}})}_{V,M} \from V \lact F(M) \to F(V \lact M), \qquad \text{natural for all \(V \in \mathcal{C}\) and \(M \in \mathcal{M}\)},
  \]
  satisfying \emph{associativity} and \emph{unitality} conditions; see~\cite[Definition~7.2.1]{EGNO}.

  \emph{Oplax} and \emph{strong} \(\cat{C}\)-module functors are defined similarly;
  in the former case,
  one considers arrows \({(F_{\mathsf{a}})}_{V, M} \from F(V \lact M) \to V \lact F(M)\),
  while for the latter \(F_{\mathsf{a}}\) should be invertible.
\end{definition}

\begin{definition}\label{def:module-transformation}
  Let \(F, G \from \cat{M} \to \cat{N}\) be \(\cat{C}\)-module functors between the left \(\cat{C}\)-module categories \(\cat{M}\) and \(\cat{N}\).
  A natural transformation \(\phi \from F \nt G\) is called a
  \emph{\(\cat{C}\)-module transformation}
  if
  \[
    {(G_{\mathsf{a}})}_{V, M} \circ (\phi_V \lact M) = \phi_{V \lact M} \circ {(F_{\mathsf{a}})}_{V, M}, \qquad\qquad
    \text{for all \(V \in \cat{C}\) and \(M \in \cat{M}\).}
  \]
\end{definition}

For $\mathcal{C}$-module categories $\mathcal{M,N}$, we obtain the category $\cat{C}\text{-}\mathbf{Mod}(\cat{M},\cat{N}) \defeq \mathbf{Str}\mathcal{C}\text{-}\mathbf{Mod}(\mathcal{M,N})$ of strong $\mathcal{C}$-module functors from $\mathcal{M}$ to $\mathcal{N}$, the category $\mathbf{Lax}\mathcal{C}\text{-}\mathbf{Mod}(\mathcal{M,N})$ of lax $\mathcal{C}$-module functors from $\mathcal{M}$ to $\mathcal{N}$, and the category $\mathbf{Oplax}\mathcal{C}\text{-}\mathbf{Mod}(\mathcal{M,N})$ of oplax $\mathcal{C}$-module functors from $\mathcal{M}$ to $\mathcal{N}$.

\begin{example}\label{ex:regularend}
  For \(V \in \mathcal{C}\), the functor \(F \defeq \blank \otimes V\) is a left \(\mathcal{C}\)-module functor.
  The transformation \(F_{\mathsf{a}}\) is given by the associator of\, \(\cat{C}\).
  Further, a morphism \(f\from V \to W\) in \(\cat{C}\) yields a \(\mathcal{C}\)-module transformation \(\blank \otimes f\from \blank \otimes V \nt \blank \otimes W\).
\end{example}

In fact, Example~\ref{ex:regularend} generalizes and completely describes module functors from the regular module.
We emphasize that this result is a consequence of the \emph{bicategorical Yoneda lemma}.

\begin{proposition}\label{prop:bicategorical-yoneda-correspondence-monads-and-algebras}
  Let \(\cat{M}\) be a left module category over the monoidal category \(\cat{C}\).
  Then there is an equivalence of module categories
  \[
    \lMod{\cat{C}}(\mathcal{C}, \mathcal{M}) \xiso \mathcal{M}, \qquad \qquad
    \Phi \mapsto \Phi(\mathbb{1}), \qquad \qquad
    \blank \lact M \longmapsfrom M.
  \]
\end{proposition}

In particular, Proposition~\ref{prop:bicategorical-yoneda-correspondence-monads-and-algebras}
yields a monoidal equivalence \(\lMod{\cat{C}}(\mathcal{C},\mathcal{C}) \simeq \mathcal{C}^{\otimes\on{rev}}\).

\begin{definition}
  We say that an object $M$ of a $\mathcal{C}$-module category $\mathcal{M}$ is \emph{closed} if the functor $\blank \triangleright M\from \mathcal{C} \to \mathcal{M}$ has a right adjoint,
  \(\hom{M, \blank}\from \cat{M}\to \cat{C}\), called the \emph{internal hom from \(M\)}.

  Similarly, we say that $M$ is \emph{coclosed} if the functor $\blank \triangleright M\from \mathcal{C} \to \mathcal{M}$ has a left adjoint;
  we denote this \emph{internal cohom from \(M\)} by \(\cohom{M, \blank}\from \cat{M}\to \cat{C}\).

  If every object of $\mathcal{M}$ is closed, we say that $\mathcal{M}$ is a \emph{closed $\mathcal{C}$-module category}.
\end{definition}

\begin{definition}\label{def:left-closed-monoidal-category}
  A monoidal category \(\cat{C}\) is called \emph{left closed} if the left regular $\mathcal{C}$-module category ${}_{\mathcal{C}}\mathcal{C}$ is closed,
  \emph{right closed} if the right regular $\mathcal{C}$-module category $\mathcal{C}_{\mathcal{C}}$ is closed,
  and \emph{closed} if it is both left and right closed.
\end{definition}

For example, a symmetric monoidal category is left closed if and only if it is right closed if and only if it is closed.

\begin{definition}\label{def:left-rigid-category}
  A \emph{dual pair} in a monoidal category \(\cat{C}\)
  is a tuple
  \[
    (V, W,\, \ev\from V \otimes W \to \mathbb{1},\, \coev \from \mathbb{1} \to W \otimes V),
  \]
  such that the \emph{snake equations} are satisfied:
  \[
    \id_W = (W \otimes \ev) \circ (\coev \otimes W), \qquad
    \id_{V} = (\ev \otimes V) \circ (V \otimes \coev).
  \]

  The object \(W\) is called the \emph{right dual} of \(V\);
  similarly, \(V\) is called the \emph{left dual} of \(W\).

  We say that \(\cat{C}\) is \emph{left rigid} if all of its objects have a left dual,
  and similarly for \emph{right rigidity}.
\end{definition}

It is customary to write \(\ld{V}\) for the left dual of \(W\),
and \(\rd{W}\) for the right dual of \(V\).
Taking left duals in a left rigid category assembles to a fully faithful functor \(\ld{(\blank)} \from \cat{C}^{\op, \mathrm{rev}} \to \cat{C}\),
and similarly we obtain \(\rd{(\blank)} \from \cat{C}^{\op, \mathrm{rev}} \to \cat{C}\).

\begin{remark}\label{rmk:rigid-is-closed}
  If\, \(\cat{C}\) is a left rigid category,
  then it is automatically left closed;
  the left internal hom is given by \({[\blank, \bblank]}_{\ell} \;\defeq\; \ld{\blank} \otimes \bblank\).
  Similar statements hold for right rigid and right closed categories.
\end{remark}

Importantly,
given a monoidal category \(\cat{C}\),
the mere presence of adjunctions
\[
  \blank \otimes V \,\dashv\, \blank \otimes LV \qquad\qquad\text{and}\qquad\qquad
  V \otimes \blank \,\dashv\, RV \otimes \blank
\]
for objects \(LV, RV \in \cat{C}\) does not lead to rigidity,
but to the weaker notion of \emph{tensor representability};
see~\cite{HZ1}.

\begin{example}\label{laxmoduleexample}
  The \(\Bbbk\)-linear symmetric monoidal category \(\kVect\) of vector spaces is closed,
  with internal hom given by \(\mathrm{Hom}_{\Bbbk}(\blank, \bblank)\).
  The functor \(\mathrm{Hom}_{\Bbbk}(V, \blank)\from \kVect \to \kVect\) is a lax \(\kVect\)-module functor;
  its module functor structure is given by
  \begin{align*}
    W \otimes \mathrm{Hom}_{\Bbbk}(V, U) &\to \mathrm{Hom}_{\Bbbk}(V, W \otimes U) \\
    w \otimes f &\mapsto (v \mapsto w \otimes f(v)).
  \end{align*}
  This is an isomorphism if and only if \(V\) is finite-dimensional, if and only if \(V\) admits a dual.
\end{example}

As illustrated by Proposition~\ref{prop:lax-is-strong-when-rigid} below, non-rigidity of \(\kVect\) is crucial in providing examples of non-strong module functors.
We remark that~\cite[Remark~4]{Os} and~\cite[Lemma~2.10]{DSPS} state Proposition~\ref{prop:lax-is-strong-when-rigid} only in the rigid (hence, both left and right rigid at the same time) case;
however, the proof of~\cite[Lemma~2.10]{DSPS} also shows the more specific statement we give.

\begin{proposition}[{\cite[Remark~4]{Os},~\cite[Lemma~2.10]{DSPS}}]\label{prop:lax-is-strong-when-rigid}
  Let \(\mathcal{C}\) be a left rigid monoidal category. Then every lax \(\mathcal{C}\)-module functor is strong.

  Dually, if\, $\mathcal{C}$ is a right rigid monoidal category, then every oplax $\mathcal{C}$-module functor is strong.
\end{proposition}

\subsection{Doctrinal adjunctions}

\begin{definition}\label{def:lift-to-monoidal-adjunction}
  Let \(\adj{F}{G}{\cat{C}}{\cat{D}}\) be an ordinary adjunction between monoidal categories.
  A \emph{lift of \(F \dashv G\)} to a lax monoidal adjunction comprises choices of lax monoidal structures on $F$ and $G$,
  together with a unit of adjunction $\eta$ and a counit of adjunction $\varepsilon$ which are both lax monoidal natural transformations.
  In other words, it is the structure necessary for a lax monoidal adjunction, whose left adjoint is $F$ and right adjoint is $G$.

  Similarly one can define a lift of \(F \dashv G\) to an oplax monoidal adjunction.
\end{definition}

The following statement is a special case of~\cite[Theorem~1.2]{Ke1}.

\begin{proposition}\label{prop:doctrinaladjunction}
  Given monoidal categories \(\cat{C}\) and \(\cat{D}\),
  as well as an adjunction \(\adj{F}{G}{\cat{C}}{\cat{D}}\),
  there is a bijection between oplax monoidal structures on \(F\)
  and lax monoidal structures on \(G\).
\end{proposition}

The proof of~\cite[Theorem~1.2]{Ke1} is given constructively in terms of \emph{mates}.
In our situation, Proposition~\ref{prop:doctrinaladjunction} says that,
given an oplax monoidal structure \((F_{\mathsf{m}}, F_{\mathsf{u}})\) on \(F\),
the multiplicative part of the corresponding lax monoidal structure on \(G\) is,
for all \(V, W \in \cat{C}\),
given by
\[
  V \otimes G(W)
  \xrightarrow{\eta_{V \otimes G(W)}} GF\big( V \otimes G(W) \big)
  \xrightarrow{{(F_{\mathsf{m}})}_{V, G(W)}} G\big(V \otimes FG(W)\big)
  \xrightarrow{G(V \otimes \varepsilon_{W})} G(V \otimes W),
\]
where \(\eta\) and \(\varepsilon\) are the unit and counit of the adjunction.
Constructing an oplax monoidal structure on \(F\) from a lax monoidal structure on \(G\) is similar.

The following proposition shows that this correspondence extends to module categories.

\begin{proposition}[{\cite[Theorem~4.13]{HZ2}}]\label{prop:special-doctrinal-module-adjunctions}
  Let \(\cat{M}\) and \(\cat{N}\) be right \(\cat{C}\)- and \(\cat{D}\)-module categories, respectively.
  Suppose that \(\adj{F}{G}{\cat{C}}{\cat{D}}\) is a monoidal adjunction,
  and \(\adj{L}{R}{\cat{M}}{\cat{N}}\) is any adjunction.
  There is a bijection between
  lifts of\, \(L \dashv R\) to a comodule adjunction
  and
  strong \(\cat{C}\)-module functor structures on \(R\).
\end{proposition}

In fact, from the proof of~\cite[Theorem~4.13]{HZ2},
one obtains a version of Proposition~\ref{prop:doctrinaladjunction}.

\begin{porism}\label{porism:doctrinal-module-adjunctions}
  Let \(\cat{M}\) and \(\cat{N}\) be right \(\cat{C}\)- and \(\cat{D}\)-module categories, respectively.
  Suppose that \(\adj{F}{G}{\cat{C}}{\cat{D}}\) is a monoidal adjunction,
  and \(\adj{L}{R}{\cat{M}}{\cat{N}}\) is any adjunction.
  Then there is a bijection between oplax \(\cat{C}\)-module structures on \(L\) over \(F\),
  and lax \(\cat{C}\)-module structure on \(R\) over \(G\).
\end{porism}

\subsection{Finite and locally finite abelian categories}

Let \(\Bbbk\) be a field.
In this section, we assume all categories and functors to be \(\Bbbk\)-linear.
We call a functor \(F\from \cat{A} \to \cat{B}\) between abelian categories
\emph{left exact} if it preserves finite limits,
\emph{right exact} if it preserves finite colimits,
and \emph{exact} if it is right and left exact.
The category of right exact functors from \(\cat{A}\) to \(\cat{B}\) will be denoted by \(\mathbf{Rex}(\cat{A}, \cat{B})\),
and we write \(\mathbf{Lex}(\cat{A}, \cat{B})\) for the category of left exact functors.
Further, let \(\proj{\cat{A}}\) denote the full subcategory of \(\cat{A}\) consisting of its projective objects,
and similarly for \(\inj{\cat{A}}\).

\begin{proposition}[{\cite[Lemma~3.3]{GKKP}}]\label{prop:right exact-is-enough-on-projectives}
  Let \(\mathcal{A}\) and \(\mathcal{B}\) be abelian categories,
  and assume that \(\mathcal{B}\) has enough projectives.
  Let \(\iota\from \proj{\mathcal{A}} \hookrightarrow \mathcal{A}\) be the inclusion functor
  from the category of projective objects in \(\mathcal{A}\) to \(\mathcal{A}\).
  Then the restriction functor\, \(\mathbf{Rex}(\mathcal{A}, \mathcal{B}) \xrightarrow{\blank \circ \iota} [\proj{\mathcal{A}}, \mathcal{B}]\) is an equivalence.
\end{proposition}

\begin{proposition}[{\cite[Lemma~3.3]{GKKP}}]\label{prop:left exact-is-enough-on-injectives}
  Let \(\mathcal{A}\) and \(\mathcal{B}\) be abelian categories,
  and assume that \(\mathcal{B}\) has enough injectives.
  Let \(\iota\from \inj{\mathcal{A}} \hookrightarrow \mathcal{A}\) be the inclusion functor from the category of injective objects in \(\mathcal{A}\) to \(\mathcal{A}\).
  Then the restriction functor \(\mathbf{Lex}(\mathcal{A},\mathcal{B}) \xrightarrow{\blank \circ \iota} [\inj{\mathcal{A}}, \mathcal{B}]\) is an equivalence.
\end{proposition}

We refer to~\cite[Chapter~1]{EGNO} for the following definitions and results.
An object of a $\Bbbk$-linear category is said to be \emph{simple} if any non-zero endomorphism thereof is an isomorphism.
A category $\mathcal{C}$ is said to be \emph{semisimple} if any of its objects is a direct sum of simple objects.

\begin{proposition}[{\cite[Theorem~C.6]{Turaev2017}}]\label{toomuchTV}
  A semisimple category is abelian.
  Any epimorphism or monomorphism in a semisimple category splits.
  In particular, any functor from a semisimple category to an abelian category is exact.
\end{proposition}

Let \(\cat{A}\) be an abelian category.
An object \(V \in \cat{A}\) has \emph{finite length} if there exists a filtration
\[
  0 = V_0 \subseteq V_1 \subseteq \dots \subseteq V_n = V,
\]
such that \(V_i/V_{i-1}\) is simple, for all \(i \in \{\, 1, \dots, n \,\}\).

A category \(\cat{C}\) is said to be \emph{finite abelian} if
it is abelian;
\emph{hom-finite}, i.e.~\(\cat{C}(x, y)\) is finite-dimensional for all \(x, y \in \cat{C}\);
has enough projectives;
has only finitely many isomorphism classes of simple objects; and
all of its objects are of finite length.

\begin{lemma}
  A category \(\mathcal{A}\) is finite abelian
  if and only if
  there is a finite-dimensional \(\Bbbk\)-algebra \(A\) such that \(\mathcal{A} \simeq \lmod{A}\).
\end{lemma}

\begin{definition}
  A category is said to be \emph{locally finite abelian} if
  it is abelian,
  hom-finite,
  and all of its objects are of finite length.
\end{definition}

\begin{lemma}\label{lem:campus-bar}
  A category \(\mathcal{C}\) is locally finite abelian
  if and only if
  there is a \(\Bbbk\)-coalgebra \(C\) such that \(\mathcal{C} \simeq \tetramodfd[C]{}\).
\end{lemma}

\subsection{(Co)completions}

We refer to~\cite{KS} for generalities on (co)limits and (co)completions.
In this subsection we give a brief account of the results regarding the monoidal (pseudo)functoriality of cocompletions and the resulting (co)completion operations for monoidal and module categories.
We again implicitly assume all categories and functors to be \(\Bbbk\)-linear.

First, let us introduce some terminology.
Let $\Phi$ be a class of diagrams.
We say that a category is \emph{$\Phi$-cocomplete} if it admits colimits of functors with domain in $\Phi$, and we say that a functor is \emph{$\Phi$-cocontinuous} if it preserves such colimits.

\begin{definition}\label{smoothcats}
  A monoidal category $\mathcal{C}$ is called \emph{separately $\Phi$-cocontinuous} if\, $\mathcal{C}$ is $\Phi$-cocomplete and $\blank \otimes_{\cat{C}} \bblank$ is separately $\Phi$-cocontinuous.

  Similarly, for a $\Phi$-cocomplete monoidal category $\cat{D}$, a $\cat{D}$-module category $\mathcal{M}$ is said to be separately $\Phi$-cocontinuous if $\mathcal{M}$ is $\Phi$-cocomplete and $\blank \triangleright_{\mathcal{M}} \bblank$ is separately $\Phi$-cocontinuous.
\end{definition}

Let $\mathbf{Cat}_{\Cocts{\Phi}}$ denote the $2$-category of $\Phi$-cocomplete categories, $\Phi$-cocontinuous functors, and transformations between them.
Let $\mathcal{C}$ be a small category.
By~\cite[Section~5.7]{Ke2}, there is a category $\Phi(\mathcal{C})$, the \emph{$\Phi$-cocompletion} of $\mathcal{C}$, endowed with an embedding $\iota\from \mathcal{C} \hookrightarrow \Phi(\mathcal{C})$, such that, for any $\Phi$-cocomplete category $\mathcal{D}$ the restriction functor $\mathbf{Cat}_{\Cocts{\Phi}}(\mathcal{C},\mathcal{D}) \xrightarrow{\blank \circ \iota} \mathbf{Cat}(\mathcal{C},\mathcal{D})$ is an equivalence.
Furthermore, by the main results of~\cite{Ko,Zo}, the inclusion $2$-functor $\mathbf{Cat}_{\Cocts{\Phi}} \to \mathbf{Cat}$ is $2$-monadic, with a left adjoint given by the pseudofunctor of $\Phi$-cocompletion, $\Phi(\blank)\from \mathbf{Cat} \to \mathbf{Cat}_{\Cocts{\Phi}}$.
This $2$-monad is a lax-idempotent $2$-monad in the sense of~\cite{KeLa1}.
By~\cite{KeLa2}, there is a similar $2$-monad $\Phi_{c}$ on $\mathbf{Cat}$, whose algebras are categories with a \emph{prescribed choice} of $\Phi$-colimits, strict morphisms are functors strictly preserving the chosen $\Phi$-colimits, and pseudomorphisms are functors preserving $\Phi$-colimits in the ordinary sense.
By~\cite[Theorem~6.2]{LF2}, this yields a pseudo-closed structure the $2$-category $\mathbf{Cat}_{\Cocts{\Phi_{c}}}$ of categories with a choice of $\Phi$-colimits, $\Phi$-cocontinuous functors, and transformations between them, where the category underlying the internal hom from $\mathcal{A}$ to $\mathcal{B}$ is given by $\mathbf{Cat}_{\Cocts{\Phi}}(\mathcal{A,B})$, and the functors comprising the $2$-monad are closed.
Further, in some cases, for instance if $\Phi$ is the class of finite colimits, the pseudo-closed structure becomes pseudo-closed monoidal.
Since monoidal categories and module categories can be formulated internally to a pseudo-closed (monoidal) structure as  pseudomonoids, we find the following.

\begin{proposition}\label{cocompletingnonsense}
  Let $\mathcal{C}$ be a monoidal category.
  Then $\Phi(\mathcal{C})$ is separately $\Phi$-cocontinuous monoidal, the inclusion $\iota\from \mathcal{C} \hookrightarrow \Phi(\mathcal{C})$ is strong monoidal and
  induces an equivalence
  \[
    \mathbf{StrMon}_{\Cocts{\Phi}}(\Phi(\mathcal{C}),\mathcal{D}) \xrightarrow[\sim]{\blank \circ \iota} \mathbf{StrMon}(\mathcal{C,D})
  \]
  for any separately $\Phi$-cocontinuous monoidal category $\mathcal{D}$.
  Similar equivalences are induced for lax and oplax monoidal functors.

  Likewise, given a $\mathcal{C}$-module category $\mathcal{M}$, the $\Phi$-cocompletion $\Phi(\mathcal{M})$ is a separately $\Phi$-cocontinuous $\Phi(\mathcal{C})$-module category, the inclusion $\iota\from \mathcal{M} \hookrightarrow \Phi(\mathcal{M})$ is a strong $\mathcal{C}$-module functor, and the functor
  \begin{equation}\label{cocompletingmodules}
    \mathbf{Str}\Phi(\mathcal{C})\text{-}\mathbf{Mod}_{\Cocts{\Phi}}(\Phi(\mathcal{M}),\mathcal{N})
    \xrightarrow[\sim]{\blank \circ \iota}
    \mathbf{Str}\mathcal{C}\text{-}\mathbf{Mod}(\mathcal{M,N}),
  \end{equation}
  is an equivalence;
  similar equivalences exist for lax and oplax module functors.
\end{proposition}

We remark that explicit constructions of the monoidal and module structures obtained in Proposition~\ref{cocompletingnonsense} and direct proofs of their universal properties are proven in many cases and many ways in the literature, see for instance~\cite[Section~3]{MMMT} and~\cite[Example~3.2.9]{CG}.

\begin{proposition}
  Let $\mathcal{A}$ be an additive category.
  The action of\, \(\kvect\) on \(\mathcal{A}\) via copowering, as described in Example~\ref{ex:copower} is the essentially unique \(\tetramodfd\)-module structure on \(\mathcal{A}\).
\end{proposition}

\begin{proof}
  Let $\mathsf{add}$ denote the collection of finite discrete categories.
  An $\mathsf{add}$-cocomplete category is simply an additive category, and an $\mathsf{add}$-cocontinuous functor is an additive functor.

  Since any $\Bbbk$-linear functor preserves direct sums, an additive $\kvect$-module category is necessarily separately additive in the sense of Definition~\ref{smoothcats}.
  The result then follows by observing that $\kvect \simeq \mathsf{add}(\setj{\oast})$, where $\setj{\oast}$ is the terminal monoidal category, so
  \[
    \mathbf{StrMon}_{\mathsf{add}}(\kvect, \mathbf{Cat}_{\mathsf{add}}(\mathcal{A,A})) \simeq \mathbf{StrMon}(\setj{\ast},\mathbf{Cat}_{\mathsf{add}}(\mathcal{A,A})) \simeq \setj{\ast},
  \]
  where $\setj{\ast}$ is the terminal category and $\mathbf{StrMon}_{\mathsf{add}}$ denotes the category of strong monoidal additive functors.
\end{proof}

\begin{proposition}\label{FinitaryAction}
  Let $\mathcal{A}$ be an additive category admitting filtered colimits.
  The action of $\kvect$ on $\mathcal{A}$ extends uniquely to a separately finitary $\kVect$-module category structure on $\mathcal{A}$.
\end{proposition}
\begin{proof}
  Let \(\setj{\mathsf{add},\mathsf{filt}}\) denote the collection of diagrams which are filtered or finite discrete.
  Then we have \(\kVect \simeq \setj{\mathsf{add},\mathsf{filt}}(\setj{\oast})\), and
  \[
    \begin{aligned}
      &\mathbf{StrMon}_{\mathsf{filt}}(\kVect, \mathbf{Cat}_{\on{add,filt}}(\mathcal{A,A})) = \mathbf{StrMon}_{\on{add,filt}}(\kVect, \mathbf{Cat}_{\on{add,filt}}(\mathcal{A,A}))\\
      &\simeq \mathbf{StrMon}(\setj{\oast},\mathbf{Cat}_{\on{add,filt}}(\mathcal{A,A})) = \mathbf{StrMon}(\setj{\oast},\mathbf{Cat}_{\mathsf{filt}}(\mathcal{A,A})) \simeq \setj{\ast}.
    \end{aligned}
  \]
\end{proof}

Filtered colimits and cocompletion under them, which we denote by \(\mathbf{Ind}(\blank)\),
are---in view of Lemma~\ref{lem:campus-bar}---particularly important in the study of locally finite abelian categories.
One reason for this is because \(\mathbf{Ind}(\tetramodfd[][][][C]) \simeq \tetramod[][][][C]\),
and categories of the latter form have desirable properties, in particular in terms of characterizing adjoint functors.
This is partly because categories of comodules are locally finitely presentable.

\begin{proposition}\label{locoinjectives}
  Let $\mathcal{C}$ be a locally finite abelian category. Then $\mathbf{Ind}(\mathcal{C})$ has enough injectives.
\end{proposition}

\begin{proposition}[Special adjoint functor theorem for right adjoints]\label{saftrex}
  If\, $\mathcal{C}$ and \(\cat{D}\) are locally finite abelian categories, then there are equivalences
  \[
    \mathbf{Rex}(\mathcal{C,D}) \xrightarrow[\sim]{\mathbf{Ind}(-)} \mathbf{Cocont}(\mathbf{Ind}(\mathcal{C}), \mathbf{Ind}(\mathcal{D})) \simeq \mathbf{Map}(\mathbf{Ind}(\mathcal{C}), \mathbf{Ind}(\mathcal{D})),
  \]
  where $\mathbf{Map}(\mathcal{A,B})$ denotes the category of functors admitting a right adjoint.
  In particular, a functor from \(\mathbf{Ind}(\mathcal{C})\) to \(\mathbf{Ind}(\mathcal{D})\) admits a right adjoint if and only if it is cocontinuous,
  and the extension of a functor \(F\from \mathcal{C} \to \mathcal{D}\) to \(\mathbf{Ind}\)-cocompletions is cocontinuous if and only if\, \(F\) is right exact.
\end{proposition}

\begin{proposition}[Special adjoint functor theorem for left adjoints]\label{saftlex}
  Let \(\mathcal{C}\) and \(\cat{D}\) be locally finite abelian categories.
  If a functor \(G\from \mathbf{Ind}(\mathcal{C}) \to \mathbf{Ind}(\mathcal{D})\) is continuous and preserves filtered colimits, then it admits a left adjoint.

  If a functor \(G\from \mathbf{Ind}(\mathcal{C}) \to \mathbf{Ind}(\mathcal{D})\) admits a left adjoint \(F\), then \(G\) preserves filtered colimits if and only if \(F\) preserves compact objects in \(\mathcal{C}\).
\end{proposition}

The latter part of Proposition~\ref{saftrex} follows from $\mathbf{Ind}(F)$ preserving both finite and filtered colimits, and thus preserving all colimits.
Since filtered colimits in a locally finitely presentable category are exact, the extension $\mathbf{Ind}(G)$ of a left exact functor $G$ is left exact.
However, it is not clear that $\mathbf{Ind}(G)$ preserves arbitrary products.

\begin{definition}
  Let $\mathcal{C}$ be a locally finite abelian category.
  An object $X \in \mathbf{Ind}(\mathcal{C})$ is \emph{finitely cogenerated} if the functor $\on{Hom}_{\mathbf{Ind}(\mathcal{C})}(\blank,X)$ preserves arbitrary products.
\end{definition}

One can verify that for a coalgebra $D$ such that $\mathbf{Ind}(\mathcal{C}) \simeq \rComod{D}$,
a comodule $M$ is finitely cogenerated if and only if $M$ embeds into a comodule of the form $D^{\oplus m}$ for some $m \in \mathbb{Z}_{\geq 0}$.
Further, observe that an object $M \in \mathbf{Ind}(\mathcal{D})$ is compact if and only if it lies in the essential image of the embedding $\mathcal{D} \hookrightarrow \mathbf{Ind}(\mathcal{D})$.

\begin{definition}\label{qfFunctor}
  Let $\mathcal{C,D}$ be locally finite abelian categories.
  A functor $F\from \mathbf{Ind}(\mathcal{C}) \to \mathbf{Ind}(\mathcal{D})$ is called \emph{quasi-finite} if
  for any finitely cogenerated injective object $I$ in $\mathbf{Ind}(\mathcal{C})$ and compact object $M$ of $\mathcal{D}$,
  the \(\Bbbk\)-vector space $\mathbf{Ind}(\mathcal{D})(M, F(I))$ is finite-dimensional.
\end{definition}

Using the equivalence $\mathcal{C} \simeq \tetramodfd[C]{}$ of Lemma~\ref{lem:campus-bar}, the following result is a direct consequence of~\cite[Proposition~1.3]{Tak}.

\begin{proposition}\label{locallyfiniteadjoints}
  Let \(G\from \mathcal{C} \to \mathcal{D}\) be a left exact functor of locally finite abelian categories.
  The functor $\mathbf{Ind}(G)$ has a left adjoint if and only if it is quasi-finite.
\end{proposition}

A stronger result holds for finite abelian categories, as a direct consequence of the Eilenberg--Watts theorem for categories of bimodules, see e.g.~\cite[Lemma~2.1]{FSS}.

\begin{proposition}\label{finiteadjoints}
  A functor \(F\from \mathcal{A} \to \mathcal{B}\) between finite abelian categories
  is right exact if and only if it has a right adjoint,
  and left exact if and only if it has a left adjoint.
\end{proposition}

Recall that an object \(A\) of a category \(\cat{A}\) that admits \(\Phi\)-colimits is said to be \emph{\(\Phi\)-small} if the functor \(\cat{A}(A, \blank)\) preserves them.
We need the following fact about these objects.

\begin{lemma}\label{leftadjointssmallobjects}
  Let $G\from \mathcal{C} \to \mathcal{D}$ be a $\Phi$-cocontinuous functor,
  and assume $G$ has a left adjoint $F$.
  Then $F$ sends $\Phi$-small objects to $\Phi$-small objects.
\end{lemma}
\begin{proof}
  Let $X \in \mathcal{D}$ be a $\Phi$-small object.
  Then $\mathcal{C}(F(X),\blank) \cong \mathcal{D}(X,G(\blank)) = \mathcal{D}(X,\blank) \circ G(\blank)$.
  The functor $\mathcal{C}(F(X),\blank)$ is thus naturally isomorphic to a composition of two $\Phi$-cocontinuous functors,
  and hence it itself is a $\Phi$-cocontinuous functor.
\end{proof}

\subsection{Algebra and module objects}\label{sec:monoids}

Let \(\mathcal{C}\) be a monoidal category.
An \emph{algebra object in \(\mathcal{C}\)} comprises
an object \(A \in \mathcal{C}\)
together with a \emph{multiplication} \(\mu\from A \otimes A \to A\)
and a \emph{unit} \(\eta\from \mathbb{1} \to A\),
satisfying associativity and unitality axioms analogous to those for a \(\Bbbk\)-algebra.
In fact, a \(\Bbbk\)-algebra is the same as an algebra object in \(\kVect\).

A \emph{coalgebra object} in \(\mathcal{C}\) is
an object \(C \in \mathcal{C}\)
with a comultiplication \(\Delta\from C \to C \otimes C\)
and a counit morphism \(\varepsilon\from C \to \mathbb{1}\),
satisfying coassociativity and counitality axioms,
such that \(C\) becomes an algebra object in \(\mathcal{C}^{\op}\).

\begin{definition}\label{ModuleInModule}
  Let \((\mathcal{M}, \lact)\) be a left \(\mathcal{C}\)-module category
  and suppose that \((A, \mu, \eta)\) is an algebra object in \(\cat{C}\).
  A \emph{left \(A\)-module in \(\mathcal{M}\)} is
  an object \(M \in \mathcal{M}\)
  together with an \emph{action} morphism \(\alpha\from A \lact M \to M\)
  such that the following diagrams commute
  \[
    \begin{mytikzcd}[ampersand replacement=\&]
      {(A \otimes A) \triangleright M} \& {A \triangleright (A \triangleright M)} \& {A \triangleright M} \& {\mathbb{1} \lact M} \& M \\
      {A \triangleright M} \&\& M \& {A \lact M}
      \arrow["{{{\mathcal{M}_{\mathrm{a}}}}}", from=1-1, to=1-2]
      \arrow["{{\mu \,\triangleright\, M}}"', from=1-1, to=2-1]
      \arrow["{{A \,\triangleright\, \alpha}}", from=1-2, to=1-3]
      \arrow["{{\alpha}}", from=1-3, to=2-3]
      \arrow["{{\cat{M}_{\mathrm{u}}}}", from=1-4, to=1-5]
      \arrow["{\eta \,\lact\, M}"', from=1-4, to=2-4]
      \arrow["{{\alpha}}"', from=2-1, to=2-3]
      \arrow["{\alpha}"', from=2-4, to=1-5]
    \end{mytikzcd}
  \]

  A \emph{morphism of modules} is an arrow \(f\from M \to N\) in \(\mathcal{M}\) commuting with the respective actions:
  \[
    \begin{mytikzcd}[ampersand replacement=\&]
      {A \triangleright M} \& M \\
      {A \triangleright N} \& N
      \arrow["{\alpha_M}", from=1-1, to=1-2]
      \arrow["{A \,\lact\, f}"', from=1-1, to=2-1]
      \arrow["f", from=1-2, to=2-2]
      \arrow["{\alpha_N}"', from=2-1, to=2-2]
    \end{mytikzcd}
  \]
\end{definition}

\begin{example}\label{MoreModules}
  Clearly, left \(A\)-modules in \(\mathcal{M}\) and their morphisms form a category,
  which we shall denote by \(\modma\).
  The obvious forgetful functor
  \(\forgetleftmod{A}{\mathcal{M}}\from \leftmod{A}{\mathcal{M}} \to \mathcal{M}\)
  admits a left adjoint
  \begin{align*}
    \freeleftmod{A}{\mathcal{M}}\from \mathcal{M} &\to \leftmod{A}{\mathcal{M}} \\[-0.25cm]
    M &\mapsto \big(A \triangleright M,\ A \lact (A \lact M) \xrightarrow{\mathcal{M}_{\mathrm{a}}^{-1}} (A \otimes A) \lact M \xrightarrow{\mu \,\triangleright\, M} A \lact M\big)
  \end{align*}
  The unit of this adjunction is given by \(M \xrightarrow{\mathcal{M}_{\mathrm{u}}} \mathbb{1} \triangleright M \xrightarrow{\eta \,\triangleright\, M} A \triangleright M\),
  and the counit is \(A \triangleright M \xrightarrow{\alpha_M} M\).
\end{example}

Given a coalgebra object \(C\) in \(\mathcal{C}\),
a \emph{left \(C\)-comodule in \(\mathcal{M}\)} is an object \(N \in \mathcal{C}\) satisfying axioms dual to those for a module in \(\mathcal{M}\),
so that a left \(C\)-comodule in \(\mathcal{M}\) is the same as a left \(C\)-module in the \(\mathcal{C}^{\op}\)-module category \(\mathcal{M}^{\op}\).
Comodule morphisms are defined similarly.
We denote the category of \(C\)-comodules in \(\mathcal{M}\) by \(\comodmc\).
There is an adjunction \(\forgetleftcomod{C}{\mathcal{M}} \dashv \freeleftcomod{C}{\mathcal{M}}\) analogous to that for modules.

Similarly, one can define right modules and comodules in a right module category,
and bimodules and bicomodules in a bimodule category.

\begin{example}\label{ex:lifting-functors-to-mod}
  Let \(A\) be an algebra object in \(\mathcal{C}\) and let \(\mathcal{M}\) and \(\cat{N}\) be left \(\mathcal{C}\)-module categories.
  A lax \(\mathcal{C}\)-module functor \(G\from \mathcal{N} \to \mathcal{M}\) induces a functor
  \[
    \leftmod{A}{G}\from \leftmod{A}{\mathcal{N}} \to \leftmod{A}{\mathcal{M}}
  \]
  by sending a left \(A\)-module \(N\) in \(\mathcal{N}\) to the left \(A\)-module \(G(N)\),
  with action \(\alpha_{G(N)}\) given by
  \[
    A \lact_{\mathcal{M}} G(N) \xrightarrow{G_{\mathrm{a};A,N}} G(A \lact_{\mathcal{N}} N) \xrightarrow{G \alpha_{N}} G(N).
  \]
  This functor satisfies the following relation:
  \[
    \begin{mytikzcd}[ampersand replacement=\&]
      {\leftmod{A}{\mathcal{N}}} \&\& {\leftmod{A}{\mathcal{M}}} \\
      {\mathcal{N}} \&\& {\mathcal{M}}
      \arrow["{\leftmod{A}{G}}", from=1-1, to=1-3]
      \arrow["{\forgetleftmod{A}{\mathcal{N}}}"', from=1-1, to=2-1]
      \arrow["{\forgetleftmod{A}{\mathcal{M}}}", from=1-3, to=2-3]
      \arrow["G"', from=2-1, to=2-3]
    \end{mytikzcd}
  \]

  Moreover, mate correspondence yields a canonical natural transformation
  \[
    \begin{mytikzcd}[ampersand replacement=\&]
      {\leftmod{A}{\mathcal{N}}} \&\& {\leftmod{A}{\mathcal{M}}} \\
      {\mathcal{N}} \&\& {\mathcal{M}}
      \arrow["{\leftmod{A}{G}}", from=1-1, to=1-3]
      \arrow["{\freeleftmod{A}{\mathcal{N}}}", from=2-1, to=1-1]
      \arrow["G"', from=2-1, to=2-3]
      \arrow[shorten <=37pt, shorten >=12pt, Rightarrow, from=2-3, to=1-1]
      \arrow["{\freeleftmod{A}{\mathcal{M}}}"', from=2-3, to=1-3]
    \end{mytikzcd}
  \]
  that is invertible if \(G\) is a strong \(\mathcal{C}\)-module functor.
\end{example}

Similarly to Example~\ref{ex:lifting-functors-to-mod},
given a coalgebra object \(C \in \mathcal{C}\),
an oplax module functor \(F\from \mathcal{C} \to \mathcal{D}\) induces a functor \(\leftcomod{C}{F}\from \leftcomod{C}{\mathcal{M}} \to \leftcomod{C}{\mathcal{N}}\)
that satisfies
\[
  \forgetleftcomod{C}{\mathcal{N}} \circ \leftmod{C}{F}
  \ =\ %
  F \circ \forgetleftcomod{C}{\mathcal{M}},
\]
and whose mate natural transformation
\(\leftcomod{C}{F} \circ \freeleftcomod{C}{M} \nt \freeleftcomod{C}{\mathcal{N}} \circ F\)
is invertible if \(F\) is a strong \(\mathcal{C}\)-module functor.

\begin{example}\label{ex:bicategory-of-bimodules}
  Let \(\cat{C}\) be a cocomplete symmetric monoidal category,
  such that its is right exact in both variables.
  Then one may define the bicategory \(\mathbf{Bimod}(\cat{C})\) of bimodules in \(\cat{C}\) as follows:
  \begin{itemize}
    \item Objects are algebra objects in \(\cat{C}\).
    \item For \(A\) and \(B\) in \(\cat{C}\), the hom-category \(\mathbf{Bimod}(\cat{C})(A, B)\)
    is given by \(B\)-\(A\)-bimodules in \(\cat{C}\) and their homomorphisms.
    Horizontal composition is given by the balanced tensor product of bimodules.
  \end{itemize}

  In fact, this category becomes a monoidal bicategory with respect to the underlying tensor product of\, \(\cat{C}\).
  Tensoring bimodules \({}_{B}M_{A}\) and \({}_{D}N_{C}\) is given by
  the \((B \otimes D)\)-\((A \otimes C)\)-bimodule
  \({}_{B}M_{A} \otimes {}_{D}N_{C}\),
  and this extends to bimodule homomorphisms in the obvious way.

  Analogously,
  there is a bicategory \(\mathbf{Bicomod}(\cat{C})\) of\, \(\cat{C}\)-bicomodules
  when \(\cat{C}\) is a complete symmetric monoidal category,
  and the tensor product is left exact in both variables.
\end{example}

\begin{example}\label{finitarythings}
  Let \(\mathcal{A}\) and \(\cat{B}\) be additive categories.
  Recall that the existence of finite biproducts endows \(\mathcal{A}\) and \(\cat{B}\) with the structure of \(\kvect\)-module categories.
  Any functor \(F\from \mathcal{A} \to \mathcal{B}\) is a \(\kvect\)-module functor with respect to these \(\kvect\)-actions.

  If \(\mathcal{B}\) admits filtered colimits,
  then \(F\) extends essentially uniquely to a finitary \(\kVect\)-module functor \(\mathbf{Ind}(\mathcal{A}) \to \mathcal{B}\),
  and any finitary \(\kVect\)-module functor is of this form.
\end{example}
\begin{proof}
  Let $\mathsf{add}(\mathcal{A})$ denote the cocompletion of $\mathcal{A}$ under finite biproducts---the so-called \emph{additive closure} of $\mathcal{A}$.
  Since finite biproducts are absolute colimits, the left adjoint $\mathsf{add}(\mathcal{A}) \to \mathcal{A}$ of the canonical inclusion $\mathcal{A} \hookrightarrow \mathsf{add}(\mathcal{A})$,
  coming from the universal property described in Equation~\eqref{cocompletingmodules},
  is an equivalence of $\kvect$-module categories.
  Hence, so is the inclusion $\mathcal{A} \hookrightarrow \mathsf{add}(\mathcal{A})$.
  We find the equivalences
  \[
    \mathbf{Cat}_{\Bbbk}(\mathcal{A},\blank) \simeq \kvect\text{\rm-Mod}(\mathsf{add}(\mathcal{A}),\blank) \simeq \lMod{\kvect}(\mathcal{A},\blank),
  \]
  establishing the first claim.
  The second claim follows similarly by using the equivalence
  \[
    \lMod{\kvect}(\mathcal{A},\mathcal{B}) \simeq \lMod{\mathbf{Ind}(\kvect)}_{\mathrm{filt}}(\mathbf{Ind}(\mathcal{A}), \mathcal{B}).
  \]
\end{proof}

Sometimes, the formalism of module categories, even those over the seemingly trivial monoidal categories $\kvect$ and $\kVect$, provides additional clarity to classical statements about algebraic and coalgebraic $\Bbbk$-linear structures. We give a brief proof of the result below as an example of this.

\begin{proposition}[{\cite[Proposition~2.1]{Tak}}]\label{TakeuchiEilenbergWatts}
  Let $\mathcal{D}$ be an abelian category admitting filtered colimits, and let $C$ be a coalgebra over $\Bbbk$. There is an equivalence
  \[
    \begin{aligned}
      \mathbf{Lexf}(\tetramod[C]{}, \mathcal{D}) &\xiso \rComod{(C,\mathcal{D})} \\
      F &\mapsto F(C) \\
      N \cotens \blank &\longmapsfrom N,
    \end{aligned}
  \]
  where the left-hand side is the category of finitary, left exact functors from $\rComod{C}$,
  and the right-hand side---following the notation of Example~\ref{MoreModules}---denotes the category of right $C$-comodules in $\mathcal{D}$.
  Here, $\mathcal{D}$ is endowed with the $\kVect$-structure described in Proposition~\ref{FinitaryAction}.

  In particular, for a coalgebra $D$ over $\Bbbk$, we have
  \[
    \mathbf{Lexf}(\tetramod[C]{}, \lComod{D}) \simeq \biComod{D}{C}.
  \]
\end{proposition}
\begin{proof}
  Since $\tetramod[C]{} \simeq \mathbf{Ind}(\tetramodfd[C]{})$,
  the functor $F$ can be seen as induced to filtered colimits from its restriction to $\tetramodfd[C]{}$,
  and as such is a $\kVect$-module functor, following Example~\ref{finitarythings}.
  Observe that $C$, being a bicomodule over itself, is a right $C$-comodule in $\tetramod[C]{}$.
  Thus, $F(C)$ is a right $C$-comodule in $\mathcal{D}$.
  To see that $F \cong F(C) \cotens \blank$, let $N \in \tetramod[C]{}$ and recall that the bar construction provides a functorial injective resolution:
  \tikzexternaldisable%
  $N \cong \on{eq}(\!\!
  \begin{tikzcd}[ampersand replacement=\&,sep=small]
    {C \otimes N} \& { C \otimes C \otimes N}
    \arrow["{\Delta_{C} \otimes N}", shift left, from=1-1, to=1-2]
    \arrow["{C \otimes \Delta_{N}}"', shift right, from=1-1, to=1-2]
  \end{tikzcd}
  \!\!)$.
  Thus,
  \[
    \begin{aligned}
      F(N)
      &\simeq F\big(\on{eq}(\!\!
        \begin{tikzcd}[ampersand replacement=\&,sep=small]
          {C \otimes N} \& { C \otimes C \otimes N}
          \arrow["{\Delta_{C} \otimes N}", shift left, from=1-1, to=1-2]
          \arrow["{C \otimes \Delta_{N}}"', shift right, from=1-1, to=1-2]
        \end{tikzcd}
        \!\!)\big)
        \simeq \on{eq}\big(\!\!
        \begin{tikzcd}[ampersand replacement=\&,sep=small]
          {F(C \otimes N)} \& {F(C \otimes C \otimes N)}
          \arrow["{F(\Delta_{C} \otimes N)}", shift left, from=1-1, to=1-2]
          \arrow["{F(C \otimes \Delta_{N})}"', shift right, from=1-1, to=1-2]
        \end{tikzcd}
        \!\!\big) \\
      &\simeq \on{eq}\big(\!\!
        \begin{tikzcd}[ampersand replacement=\&,sep=small]
          {F(C) \otimes N} \& {F(C) \otimes C \otimes N}
          \arrow["{F(\Delta_{C}) \otimes N}", shift left, from=1-1, to=1-2]
          \arrow["{F(C) \otimes \Delta_{N}}"', shift right, from=1-1, to=1-2]
        \end{tikzcd}
        \!\!\big)
        = F(C) \cotens N,
    \end{aligned}
  \]
  where the first isomorphism is using the bar construction,
  the second follows from left exactness of $F$,
  the third uses the fact that $F$ is a right $\kVect$-module functor,
  and the fourth follows by observing that $F(\Delta_{C}) = \Delta_{F(C)}$.
  \tikzexternalenable%
\end{proof}

\subsection{Tensor categories and ring categories}

In the sequel, some of our results will phrased in the language of \emph{tensor categories}.
We briefly recall the most important definitions here,
but refer the reader to~\cite[Chapter~4]{EGNO} for a comprehensive account.
All categories and functors in this section are again assumed to be \(\Bbbk\)-linear.

\begin{definition}
  A \emph{tensor category} is a locally finite abelian rigid monoidal category $\mathcal{C}$, such that $\on{End}_{\mathcal{C}}(\mathbb{1}) \cong \Bbbk$.

  A \emph{finite tensor category} is a finite abelian rigid monoidal category $\mathcal{C}$, such that $\on{End}_{\mathcal{C}}(\mathbb{1}) \cong \Bbbk$.

  A \emph{ring category} is a locally finite abelian separately exact monoidal category $\mathcal{C}$, such that $\on{End}_{\mathcal{C}}(\mathbb{1}) \cong \Bbbk$.

  A \emph{finite ring category} is a finite abelian separately exact monoidal category $\mathcal{C}$, such that $\on{End}_{\mathcal{C}}(\mathbb{1}) \cong \Bbbk$.
\end{definition}

We remark that despite apparent similarity in terminology,
the notion of a ring category is not related to that of a \emph{rig category} (also known as a bimonoidal category),
such as those considered in~\cite{JY}.

\begin{definition}
  Let $\mathcal{C}$ be a ring category. A \emph{fiber functor} for $\mathcal{C}$ is a faithful and exact monoidal functor $U\from \mathcal{C} \to \kvect$.
\end{definition}

The next theorem can be seen as a version of \emph{Tannaka--Krein reconstruction} for ring categories.

\begin{theorem}[{\cite[Theorem~5.4.1]{EGNO}}]
  Let $\mathcal{C}$ be a ring category, and assume it admits a fiber functor $U\from \mathcal{C} \to  \kvect$.
  Then there exists a bialgebra $B$ such that there is a monoidal equivalence $K_B\from \mathcal{C} \to \tetramodfd[B]{}$
  and a monoidal natural isomorphism $U_{B} \circ K_{B} \cong U$, where $U_{B}\from \tetramodfd[B]{} \to \kvect$ is the forgetful functor.
  Further, $B$ is unique up to isomorphism, Hopf if and only if\, $\mathcal{C}$ is a tensor category, and finite-dimensional if and only if\, $\mathcal{C}$ is a finite ring category.
\end{theorem}

\subsection{Monads and Beck's monadicity theorem}\label{sec:monads}

As monads and comonads will play an important role in our future investigations,
in this section we will recall elementary facts about their theory.
For a more extensive account, see for example~\cite{MacL}.

\begin{definition}\label{def:monad}
  Let \(\cat{C}\) be a category.
  A \emph{monad on \(\mathcal{C}\)} is an algebra object in the monoidal category \(\mathbf{Cat}(\mathcal{C},\mathcal{C})\).
  More explicitly, it comprises a functor \(T\from \mathcal{C} \to \mathcal{C}\)
  together with natural transformations \(\mu\from T^2 \nt T\) and \(\eta\from \Id_{\mathcal{C}} \nt T\),
  subject to associativity and unitality axioms.
\end{definition}

Clearly, \(\mathcal{C}\) is a left \(\mathbf{Cat}(\mathcal{C},\mathcal{C})\)-module category under evaluation.
Given a monad \(T\) on \(\mathcal{C}\),
the \(T\)-modules in \(\mathcal{C}\) are often referred to as \emph{algebras over \(T\)}.
However,
since we prefer the term algebra object to monoid,
this would be unnecessarily confusing in our setting.
We will thus refer to them simply as \emph{\(T\)-modules} in \(\mathcal{C}\), which is compatible with the more general, bicategorical terminology for \(T\)-modules in categories of the form \(\mathbf{Cat}(\mathcal{C},\mathcal{D})\).
Explicitly, a \(T\)-module in \(\mathcal{C}\) is an object \(X \in \mathcal{C}\) together with a morphism \(\nabla_X\from T(X) \to X\) satisfying the axioms for a module object.
We will denote the category \(\leftmod{T}{\mathcal{C}}\) by \(\mathbf{EM}(T)\) and refer to it as the \emph{Eilenberg--Moore category} of \(T\).
Further, \(\freeleftmod{T}{\mathcal{C}}\) and \(\forgetleftmod{T}{\mathcal{C}}\) will be denoted by \(L_{\mathbf{EM}(T)}\) and \(R_{\mathbf{EM}(T)}\), respectively.

To any monad \(T\) on \(\cat{C}\),
we may also associate its \emph{Kleisli category} \(\mathbf{Kl}(T)\).
On objects, it is given by \(\mathbf{Ob}(\mathbf{Kl}(T)) \defeq \mathbf{Ob}(\mathcal{C})\).
For \(X, Y \in \mathbf{Kl}(T)\), we have \(\mathbf{Kl}(T)(X,Y) \defeq \mathcal{C}(X,T(Y))\);
composition is defined as follows:
\begin{align*}
  \circ \from \mathbf{Kl}(T)(Y, Z) \times \mathbf{Kl}(T)(X, Y) &\to \mathbf{Kl}(T)(X, Z) \\
  (g, f) &\mapsto X \xrightarrow{\ f\ } T(Y) \xrightarrow{\;Tg\;} T^2(Z) \xrightarrow{\mu_Z} T(Z)
\end{align*}

\begin{proposition}\label{prop:Kleisli-adjunction}
  Let \(T\) be a monad on a category \(\cat{C}\).
  There is an adjunction
  \[
    \begin{mytikzcd}[ampersand replacement=\&]
      {\cat{C}} \& {\mathbf{Kl}(T)}
      \arrow[""{name=0, anchor=center, inner sep=0}, "{L_{\mathbf{Kl}(T)}}", shift left=2, from=1-1, to=1-2]
      \arrow[""{name=1, anchor=center, inner sep=0}, "{R_{\mathbf{Kl}(T)}}", shift left=2, from=1-2, to=1-1]
      \arrow["\dashv"{anchor=center, rotate=-90}, draw=none, from=0, to=1]
    \end{mytikzcd}
  \]
  where \(L_{\mathbf{Kl}(T)}\) is identity on objects and sends \(f \in \cat{C}(X,Y)\) to \(\eta_{Y} \circ f\),
  and \(R_{\mathbf{Kl}(T)}\) sends \(X\) to \(T(X)\) and \(f \in \cat{C}(X, Y)\) to \(\mu_{Y} \circ T(f)\).
\end{proposition}

\begin{example}\label{ex:comonads}
  Dualising the above considerations, a \emph{comonad} may be defined as a coalgebra object in the monoidal category \(\mathbf{Cat}(\cat{C},\cat{C})\).
  For a comonad \(S\) on \(\cat{C}\),
  an \emph{\(S\)-comodule} is an object \(C \in \mathcal{C}\) together with a morphism \(C \to S(C)\) satisfying the axioms for a comodule object.
  We will denote the category of \(S\)-comodules by \(\mathbf{EM}(S)\) and the co-Kleisli category of \(S\) by \(\mathbf{Kl}(S)\).
  They have analogously defined adjunctions
  \[
    \begin{mytikzcd}[ampersand replacement=\&]
      {\mathbf{Kl}(S)} \& {\mathcal{C}} \& {\text{and}} \& {\mathbf{EM}(S)} \& {\mathcal{C}}
      \arrow[""{name=0, anchor=center, inner sep=0}, "{{L_{\mathbf{Kl}(S)}}}", shift left=2, from=1-1, to=1-2]
      \arrow[""{name=1, anchor=center, inner sep=0}, "{{R_{\mathbf{Kl}(S)}}}", shift left=2, from=1-2, to=1-1]
      \arrow[""{name=2, anchor=center, inner sep=0}, "{L_{\mathbf{EM}(S)}}", shift left=2, from=1-4, to=1-5]
      \arrow[""{name=3, anchor=center, inner sep=0}, "{R_{\mathbf{EM}(S)}}", shift left=2, from=1-5, to=1-4]
      \arrow["\dashv"{anchor=center, rotate=-90}, draw=none, from=0, to=1]
      \arrow["\dashv"{anchor=center, rotate=-90}, draw=none, from=2, to=3]
    \end{mytikzcd}
  \]
\end{example}

\begin{example}\label{ex:dual-pairs}
  A dual pair \((\ld{X}, X, \varepsilon, \eta)\) in a monoidal category \(\mathcal{C}\)
  induces an algebra object \((X \otimes \ld{X},\ X \otimes \varepsilon \otimes \ld{X},\ \eta)\),
  as well as a coalgebra object \((\ld{X} \otimes X,\ \ld{X} \otimes \eta \otimes X,\ \varepsilon)\),
  in \(\cat{C}\).

  In particular,
  given an adjunction \(\adj{F}{U}{\cat{C}}{\cat{D}}\),
  the functors \(U F\) and \(F U\) have canonical monad and comonad structures on \(\cat{C}\) and \(\cat{D}\), respectively.
\end{example}

For an adjunction \(\adj{F}{U}{\cat{C}}{\cat{D}}\) with unit \(\eta\) and counit \(\varepsilon\),
the monad \(T \defeq UF\) is the same as the monad induced by the associated Eilenberg--Moore and Kleisli adjunctions.
There exist canonical \emph{comparison functors}
\[
  \begin{aligned}
    K_{\mathbf{EM}(T)}\from \mathcal{D} &\to \mathbf{EM}(T) \\
    D &\mapsto \big(U(D),\ U\varepsilon_{D}\from UFU(D) \to U(D)\big) \\
    D \xrightarrow{\ f\ } D' &\mapsto (U(D), U\varepsilon_D) \xrightarrow{U(f)} (U(D'),\ U\varepsilon_{D'})
  \end{aligned}
\]
and
\[
  K_{\mathbf{Kl}(T)}\from \mathbf{Kl}(T) \to \mathcal{D}, \qquad X \mapsto F(X), \qquad f \in \mathcal{D}(X,T(Y)) \mapsto \varepsilon_{F(Y)} \circ F(f).
\]

\begin{remark}\label{rmk:initial-terminal-adjunction-categories}
  In fact, the Eilenberg--Moore and the Kleisli category are the terminal and initial objects in the suitably defined category of adjunctions producing the monad \(T\).
  Thus, the diagram
  \begin{equation}\label{EMKlDiagram}
    \begin{mytikzcd}[ampersand replacement=\&]
      {\mathbf{Kl}(T)} \&\& {\mathcal{D}} \&\& {\mathbf{EM}(T)} \\
      \\
      \&\& {\mathcal{C}}
      \arrow["{K_{\mathbf{Kl}(T)}}", from=1-1, to=1-3]
      \arrow[""{name=0, anchor=center, inner sep=0}, "{R_{\mathbf{Kl}(T)}}"{pos=0.4}, shift left, from=1-1, to=3-3]
      \arrow["{K_{\mathbf{EM}(T)}}", from=1-3, to=1-5]
      \arrow[""{name=1, anchor=center, inner sep=0}, "G", shift left=2, from=1-3, to=3-3]
      \arrow[""{name=2, anchor=center, inner sep=0}, "{R_{\mathbf{EM}(T)}}", shift left=3, from=1-5, to=3-3]
      \arrow[""{name=3, anchor=center, inner sep=0}, "{L_{\mathbf{Kl}(T)}}", shift left=3, from=3-3, to=1-1]
      \arrow[""{name=4, anchor=center, inner sep=0}, "F", shift left=2, from=3-3, to=1-3]
      \arrow[""{name=5, anchor=center, inner sep=0}, "{L_{\mathbf{EM}(T)}}"{pos=0.6}, shift left, from=3-3, to=1-5]
      \arrow["\dashv"{anchor=center}, draw=none, from=4, to=1]
      \arrow["\dashv"{anchor=center, rotate=49}, draw=none, from=3, to=0]
      \arrow["\dashv"{anchor=center, rotate=-50}, draw=none, from=5, to=2]
    \end{mytikzcd}
  \end{equation}
  commutes,
  and its commutativity characterizes \(K_{\mathbf{EM}(T)}\) and \(K_{\mathbf{Kl}(T)}\) completely.

  Further, the composite \(K_{\mathbf{EM}(T)} \circ K_{\mathbf{Kl}(T)}\) is fully faithful,
  its full image consisting of \emph{free modules}.
  We will denote the canonical inclusion \(\mathbf{Kl}(T) \hookrightarrow \mathbf{EM}(T)\) by
  \begin{equation}\label{denotedbyiota}
    \iota\from \mathbf{Kl}(T) \hookrightarrow \mathbf{EM}(T), \qquad X \mapsto (T(X), \mu_X), \qquad f \in \cat{C}(X, T(Y)) \mapsto \mu \circ T(f).
  \end{equation}
\end{remark}

Similarly, to Remark~\ref{rmk:initial-terminal-adjunction-categories},
the Eilenberg--Moore category and the Kleisli category for the comonad \(S \defeq FG\) can be characterized as a terminal and initial object, respectively.
This yields functors \(K_{\mathbf{Kl}(S)}\) and \(K_{\mathbf{EM}(S)}\) such that the following diagram commutes
\[
  \begin{mytikzcd}[ampersand replacement=\&]
    \&\& {\mathcal{D}} \\
    \\
    {\mathbf{Kl}(S)} \&\& {\mathcal{C}} \&\& {\mathbf{EM}(S)}
    \arrow[""{name=0, anchor=center, inner sep=0}, "{{R_{\mathbf{Kl}(S)}}}"{pos=0.6}, shift left, from=1-3, to=3-1]
    \arrow[""{name=1, anchor=center, inner sep=0}, "G", shift left=2, from=1-3, to=3-3]
    \arrow[""{name=2, anchor=center, inner sep=0}, "{{R_{\mathbf{EM}(S)}}}", shift left=3, from=1-3, to=3-5]
    \arrow[""{name=3, anchor=center, inner sep=0}, "{{L_{\mathbf{Kl}(S)}}}", shift left=3, from=3-1, to=1-3]
    \arrow["{{K_{\mathbf{Kl}(S)}}}"', from=3-1, to=3-3]
    \arrow[""{name=4, anchor=center, inner sep=0}, "F", shift left=2, from=3-3, to=1-3]
    \arrow["{{K_{\mathbf{EM}(S)}}}"', from=3-3, to=3-5]
    \arrow[""{name=5, anchor=center, inner sep=0}, "{{L_{\mathbf{EM}(S)}}}"{pos=0.4}, shift left, from=3-5, to=1-3]
    \arrow["\dashv"{anchor=center, rotate=-48}, draw=none, from=3, to=0]
    \arrow["\dashv"{anchor=center}, draw=none, from=4, to=1]
    \arrow["\dashv"{anchor=center, rotate=49}, draw=none, from=5, to=2]
  \end{mytikzcd}
\]

The following two results are the motivating observations for the theory of \emph{contramodules},
which will be relevant for us in Section~\ref{sec:Hopf-trimodules}.

\begin{proposition}[{\cite[Theorem~V.8.2]{MLM}}]\label{pyramidscheme}
  If\, \(T\) is a monad and \(G\) is a right adjoint to \(T\), then \(G\) is a comonad and there is a canonical isomorphism \(\mathbf{EM}(T) \cong \mathbf{EM}(G)\).
\end{proposition}

\begin{proposition}[{\cite[Theorem~3]{Kl}}]\label{einerKleinerKleisliÄquivalenz}
  If\, \(T\) is a monad and \(K\) is a left adjoint to \(T\), then \(K\) is a comonad and there is a canonical isomorphism \(\mathbf{Kl}(T) \cong \mathbf{Kl}(K)\).
\end{proposition}

Importantly, the Eilenberg–Moore category of a monad left adjoint to a comonad is \emph{not} equivalent to the Eilenberg–Moore category of its right adjoint.

\section{Lax and oplax module monads and comonads}\label{sec:module-monads}

\subsection{Module categories associated to (op)lax module (co)monads}

The aim of this section is to study lax and oplax module monads,
as well as their comonadic counterparts. To that end, we use the results of~\cite{HZ2} and establish additional variants thereof.

In the following definition,
we need to be careful in specifying that the unit and the counit of the adjunction are $\mathcal{C}$-module transformations.
Otherwise, unexpected counterexamples may arise, see~\cite{HZ1}.

\begin{definition}\label{OplaxModuleAdjunctions}
  An \emph{oplax \(\mathcal{C}\)-module adjunction} consists of
  an adjoint pair of oplax \(\mathcal{C}\)-module functors \(\adj{F}{U}{\cat{M}}{\cat{N}}\),
  such that the unit and counit of the adjunction are transformations of oplax \(\mathcal{C}\)-module functors.
\end{definition}

\begin{definition}\label{def:oplax-module-monad}
  An \emph{oplax \(\mathcal{C}\)-module monad} is a monoid in the category
  \(\mathbf{Oplax}\mathcal{C}\text{-}\mathbf{Mod}(\cat{M},\cat{M})\)
  of oplax \(\cat{C}\)-module endofunctors on a \(\cat{C}\)-module category \(\cat{M}\).
  Explicitly,
  this comprises a monad \(T\) on \(\cat{M}\)
  that is also an oplax \(\mathcal{C}\)-module functor,
  such that the unit and multiplication maps are transformations of oplax \(\mathcal{C}\)-module functors.
\end{definition}

Analogously to
Definitions~\ref{OplaxModuleAdjunctions} and~\ref{def:oplax-module-monad},
one defines the notions lax \(\mathcal{C}\)-module adjunctions,
as well as lax \(\mathcal{C}\)-module comonads.

\begin{remark}\label{rmk:oplax-module-monad-as-comodule-monad}
  An oplax $\mathcal{C}$-module monad of Definition~\ref{def:oplax-module-monad} can be seen as a special case of a \emph{comodule monad} in the sense of~\cite[Definition~4.8]{HZ2}.
  For an opmonoidal monad $B$ on $\mathcal{C}$, a comodule monad is a monad on a right \(\cat{C}\)-module category \(\cat{M}\), equipped with a coherent coaction \(C(\blank \ract \bblank) \nt C(\blank) \ract B(\bblank)\).
  An oplax $\mathcal{C}$-module monad is%
  ---barring the fact that we study left rather than right $\mathcal{C}$-module categories---%
  a comodule monad for the identity functor on $\mathcal{C}$.
\end{remark}

\begin{lemma}
  The left adjoint in a lax module adjunction is a strong module functor. Similarly, the right adjoint in an oplax module adjunction is a strong module functor.
\end{lemma}

\begin{proof}
  If \(F \dashv G\) is an oplax \(\cat{C}\)-module adjunction, it is easy to verify that
  the inverse of \(G_{\mathsf{a}}\) is,
  for all \(V \in \cat{C}\) and \(N \in \cat{N}\),
  given by
  \begin{equation}\label{rightadjointoplaxisstrong}
    V \lact G(N)
    \xrightarrow{\eta_{V \lact G(N)}} GF(V \lact G(N))
    \xrightarrow{G F_{\mathsf{a}}} G(V \lact FG(N))
    \xrightarrow{G(V \lact \varepsilon_N)} G(V \lact N).
  \end{equation}
  The case of lax \(\cat{C}\)-module adjunctions is similar.
\end{proof}

Similarly to Definition~\ref{def:lift-to-monoidal-adjunction},
one defines the concept of a \emph{lift} of \(F \dashv G\) to a lax or oplax \(\cat{C}\)-module adjunction.
The following results are simplified versions of~\cite[Theorems~4.13 and~2]{HZ2}.

\begin{proposition}\label{prop:module-adjunctions}
  Let\, \(\cat{C}\) be a monoidal category,
  and suppose that \(\cat{M}\) and \(\cat{N}\) are left \(\cat{C}\)-module categories.
  Given an adjunction \(\adj{F}{U}{\cat{M}}{\cat{N}}\),
  there is a one-to-one correspondence between
  strong \(\cat{C}\)-module functor structures on \(U\),
  and lifts of\, \(F \dashv U\) to an oplax \(\cat{C}\)-module adjunction.
\end{proposition}

Explicitly, Proposition~\ref{prop:module-adjunctions} yields the following correspondence:
given a strong \(\cat{C}\)-module functor structure on \(U\),
for all \(V \in \cat{C}\) and \(M \in \cat{M}\)
one may define \(F_{\mathsf{a}} \from F(V \lact M) \nt V \lact F(M)\) by
\[
  F(V \lact M)
  \xrightarrow{F(V \lact \eta_M)} F(V \lact UF(M))
  \xrightarrow{F U_{\mathsf{a}}^{-1}} FU(V \lact F(M))
  \xrightarrow{\varepsilon_{V \lact F(M)}} V \lact F(M).
\]
The strong \(\cat{C}\)-module functor structure on \(U\) obtained from a lift of \(F \dashv U\) to an oplax \(\cat{C}\)-module adjunction is described in Equation~\ref{rightadjointoplaxisstrong}.

\begin{corollary}\label{oplaxEMcorrespondence}
  Let \(\mathcal{M}\) be a \(\mathcal{C}\)-module category,
  and suppose \(T\) to be a monad on \(\mathcal{M}\).
  Define \(L \defeq L_{\mathbf{EM}(T)}\)
  and \(R \defeq R_{\mathbf{EM}(T)}\).
  There is a bijection
  \begin{align*}
    \big\{
    \text{oplax \(\mathcal{C}\)-module monad structures on \(T\)}
    \big\}
    \xiso
    &\,\Bigg\{
      {\begin{aligned}
        &\text{\(\mathcal{C}\)-module category structures on \(\mathbf{EM}(T)\)}\\
        &\text{such that \(R\) is a strict module functor}
      \end{aligned}}
      \Bigg\} \\
    \setj{T(\blank \triangleright \bblank) \xrightarrow{\;\;\sigma\;\;} \blank \triangleright T(\bblank)}
    \mapsto &\setj{T(\blank \triangleright \bblank) \xrightarrow{\;\;\sigma\;\;} \blank \triangleright T(\bblank) \xrightarrow{\id \,\triangleright\, \mathrm{act}} \blank \triangleright \bblank}\\
    \setj{R L(\blank \lact \bblank) \xrightarrow{\delta^{(R)} \,\circ\, R \delta^{(L)}} \blank \lact R L(\bblank)}
    \longmapsfrom& \setj{R(\blank \lact \bblank) \xrightarrow{\;\delta^{(R)}\;} \blank \lact R(\bblank)}
  \end{align*}
\end{corollary}

Proposition~\ref{prop:dual-module-adjunctions} is the dual variant of Proposition~\ref{prop:module-adjunctions}, where strong $\mathcal{C}$-module structures on the left%
---rather than the right---%
adjoint are studied.

\begin{proposition}\label{prop:dual-module-adjunctions}
  Let \(\cat{C}\) be a monoidal category,
  and suppose that \(\cat{M}\) and \(\cat{N}\) are left
  \(\cat{C}\)-module categories.
  Given an adjunction \(\adj{F}{U}{\cat{M}}{\cat{N}}\),
  there is a one-to-one correspondence between
  lifts of\, \(F\) to a strong \(\cat{C}\)-module functor,
  and lifts of\, \(F \dashv U\) to a lax \(\cat{C}\)-module adjunction.
\end{proposition}
\begin{proof}
  Suppose that \(\adj{F}{U}{\cat{M}}{\cat{N}}\) is a lax \(\cat{C}\)-module adjunction.
  Define the inverse of the action \(F_{\mathrm{a}} \from \blank \lact F(\bblank) \nt F(\blank \lact \bblank)\)
  by
  \[
    F(\blank \lact \bblank)
    \xrightarrow{F(\eta \,\lact\, \bblank)} F(UF(\blank) \lact \bblank)
    \xrightarrow{F U_{\mathrm{a}}} FU(F(\blank) \lact \bblank)
    \xrightarrow{\;\;\varepsilon\;\;} F(\blank) \lact \bblank.
  \]
  Conversely, given an adjunction \(\adj{F}{U}{\cat{M}}{\cat{N}}\)
  such that \(F\) is a strong \(\cat{C}\)-module functor,
  define
  \[
    U(\blank) \lact \bblank
    \xrightarrow{\;\;\eta\;\;} UF(U(\blank) \lact \bblank)
    \xrightarrow{U F_{\mathrm{a}}^{-1}} U(FU(\blank) \lact \bblank)
    \xrightarrow{U(\varepsilon \,\lact\, \bblank)} U(\blank \lact \bblank).
  \]

  Reading the string diagrams in the proof of~\cite[Theorem~4.13]{HZ2} upside down
  verifies the necessary the coherence conditions,
  and that these two constructions are inverses of each other.
\end{proof}

We immediately obtain a version of Corollary~\ref{oplaxEMcorrespondence}
for the Kleisli category of a monad.

\begin{corollary}\label{laxKlCorrespondence}
  Let \(\mathcal{M}\) be a left \(\mathcal{C}\)-module category,
  and suppose \(T\) to be a monad on \(\mathcal{M}\).
  There is a bijection between
  lax \(\cat{C}\)-module monad structures on \(T\),
  and
  \(\cat{C}\)-module category structures on \(\mathbf{Kl}(T)\),
  such that \(L_{\mathbf{Kl}(T)}\) is a strict \(\cat{C}\)-module functor.
\end{corollary}
\begin{proof}
  Let \(\mathbf{Kl}(T)\) be a \(\cat{C}\)-module category
  such that \(L_{\mathbf{Kl}(T)}\) is a strict \(\cat{C}\)-module functor.
  Then the functor
  \(T \defeq R_{\mathbf{Kl}(T)} \circ L_{\mathbf{Kl}(T)} \from \cat{M} \to \cat{M}\)
  is a lax \(\cat{C}\)-module monad by Proposition~\ref{prop:dual-module-adjunctions}.

  Conversely, if \(T\) is a lax \(\cat{C}\)-module monad,
  \(\mathbf{Kl}(T)\) becomes a \(\cat{C}\)-module category as follows: for \(V \in \cat{C}\) and \(M,N \in \cat{M}\), we set \(V \lact_{\mathbf{Kl}(T)} M \defeq V \lact_{\mathcal{M}} M\) on the level of objects, and on the level of morphisms, we define ${(\blank \lact_{\mathbf{Kl}(T)} \bblank)}_{(V,M),N}$ as follows:
  \[
    \mathscale{0.88}{%
      \mathbf{Kl}(T)(M,N) = \mathcal{M}(M,T(N))
      \xrightarrow{V \lact \blank} \mathcal{M}(V \lact M, V \lact T(N))
      \xrightarrow{T_{\mathsf{a}} \circ \blank} \mathcal{M}(V \lact M, T(V \lact N)) = \mathbf{Kl}(T)(V\lact M, V \lact N)
    }.
  \]
  It is easy to check that these assignments define a \(\cat{C}\)-module structure for which \(L_{\mathbf{Kl}(T)}\) is a strict \(\cat{C}\)-module functor.

  Since the \(\cat{C}\)-module structure of \(\mathbf{Kl}(T)\)
  and the lax module monad structure on \(T\) uniquely determine the \(\cat{C}\)-module structure of \(L_{\mathbf{Kl}(T)}\), the inverse
  \({(L_{\mathbf{Kl}(T)})}_{\mathsf{a}}^{-1}\) is given as in Proposition~\ref{prop:dual-module-adjunctions}.
  Thus, these two constructions are again inverse to each other.
\end{proof}

Proposition~\ref{strongcomparisons} below is a simplified version of~\cite[Proposition~4.14]{HZ2} in the case of the Eilenberg--Moore category;
similarly to Corollary~\ref{oplaxEMcorrespondence}, the case of the Kleisli category is analogous.

\begin{proposition}\label{strongcomparisons}
  Let \(\cat{M}\) and \(\cat{N}\) be left \(\cat{C}\)-module categories.
  \begin{enumerate}
    \item Let \(\adj{F}{U}{\cat{M}}{\cat{N}}\) be an oplax\, \(\mathcal{C}\)-module adjunction
    and consider the corresponding \(\mathcal{C}\)-module category structure on \(\mathbf{EM}(T)\) from Proposition~\ref{oplaxEMcorrespondence}.
    Then the functor \(K_{\mathbf{EM}}(U F)\) is a strong \(\mathcal{C}\)-module functor.
    \item Let \(\adj{F}{U}{\cat{M}}{\cat{N}}\) be
    a lax\, \(\mathcal{C}\)-module adjunctions
    and consider the corresponding \(\mathcal{C}\)-module category structure on \(\mathbf{Kl}(T)\) from Proposition~\ref{laxKlCorrespondence}.
    Then the functor \(K_{\mathbf{Kl}}(U F)\) is a strong \(\mathcal{C}\)-module functor.
  \end{enumerate}
\end{proposition}

\subsection{Extending module category structures from Kleisli to Eilenberg--Moore categories}

In this section, we investigate how and under which conditions one can extend the $\mathcal{C}$-module structure of the Kleisli category of a lax $\mathcal{C}$-module monad to its category of modules.

From now on,
we implicitly assume all categories and functors to be \(\Bbbk\)-linear.

\begin{proposition}[{\cite[Proposition~3.2]{BZBJ}}]\label{prop:abelian-EM-cat}
  Suppose that \(\cat{A}\) is abelian,
  let \(T\) be a right exact monad,
  and \(S\) a left exact comonad on \(\cat{A}\).
  Then \(\mathbf{EM}(T)\) and \(\mathbf{EM}(S)\) are abelian.
\end{proposition}

Conversely, of course, the monad associated to an exact monadic functor of abelian categories is right exact, and similarly for left exact comonads.

\begin{proposition}\label{locomonads}
  Let $\mathcal{A}$ be a locally finite abelian category, and let $K$ be a left exact, finitary comonad on $\mathbf{Ind}(\mathcal{A})$. Then $\mathbf{EM}(K)$ also is of the form $\mathbf{Ind}(\mathcal{E})$ for a locally finite abelian category $\mathcal{E}$. Further, $\mathcal{E}$ can be chosen to be the category of compact objects in $\mathbf{EM}(K)$, and it can be characterized as the objects sent to compact objects under the forgetful functor $L_{\mathbf{EM}(K)}\from \mathbf{EM}(K) \to \mathcal{A}$.

  In particular, if $S$ is a left exact comonad on $\mathcal{A}$, then $\mathbf{EM}(\mathbf{Ind}(S)) \simeq \mathbf{Ind}(\mathbf{EM}(S))$, and $\mathbf{EM}(S)$ is locally finite abelian.
\end{proposition}

\begin{proof}
  Let $D$ be the $\Bbbk$-coalgebra such that $\mathcal{A} \simeq \tetramodfd[D]{}$, so that $\mathbf{Ind}(\mathcal{A}) \simeq \tetramod[D]{}$. By Example~\ref{TakeuchiEilenbergWatts}, there is a $D$-$D$-bicomodule $C$ such that $K \simeq C \cotens \blank$.

  Similarly to~\cite[Remark~2.4]{Tak},
  under this isomorphism the comonad structure on \(K\) corresponds
  to maps \(C \to C \square_{D} C\) and \(C \to D\),
  which endow \(C\) with the structure of a \(\Bbbk\)-coalgebra,
  together with a coalgebra morphism \(C \to D\).
  Formally, a \(K\)-comodule in \(\tetramod[D]{}\) is a \(D\)-comodule
  together with a \(C\)-comodule structure
  that restricts to the given \(D\)-comodule along the coalgebra morphism \(C \to D\).
  Morphisms of \(K\)-comodules in \(\tetramod[D]{}\) are precisely \(C\)-comodule morphisms,
  and so we have \(\mathbf{EM}(K) \simeq \tetramod[C]{}\). Thus we may set $\mathcal{E} = \tetramodfd[C]{}$. The characterization of compact objects in $\mathbf{EM}(K)$ in terms of the images of the functor $L_{\mathbf{EM}(K)}$ follows immediately from observing that under the equivalence \(\mathbf{EM}(K) \simeq \tetramod[C]{}\), the functor $L_{\mathbf{EM}(K)}$ is simply the restriction functor $\tetramod[C]{} \to \tetramod[D]{}$.

  The second part of the statement follows by observing that $\mathbf{EM}(S)$ is the category of compact objects in $\mathbf{EM}(\mathbf{Ind}(S))$, by the first part of the statement, and by observing that $\mathbf{EM}(S)$ is also the category of compact objects in $\mathbf{Ind}(\mathbf{EM}(S))$. Since both $\mathbf{Ind}(\mathbf{EM}(S))$ and $\mathbf{EM}(\mathbf{Ind}(S))$ are locally finitely presentable, the equivalence between their respective categories of compact objects establishes an equivalence between $\mathbf{Ind}(\mathbf{EM}(S))$ and $\mathbf{EM}(\mathbf{Ind}(S))$.
\end{proof}

The analogue of Proposition~\ref{locomonads} for finite abelian categories is significantly simpler:

\begin{proposition}\label{finitemonads}
  Let \(T\) be a right exact monad on a finite abelian category \(\mathcal{A}\).
  Then \(\mathbf{EM}(T)\) is a finite abelian category.

  Further, any projective object in \(\mathbf{EM}(T)\) is a direct summand of one of the form $\iota(P)$,
  where $\iota\from \mathbf{Kl}(T) \to \mathbf{EM}(T)$ is the canonical embedding of Equation~\eqref{denotedbyiota} and $P \in \proj{\mathcal{A}}$.
  Denoting the category of objects of this form by $\mathbf{Kl}_{p}(T)$,
  we obtain an equivalence $\proj{\mathbf{EM}(T)} \simeq {\mathbf{Kl}_{p}(T)}^{\mathrm{Cy}}$.
\end{proposition}
\begin{proof}
  Let \(A\) be a finite-dimensional \(\Bbbk\)-algebra such that \(\mathcal{A} \simeq \lmod{A}\).
  Then there is an \(A\)-\(A\)-bimodule \(B\) such that \(T \cong B \otimes_{A} \blank\).
  Similarly to the proof of Proposition~\ref{locomonads},
  monad structures on \(T\) correspond to pairs consisting of a finite-dimensional \(\Bbbk\)-algebra structure on \(B\) and an algebra homomorphism \(A \to B\).
  Under this correspondence we have \(\mathbf{EM}(T) \simeq \lmod{B}\).
  The latter category is clearly a finite abelian category.

  Let $P \in \proj{\mathcal{A}}$.
  Then $\mathbf{EM}(T)(\iota(P), \blank) \cong \mathcal{A}(P, \blank)$, proving the exactness of the left-hand side, and thus projectivity of $\iota(P)$.
  For $X \in \mathbf{EM}(T)$, using the fact that $\mathcal{A}$ has enough projectives,
  we may fix an epimorphism $Q \xtwoheadrightarrow{\ q\ } X$ in $\mathcal{A}$,
  where $Q \in \proj{\mathcal{A}}$.
  We obtain a composite epimorphism
  \[
    \iota(P) \xtwoheadrightarrow{\iota(q)} \iota(X) = TX \xtwoheadrightarrow{\nabla_{X}} X,
  \]
  the first part of which is epic by right exactness of $\iota$.
  This proves the latter claim.
\end{proof}

\begin{lemma}\label{finitecorestriction}
  The embedding
  \(\mathbf{EM}(T) \hookrightarrow [{\mathbf{Kl}(T)}^{\on{op}}, \mathbf{Vec}]\)
  can be corestricted to an embedding
  \(\mathbf{EM}(T) \hookrightarrow \fincolimit{\mathbf{Kl}(T)}\) to the finite cocompletion of\, \(\mathbf{Kl}(T)\).
\end{lemma}
\begin{proof}
  For \(T\)-modules \(X\) and \(Y\),
  notice that there is a bijection
  \begin{equation}\label{eq:fincorestr:iso}
    \begin{aligned}
      \mathbf{EM}(T)(T(X), Y) &\,\ \cong\,\ \cat{C}(X, Y) \\
      f\from T(X) \to Y &\mapsto f \circ \eta_X \\
      \nabla_Y \circ Tg &\longmapsfrom g\from X \to Y,
    \end{aligned}
  \end{equation}
  which induces an isomorphism \(\mathbf{EM}(T)(T(\blank), \bblank) \cong \cat{C}(\blank, \bblank)\),
  the left-hand side also being obtained from the functor
  \[
    \mathbf{EM}(T)
    \xrightarrow{\ \yo\ } [{\mathbf{EM}(T)}^{\on{op}}, \mathbf{Vec}]
    \xrightarrow{\ \iota^{\ast}\ } [{\mathbf{Kl}(T)}^{\on{op}},\mathbf{Vec}].
  \]

  Now, recall that the coequalizer
  \[
    \begin{mytikzcd}[ampersand replacement=\&]
      {T^{2}(X)} \& {T(X)} \& X
      \arrow["{\mu_{X}}", shift left=1, from=1-1, to=1-2]
      \arrow["{T \nabla_X}"', shift right=1, from=1-1, to=1-2]
      \arrow["{\nabla_X}", from=1-2, to=1-3]
    \end{mytikzcd}
  \]
  in \(\mathbf{EM}(T)\) is an absolute coequalizer\footnote{%
    \,A coequalizer is called \emph{absolute} if it is preserved by any functor;
    see e.g.~\cite{Pa} for details.%
  }
  in \(\cat{C}\);
  thus, its image under the Yoneda embedding $\yo$ yields a coequalizer
  \[
    \begin{mytikzcd}[ampersand replacement=\&]
      {\cat{C}(\blank, T^2(X))} \& {\cat{C}(\blank, T(X))} \& {\cat{C}(\blank, X).}
      \arrow["{{(\mu_X)}_{*}}", shift left=1, from=1-1, to=1-2]
      \arrow["{{(T \nabla_X)}_{*}}"', shift right=1, from=1-1, to=1-2]
      \arrow["{\nabla_{X;*}}", from=1-2, to=1-3]
    \end{mytikzcd}
  \]
  Passing under the isomorphism of Equation~\eqref{eq:fincorestr:iso},
  one obtains a coequaliser
  \[
    \begin{mytikzcd}[ampersand replacement=\&]
      {\mathbf{EM}(T)(T(\blank), T^2(X))} \& {\mathbf{EM}(T)(T(\blank), T(X))} \& {\mathbf{EM}(T)(T(\blank), X),}
      \arrow["{{(\mu_X)}_{*}}", shift left=1, from=1-1, to=1-2]
      \arrow["{{(T\nabla_X)}_{*}}"', shift right=1, from=1-1, to=1-2]
      \arrow["{\nabla_{X;*}}", from=1-2, to=1-3]
    \end{mytikzcd}
  \]
  which, in turn, is isomorphic to
  \[
    \begin{mytikzcd}[ampersand replacement=\&]
      {\mathbf{Kl}(T)(\blank, T(X))} \& {\mathbf{Kl}(T)(\blank, X)} \& {\mathbf{EM}(T)(T(\blank), X).}
      \arrow["{{(\mu_X)}_{*}}", shift left=1, from=1-1, to=1-2]
      \arrow["{{(T\nabla_X)}_{*}}"', shift right=1, from=1-1, to=1-2]
      \arrow["{\nabla_{X;*}}", from=1-2, to=1-3]
    \end{mytikzcd}
  \]
  This proves the result.
\end{proof}

\begin{proposition}
  Let \(T\) be a right exact monad on an abelian category \(\mathcal{A}\).
  The inclusion functor \(\mathbf{EM}(T) \hookrightarrow \fincolimit{\mathbf{Kl}(T)}\) has a left adjoint, and the counit of the adjunction is a natural isomorphism.
  In other words, \(\mathbf{EM}(T)\) is a reflective subcategory of\, \(\fincolimit{\mathbf{Kl}(T)}\).
  Further, the left adjoint \(\fincolimit{\mathbf{Kl}(T)} \to \mathbf{EM}(T)\) is the right exact extension of the inclusion \(\iota \from \mathbf{Kl}(T) \hookrightarrow \mathbf{EM}(T)\).
\end{proposition}
\begin{proof}
  By Proposition~\ref{prop:abelian-EM-cat},
  the Eilenberg--Moore category of \(T\) is finitely cocomplete.
  Similarly to~\cite[Proposition~6.3.1]{KS},
  the functor \(\fincolimit{\mathbf{EM}(T)} \to \mathbf{EM}(T)\)
  that extends the functor \(\on{Id}_{\mathbf{EM}(T)}\from \mathbf{EM}(T) \to \mathbf{EM}(T)\)
  to the freely adjoined finite colimits
  is left adjoint to the inclusion \(\mathbf{EM}(T) \hookrightarrow \fincolimit{\mathbf{EM}(T)}\).

  On the other hand, there is an adjunction
  \begin{equation}\label{eq:kan-adjunction}
    \begin{mytikzcd}[ampersand replacement=\&]
      {[{\mathbf{Kl}(T)}^{\on{op}},\mathbf{Vec}]} \& {[{\mathbf{EM}(T)}^{\on{op}},\mathbf{Vec}],}
      \arrow[""{name=0, anchor=center, inner sep=0}, "{\iota_{!}}", shift left=2, from=1-1, to=1-2]
      \arrow[""{name=1, anchor=center, inner sep=0}, "{\iota^{\ast}}", shift left=2, from=1-2, to=1-1]
      \arrow["\dashv"{anchor=center, rotate=-90}, draw=none, from=0, to=1]
    \end{mytikzcd}
  \end{equation}
  where \(\iota_{{!}}\) is the left Kan extension of \(\yo_{\mathbf{EM}(T)} \circ \iota\) along \(\yo_{\mathbf{Kl}(T)}\), see for example~\cite[Corollary~3.3]{Str3}.
  Equation~\eqref{eq:kan-adjunction} restricts to an adjunction
  \[
    \begin{mytikzcd}[ampersand replacement=\&]
      {\fincolimit{\mathbf{Kl}(T)}} \& {\fincolimit{\mathbf{EM}(T)},}
      \arrow[""{name=0, anchor=center, inner sep=0}, "{\iota_{!}}", shift left=2, from=1-1, to=1-2]
      \arrow[""{name=1, anchor=center, inner sep=0}, "{\iota^{\ast}}", shift left=2, from=1-2, to=1-1]
      \arrow["\dashv"{anchor=center, rotate=-90}, draw=none, from=0, to=1]
    \end{mytikzcd}
  \]
  since, by definition, \(\iota_{!}\) sends finite colimits of representables to finite colimits of representables,
  and \(\iota^{\ast}\) has the same property, since it preserves colimits, and, by Lemma~\ref{finitecorestriction}, sends representables to finite colimits of representables.
  Thus we have the following commutative diagram
  \[
    \begin{mytikzcd}[ampersand replacement=\&]
      {\mathbf{EM}(T)} \& {\fincolimit{\mathbf{EM}(T)}} \& {\mathbf{EM}(T)} \\
      \& {\fincolimit{\mathbf{Kl}(T)}} \& {\mathbf{Kl}(T)}
      \arrow[""{name=0, anchor=center, inner sep=0}, shift right=2, from=1-1, to=1-2]
      \arrow[curve={height=13pt}, hook', from=1-1, to=2-2]
      \arrow[""{name=1, anchor=center, inner sep=0}, shift right=2, from=1-2, to=1-1]
      \arrow[""{name=2, anchor=center, inner sep=0}, "{\iota^{\ast}}", shift left=2, from=1-2, to=2-2]
      \arrow[hook', from=1-3, to=1-2]
      \arrow[""{name=3, anchor=center, inner sep=0}, "{\iota_{!}}", shift left=2, from=2-2, to=1-2]
      \arrow["\iota"', from=2-3, to=1-3]
      \arrow[hook', from=2-3, to=2-2]
      \arrow["\dashv"{anchor=center, rotate=-90}, draw=none, from=1, to=0]
      \arrow["\dashv"{anchor=center}, draw=none, from=3, to=2]
    \end{mytikzcd}
  \]
  This realizes the embedding \(\mathbf{EM}(T) \hookrightarrow \fincolimit{\mathbf{Kl}(T)}\) as the composition of two right adjoints,
  and thus a right adjoint.

  Lastly, the composite
  \[
    \mathbf{Kl}(T) \to \mathbf{EM}(T) \to \fincolimit{\mathbf{EM}(T)} \to \mathbf{EM}(T)
  \]
  is naturally isomorphic to \(\iota \from \mathbf{Kl}(T) \hookrightarrow \mathbf{EM}(T)\),
  which proves the latter statement.
\end{proof}

The next result proves the essential uniqueness of the module structure
on the Eilenberg--Moore category,
under the assumption that it is `induced' from the Kleisli category.

\begin{theorem}\label{thm:one-module-structure-on-EM}
  Let \({\mathbf{Kl}(T)}_{0}\) denote the Kleisli category for a monad \(T\) on \(\cat{C}\),
  equipped with a fixed left \(\mathcal{C}\)-module structure.
  Equip \(\mathbf{EM}(T)\) with two left \(\mathcal{C}\)-module category structures%
  ---denoted by \({\mathbf{EM}(T)}_{1}\) and \({\mathbf{EM}(T)}_{2}\)---%
  such that the inclusion
  \(\iota \from \mathbf{Kl}(T) \hookrightarrow \mathbf{EM}(T)\)
  gives strong \(\mathcal{C}\)-module functors
  \({\mathbf{Kl}(T)}_{0} \to {\mathbf{EM}(T)}_{1}\)
  and \({\mathbf{Kl}(T)}_{0} \to {\mathbf{EM}(T)}_{2}\).

  Then there is an equivalence \({\mathbf{EM}(T)}_{1} \simeq {\mathbf{EM}(T)}_{2}\) of left \(\mathcal{C}\)-module categories.
\end{theorem}
\begin{proof}
  Consider the following diagram:
  \[
    \begin{mytikzcd}[ampersand replacement=\&]
      \& {{\mathbf{Kl}(T)}_{0}} \\
      \& {\fincolimit{{\mathbf{Kl}(T)}_{0}}} \\
      {{\mathbf{EM}(T)}_{1}} \&\& {{\mathbf{EM}(T)}_{2}}
      \arrow[from=1-2, to=2-2]
      \arrow["\iota"', curve={height=18pt}, from=1-2, to=3-1]
      \arrow["\iota", curve={height=-18pt}, from=1-2, to=3-3]
      \arrow[""{name=0, anchor=center, inner sep=0}, "{{\iota_{!}}}", shorten <=6pt, shift left=2, from=2-2, to=3-1]
      \arrow[""{name=1, anchor=center, inner sep=0}, "{{\iota_{!}}}"', shorten <=6pt, shift right=2, from=2-2, to=3-3]
      \arrow[""{name=2, anchor=center, inner sep=0}, "{{G_{1}}}", shift left=2, shorten >=6pt, from=3-1, to=2-2]
      \arrow[""{name=3, anchor=center, inner sep=0}, "{{G_{2}}}"', shift right=2, shorten >=6pt, from=3-3, to=2-2]
      \arrow["\dashv"{anchor=center, rotate=62}, draw=none, from=1, to=3]
      \arrow["\dashv"{anchor=center, rotate=118}, draw=none, from=0, to=2]
    \end{mytikzcd}
  \]
  In particular,
  by Proposition~\ref{prop:dual-module-adjunctions},
  in this diagram every functor is a lax \(\mathcal{C}\)-module functor,
  and every adjunction is a \(\mathcal{C}\)-module adjunction.
  We claim that \(\iota_{!}\circ G_{1}\) and \(\iota_{!} \circ G_{2}\) define mutually quasi-inverse equivalences of \(\mathcal{C}\)-module categories.
  Indeed,
  \[
    \iota_{!} \circ G_{2} \circ \iota_{!} \circ G_{1} \xrightarrow{\varepsilon_{2} \hcomp \varepsilon_{1}} \on{Id}_{\mathbf{EM}_{2}}
  \]
  is a \(\mathcal{C}\)-module isomorphism.
  Similarly, there is a \(\mathcal{C}\)-module isomorphism \(\on{Id}_{\mathbf{EM}_{1}} \simeq \iota_{!} \circ G_{1} \circ \iota_{!} \circ G_{2}\).
\end{proof}

\subsubsection{Extendable monads}

We shall now complement the uniqueness result of
Theorem~\ref{thm:one-module-structure-on-EM}
with an existence result.

\begin{definition}\label{extendable}
  We call a lax \(\mathcal{C}\)-module monad \(T\from \mathcal{M} \to \mathcal{M}\) \emph{extendable}
  if the category \(\mathbf{EM}(T)\) is a \(\mathcal{C}\)-module category,
  such that the embedding \(\iota\from \mathbf{Kl}(T) \hookrightarrow \mathbf{EM}(T)\) is a strong \(\mathcal{C}\)-module functor.

  If \(T\) is extendable, we refer to the essentially unique \(\mathcal{C}\)-module structure on \(\mathbf{EM}(T)\) coming from \(\mathbf{Kl}(T)\)
  as the \emph{extended} \(\mathcal{C}\)-module structure on \(\mathbf{EM}(T)\).
\end{definition}

Analogously to Definition~\ref{extendable},
one defines extendable oplax \(\mathcal{C}\)-module comonads
and extended module structures on their Eilenberg--Moore categories.

\begin{proposition}\label{prop:strong-module-monads-are-extendable}
  Any strong \(\cat{C}\)-module monad \((T, \mu, \eta)\from \mathcal{M} \to \mathcal{M}\) is extendable.
\end{proposition}
\begin{proof}
  Since $T$ is in particular a lax $\mathcal{C}$-module monad, by Proposition~\ref{strongcomparisons}, the comparison functor $K_{\mathbf{Kl}}(T)\from \mathbf{Kl}(T) \to \mathcal{M}$ is a strong $\mathcal{C}$-module functor.
  Similarly, since $T$ is in particular an oplax $\mathcal{C}$-module monad, again by Proposition~\ref{strongcomparisons} we conclude that also the comparison functor $K_{\mathbf{EM}}(T)\from \mathcal{M} \to \mathbf{EM}(T)$ is a strong $\mathcal{C}$-module functor.
  Thus, $\iota = K_{\mathbf{EM}}(T) \circ K_{\mathbf{Kl}}(T)$ is a strong $\mathcal{C}$-module functor.
\end{proof}

\begin{corollary}
  If\, \(\mathcal{C}\) is left rigid, then any lax \(\mathcal{C}\)-module monad is extendable.
\end{corollary}

\begin{proof}
  By Proposition~\ref{prop:lax-is-strong-when-rigid}, a lax \(\mathcal{C}\)-module monad is automatically a strong \(\mathcal{C}\)-module monad. The result follows by Proposition~\ref{prop:strong-module-monads-are-extendable}.
\end{proof}

\subsection{Extending module structures via Linton coequalizers}\label{LintinTime}

Our aim is to establish Theorem~\ref{extendingalways} and Theorem~\ref{strongembeddingalways}, which prove extendability of lax $\mathcal{C}$-module monads under the following assumptions.
We state these now and maintain them for the remainder of Section~\ref{sec:module-monads}:
\begin{enumerate}
 \item $\mathcal{C}$ is an abelian monoidal category and $\mathcal{M}$ is an abelian $\mathcal{C}$-module category;
 \item the action functor $\lact \defeq \lact_{\cat{M}} \from \mathcal{C} \otimes \mathcal{M} \to \mathcal{M}$ is right exact in both variables;
 \item $T$ is a right exact lax $\mathcal{C}$-module monad.
\end{enumerate}

\begin{definition}
 For $V \in \mathcal{C}$ and $M \in \mathbf{EM}(T)$, we define the \emph{Linton coequalizer} $V \blacktriangleright M$ as
 \[
   V \blacktriangleright M \defeq \on{coeq}
   \Big(
   \tikzexternaldisable
   \begin{tikzcd}[ampersand replacement=\&]
     {T(V \triangleright_{\mathcal{M}} T(M))} \& {T(V\triangleright_{\mathcal{M}} M)}
     \arrow["{T(V \triangleright_{\mathcal{M}} \nabla_{M})}", shift left=2, from=1-1, to=1-2]
     \arrow["{\mu_{V \triangleright_{\mathcal{M}} M}\circ T({(T_{\mathsf{a}})}_{V,M})}"', shift right=2, from=1-1, to=1-2]
   \end{tikzcd}
   \tikzexternalenable
   \Big)
 \]
\end{definition}

As $\mathbf{EM}(T)$ is abelian, by functoriality of colimits we get a functor $\blank \blacktriangleright \bblank\from \mathcal{C} \kotimes \mathbf{EM}(T) \to \mathbf{EM}(T)$.

\begin{lemma}\label{rightexactaction}
 For $V \in \mathcal{C}$, the functor $V \blacktriangleright \blank$ is right exact.
\end{lemma}
\begin{proof}
 This immediately follows from the right exactness of $T(V \triangleright_{\mathcal{M}} T(\blank))$ and $T(V\triangleright_{\mathcal{M}} \blank)$.
\end{proof}

\begin{lemma}\label{isounitality}
 There is a natural isomorphism $\mathbb{1} \blacktriangleright \blank \cong \on{Id}_{\mathbf{EM}(T)}$.
\end{lemma}
\begin{proof}
  Recall that
  \[\begin{mytikzcd}[ampersand replacement=\&]
      {T^{2}(X)} \& {T(X)} \& X
      \arrow["{\mu_{X}}", shift left=1, from=1-1, to=1-2]
      \arrow["{T(\nabla_{X})}"', shift right=1, from=1-1, to=1-2]
      \arrow["{\nabla_{X}}"', from=1-2, to=1-3]
    \end{mytikzcd}\]
  is a coequalizer in $\mathbf{EM}(T)$, created by the underlying split coequalizer in $\mathcal{M}$. Observe that
  \[\begin{mytikzcd}[ampersand replacement=\&]
      {T(\mathbb{1} \triangleright T(X))} \& {T(\mathbb{1} \triangleright X)} \\
      {T^{2}(X)} \& {T(X)}
      \arrow["{T(\mathbb{1}\triangleright \nabla_{X})}", shift left=1, from=1-1, to=1-2]
      \arrow["{\mu_{\mathbb{1}\triangleright X} \circ T({(T_{\mathsf{a}})}_{\mathbb{1},X})}"', shift right=1, from=1-1, to=1-2]
      \arrow["{T(l_{T(X)})}"', from=1-1, to=2-1]
      \arrow["\simeq", from=1-1, to=2-1]
      \arrow["{T(l_{X})}", from=1-2, to=2-2]
      \arrow["\simeq"', from=1-2, to=2-2]
      \arrow["{T(\nabla_{X})}", shift left=1, from=2-1, to=2-2]
      \arrow["{\mu_{X}}"', shift right=1, from=2-1, to=2-2]
    \end{mytikzcd}\]
  defines an isomorphism in the category of parallel pairs of morphisms in $\mathbf{EM}(T)$---the upper square commutes due to the naturality of $l$, and the lower by coherence for $\blank \triangleright \bblank$.
  This isomorphism induces an isomorphism on the level of coequalizers, $X \cong \mathbb{1} \blacktriangleright X$.
  Further, this construction of morphisms of parallel pairs is clearly natural in $X$, establishing the naturality of the isomorphism of coequalizers.
\end{proof}

\begin{lemma}\label{actiononfree}
 For any $X \in \mathcal{M}$, the coequalizer of
\[\begin{mytikzcd}[ampersand replacement=\&]
	{T(V \triangleright T^{2}(X))} \& {T(V\triangleright T(X))}
	\arrow["{T(V \triangleright \mu_{X})}", shift left=1, from=1-1, to=1-2]
	\arrow["{\mu_{V\triangleright T(X)}\circ T({(T_{\mathsf{a}})}_{V,T(X)})}"', shift right=1, from=1-1, to=1-2]
\end{mytikzcd}\]
in $\mathbf{EM}(T)$ is split and isomorphic to $T(V \triangleright X)$. Thus, $V \blacktriangleright T(X) \simeq T(V \triangleright X)$. This isomorphism is natural in $X$.
\end{lemma}

\begin{proof}
 The splitting data is given by
\[\begin{mytikzcd}[ampersand replacement=\&]
	{T(V \triangleright T^{2}(X))} \& {T(V\triangleright T(X))} \& {T(V\triangleright X)}
	\arrow["{T(V \triangleright \mu_{X})}", from=1-1, to=1-2]
	\arrow["{\mu_{V\triangleright T(X)}\circ T({(T_{\mathsf{a}})}_{V,T(X)})}"', shift right=2, from=1-1, to=1-2]
	\arrow["{T(V \triangleright T(\eta_{X}))}"', curve={height=18pt}, from=1-2, to=1-1]
	\arrow["{\mu_{V \triangleright X}\circ T({(T_{\mathsf{a}})}_{V,X})}"', from=1-2, to=1-3]
	\arrow["{T(V \triangleright \eta_{X})}"', curve={height=18pt}, from=1-3, to=1-2]
\end{mytikzcd}\]
\end{proof}

\begin{lemma}\label{isoassociativity}
 For $V,W \in \mathcal{C}$, there is a natural isomorphism $(V \otimes W) \blacktriangleright \blank \simeq V \blacktriangleright (W \blacktriangleright \blank)$.
\end{lemma}

\begin{proof}
 We have
 \[
   \begin{aligned}
     V \blacktriangleright (W \blacktriangleright X)
     &= V \blacktriangleright \on{coeq}(T(W \triangleright \nabla_{X}),\, \mu_{V \triangleright X} \circ  T{(T_{\mathsf{a}})}_{W,M})\\
     &\simeq \on{coeq}(V \blacktriangleright T(W \lact \nabla_{X}),\, V \blacktriangleright \mu_{V \triangleright X} \circ  T({(T_{\mathsf{a}})}_{W,M}))\\
     &\simeq \on{coeq}(T(V\triangleright (W \triangleright \nabla_{X})),\, \mu_{W \triangleright V \triangleright X}\circ T({(T_{\mathsf{a}})}_{V,W\triangleright X} \circ W\triangleright {(T_{\mathsf{a}})}_{V,X}))\\
     &\simeq \on{coeq}(T((V \otimes W) \triangleright \nabla_{X}),\, \mu_{(V \otimes W)\triangleright T(X)} \circ T({(T_{\mathsf{a}})}_{V \otimes W, T(X)})) \\
     &= (V \otimes W) \blacktriangleright X,
   \end{aligned}
 \]
 where the first isomorphism follows from Lemma~\ref{rightexactaction}, the second from Lemma~\ref{actiononfree} and the third follows by coherence of $\blank \triangleright \blank$.
\end{proof}

In view of Lemma~\ref{isoassociativity} and Lemma~\ref{isounitality}, all that remains in order to establish that $\blank \blacktriangleright \bblank$ defines a $\mathcal{C}$-module category structure on $\mathbf{EM}(T)$ is the verification of coherence axioms. While this can be accomplished by diagram chasing—see~\cite[Theorem~2.6.4]{seal13:tensor}—we choose to give a more formal argument, mimicking the multicategorical approach used by~\cite{aguiar18:monad} in the analogous case of monoidal monads. Since this proof requires very different techniques than the rest of this article, we present it separately in Section~\ref{multicategorical}.

\begin{theorem}\label{extendingalways}
 The functor $\blank \blacktriangleright \bblank$ defines a $\mathcal{C}$-module category structure on $\mathbf{EM}(T)$.
\end{theorem}

Observe that assuming Theorem~\ref{extendingalways}, we find that $V \blacktriangleright T(X) = T(V \triangleright X)$, and so the image of $\iota\from \mathbf{Kl}(T) \to \mathbf{EM}(T)$ is a $\mathcal{C}$-module subcategory, where $\mathbf{Kl}(T)$ is endowed with the action given in the proof of Corollary~\ref{laxKlCorrespondence}.
Thus, we additionally obtain the following important property.

\begin{theorem}\label{strongembeddingalways}
 The functor $\iota\from (\mathbf{Kl}(T), \triangleright_{\mathbf{Kl}(T)}) \to (\mathbf{EM}(T),\blacktriangleright)$ is a strong $\mathcal{C}$-module functor.
\end{theorem}

\subsection{Coherence for Linton coequalizers via multiactegories}\label{multicategorical}

In this section, we give a proof of Theorem~\ref{extendingalways}. Since our approach is a rather minor modification of the proof of an analogous statement for lax \emph{monoidal} monads (rather than lax \emph{module} monads), given in~\cite{aguiar18:monad}, we will only describe the modifications to that proof necessary to establish Theorem~\ref{extendingalways}.

Recall that a (locally small) multicategory $\mathbf{C}$ consists of
\begin{itemize}
 \item a class $\on{Ob}\mathbf{C}$ of objects of $\mathbf{C}$, where we write \(V \in \mathbf{C}\) for \(V \in \on{Ob}\mathbf{C}\);
 \item for any finite sequence $(V_{1},\ldots, V_{n},W)$ of objects in $\mathbf{C}$, a set $\mathbf{C}(V_{1},\ldots, V_{n}, W)$ of \emph{multimorphisms} from $(V_{1},\ldots,V_{n})$ to $W$;
 \item for every $V \in \mathbf{C}$, an identity multimorphism $\on{id}_{V} \in \mathbf{C}(V;V)$; and
 \item for any $W \in \on{Ob}(\mathbf{C})$, ${(V_{i})}_{i=1}^{n} \in {\on{Ob}(\mathbf{C})}^{n}$ and ${({(U_{i}^{j})}_{j=1}^{k_{i}})}_{i=1}^{n} \in {\on{Ob}(\mathbf{C})}^{k_{1}+\cdots + k_{n}}$, a composition operation
 \begin{equation}\label{multicompose}
  \mathbf{C}(V_{1},\ldots,V_{n};W) \times \mathbf{C}(U_{1}^{1},\ldots, U_{1}^{k_{1}};V_{1}) \times \cdots \times \mathbf{C}(U_{n}^{1},\ldots, U_{n}^{k_{n}};V_{n}) \to \mathbf{C}(U_{1}^{1},\ldots,U_{n}^{k_{n}}; W),
 \end{equation}
\end{itemize}
 subject to natural associativity and unitality conditions; see for example~\cite[Section~2.1]{leinster04:higher} or~\cite[Definition~2.1]{hermida00:repres} for a full definition. We remark that, since we want our multicategories and multiactegories to correspond to $\Bbbk$-linear monoidal and module categories, we should replace sets of multimorphisms with vector spaces, and the Cartesian products with tensor products. However, since our aim is to show Theorem~\ref{extendingalways}, which can be verified on the level of the underlying (ordinary) categories, this is not essential to us.

As a special (non-skew) case of~\cite[Definition~4.4]{arkor24}, a (locally small) \emph{left multiactegory $\mathbf{M}$ over a multicategory $\mathbf{C}$} consists of
\begin{itemize}
 \item a class $\on{Ob}\mathbf{M}$ of objects of $\mathbf{M}$, where we write \(M \in \mathbf{M}\) for \(M \in \on{Ob}\mathbf{M}\);
 \item for any finite (possibly empty) sequence $(V_{1},\ldots, V_{n})$ of objects in $\mathbf{C}$ and a pair of objects $(M,N)$ of $\mathbf{M}$, a set $\mathbf{M}(V_{1},\ldots, V_{n};M; N)$ of \emph{multimorphisms} from $(V_{1},\ldots,V_{n};M)$ to $N$;
 \item for every $M \in \mathbf{M}$, an identity multimorphism $\on{id}_{M} \in \mathbf{M}(M;M)$; and
 \item for any $W \in \on{Ob}(\mathbf{C})$, ${(V_{i})}_{i=1}^{n} \in {\on{Ob}(\mathbf{C})}^{n}$ and ${({(U_{i}^{j})}_{j=1}^{k_{i}})}_{i=1}^{n+1} \in {\on{Ob}(\mathbf{C})}^{k_{1}+\cdots + k_{n}}$, a composition operation
 \[
 \begin{aligned}
  &\mathbf{M}(V_{1},\ldots,V_{n};M,N) \times \mathbf{C}(U_{1}^{1},\ldots, U_{1}^{k_{1}};V_{1}) \times \cdots \times \mathbf{C}(U_{n}^{1},\ldots, U_{n}^{k_{n}};V_{n}) \times \mathbf{M}(U_{n+1}^{1},\ldots,U_{n+1}^{k_{n+1}};L;M)\\
  &\to \mathbf{M}(U_{1}^{1},\ldots,U_{n+1}^{k_{n+1}};L; N),
 \end{aligned}
 \]
\end{itemize}
 subject to the (non-marked) associativity and unitality axioms of~\cite[Definition~4.4]{arkor24}.

 From now on, we will often abbreviate sequences of objects using the vector notation similar to~\cite{hermida00:repres};
 e.g., writing $\overrightarrow{V}$ for $(V_{1},\ldots, V_{m})$ and $(\overrightarrow{V}_{1}, \overrightarrow{V}_{2})$ for $(V_{1}^{1},\ldots, V_{1}^{m_{1}}, V_{2}^{1}, \ldots, V_{2}^{m_{2}})$.
 We will use this notation for clarity and brevity, when keeping track of the length of the tuples is not required.
 We also fix the notation for the remainder of Section~\ref{multicategorical}, letting $\mathbf{C}$ be a multicategory and letting $\mathbf{M}$ be a $\mathbf{C}$-multiactegory.

 Recall from~\cite[Chapter~8]{hermida00:repres} that a \emph{pre-universal arrow} for a tuple $\overrightarrow{V}$ of objects of $\mathbf{C}$ consists of an object $\otimes(\overrightarrow{V})$ of $\mathbf{C}$ together with a multimorphism $\pi \in \mathbf{C}(\overrightarrow{V};\otimes(\overrightarrow{V}))$, inducing isomorphisms $\mathbf{C}(\otimes(\overrightarrow{V});W) \xiso \mathbf{C}(\overrightarrow{V};W)$. If $\pi$ further induces isomorphisms
 \[
 \mathbf{C}(\overrightarrow{U}, \otimes(\overrightarrow{V}), \overrightarrow{U}'; W) \xiso \mathbf{C}(\overrightarrow{U}, \overrightarrow{V}, \overrightarrow{U}'; W),
 \]
  we say that $\pi$ is \emph{universal}. If any sequence of objects of $\mathbf{C}$ is the domain of a universal arrow, we say that $\mathbf{C}$ is \emph{representable}. A \emph{representation} of $\mathbf{C}$ is a choice of a universal arrow from every sequence of objects in $\mathbf{C}$.

 \begin{definition}
  Let $\mathbf{M}$ be a $\mathbf{C}$-multiactegory. A \emph{pre-universal arrow} for a tuple $\overrightarrow{V}$ of objects of $\mathbf{C}$ and a single object $X \in \mathbf{M}$, consists of an object $\triangleright(\overrightarrow{V};X)$ together with a multimorphism $\pi \in \mathbf{M}(\overrightarrow{V};X; \triangleright(\overrightarrow{V};X))$, inducing isomorphisms
  \[
   \mathbf{M}(\triangleright(\overrightarrow{V};X);Y) \xiso \mathbf{M}(\overrightarrow{V};X;Y).
  \]
   If $\pi$ further induces isomorphisms
  \[
   \mathbf{M}(\overrightarrow{U},\triangleright(\overrightarrow{V};X);Y) \xiso \mathbf{M}(\overrightarrow{U},\overrightarrow{V};X;Y),
  \]
  then $\pi$ is said to be \emph{universal.}
 \end{definition}

 \begin{definition}

  We say that a multimorphism $\rho \in \mathbf{C}(\overrightarrow{V};W)$ is \emph{pre-universal in $\mathbf{M}$} if it induces isomorphisms
  \[
   \mathbf{M}(W;X;Y) \xiso \mathbf{M}(\overrightarrow{V};X;Y).
  \]
 If it also induces isomorphisms
 \[
  \mathbf{M}(\overrightarrow{U}, \overrightarrow{V}, \overrightarrow{U'}; X; Y) \xiso \mathbf{M}(\overrightarrow{U}, W, \overrightarrow{U'}; X; Y)
 \]
 we say that it is \emph{universal in $\mathbf{M}$}.
 \end{definition}

 \begin{definition}\label{representablemultiactegory}
  For $\mathbf{C}$ representable and a representation $\mathsf{R}$ of $\mathbf{C}$,
  we say that $\mathbf{M}$ is a \emph{representable with respect to $\mathsf{R}$}, if for every sequence $\overrightarrow{V}$ of objects of $\mathbf{C}$ and every object $Y$, there is a universal arrow $\pi_{(\overrightarrow{V};X)} \in \mathbf{M}(\overrightarrow{V};X;\lact(\overrightarrow{V},X))$, and all morphisms of $\mathsf{R}$ are universal in $\mathbf{M}$.
 \end{definition}

 \begin{remark}
  While it is clear that non-universal multimorphisms in $\mathbf{C}$ can be universal in a $\mathbf{C}$-multiactegory, it is not clear to the authors whether universal arrows in $\mathbf{C}$ necessarily act universally in any $\mathbf{C}$-multiactegory, which would make the latter part of Definition~\ref{representablemultiactegory} superfluous.
 \end{remark}

 Recall from~\cite[Proposition~8.5]{hermida00:repres} that universal arrows in a multicategory $\mathbf{C}$ are closed under composition; further, if every sequence of objects in $\mathbf{C}$ is the domain of a pre-universal arrow and pre-universal arrows are closed under composition, then a pre-universal arrow in $\mathbf{C}$ is universal. A similar result holds for multiactegories:

 \begin{lemma}
  Let $\mathbf{M}$ be a left multiactegory over a multicategory $\mathbf{C}$.
  \begin{enumerate}
   \item Composition of universal arrows of $\mathbf{C}$ in $\mathbf{M}$ and universal arrows of $\mathbf{M}$ yields universal arrows in $\mathbf{M}$.
   \item If composition of pre-universal arrows of $\mathbf{C}$ in $\mathbf{M}$ and pre-universal arrows of $\mathbf{M}$ yields pre-universal arrows in $\mathbf{M}$, then a pre-universal morphism in $\mathbf{M}$ is universal.
   \item If pre-universal arrows in $\mathbf{C}$ act pre-universally in $\mathbf{M}$, and pre-universal arrows of $\mathbf{C}$ and $\mathbf{M}$ are closed under composition in $\mathbf{M}$, then $\mathbf{M}$ is representable.
  \end{enumerate}
 \end{lemma}

 \begin{proof}
  For the first part, let $(\pi_{0}, \rho_{1},\ldots, \rho_{n}, \pi_{n+1})$ be a sequence composable in $\mathbf{M}$, with $\pi_{0},\pi_{n+1}$ in $\mathbf{M}$ and $\rho_{1},\ldots, \rho_{n} \in \mathbf{C}$, as indicated in Equation~\ref{multicompose}, all of whose elements are universal in $\mathbf{M}$. Then the composite $\pi_{0} \circ (\rho_{1},\ldots, \rho_{n}, \pi_{n+1})$ gives rise to the sequence of isomorphisms
  \[
  \begin{aligned}
   \mathbf{M}(\overrightarrow{W}, \triangleright(\otimes(\overrightarrow{U}_{1}),\ldots, \otimes(\overrightarrow{U}_{n});\triangleright(\overrightarrow{U}_{n+1},X)),\blank)
   &\simeq \mathbf{M}(\overrightarrow{W}, \otimes(\overrightarrow{U}_{1}), \ldots, \otimes(\overrightarrow{U}_{n});\triangleright(\overrightarrow{U}_{n+1},X);\blank) \\
   &\simeq \mathbf{M}(\overrightarrow{W},\ldots, \overrightarrow{U}_{1},\ldots,\overrightarrow{U}_{n+1};X;\blank),
  \end{aligned}
  \]
 proving the universality of the composite.

 For the second part, consider the following chain of maps:
 \begin{align*}
   \mathbf{M}(\overrightarrow{U},\overrightarrow{V},\overrightarrow{U'};X;-)
   & \rightarrow \mathbf{M}(\overrightarrow{U},\otimes(\overrightarrow{V}),\overrightarrow{U'};X;-)
     \rightarrow \mathbf{M}(\otimes(\overrightarrow{U}),\otimes(\overrightarrow{V}),\otimes(\overrightarrow{U'});X;-) \\
   & \rightarrow \mathbf{M}(\lact (\otimes(\overrightarrow{U}),\otimes(\overrightarrow{V}),\otimes(\overrightarrow{U'});X);-),
 \end{align*}
 where the first morphism is induced by the pre-universal arrow $\rho_{\overrightarrow{V}}$ of $\mathbf{C}$, the latter by the pre-universal arrows $\rho_{\overrightarrow{U}}$ and $\rho_{\overrightarrow{U'}}$ of $\mathbf{C}$, and the third is induced by the universal arrow $\pi_{(\overrightarrow{U},\overrightarrow{V},\overrightarrow{U'}; X)}$. The composite of all three maps is an isomorphism, being induced by a composite of pre-universal arrows, and hence a pre-universal arrow. The same holds for the composite of the latter two maps. These two observations establish the invertibility of the first map.

 The third part is established by applies the proof of the second part above to the case of pre-universal arrows of $\mathbf{C}$.
 \end{proof}

Assume $\mathbf{C}$ is representable, $\mathsf{R}$ is a representation of $\mathbf{C}$, and that $\mathbf{M}$ is representable with respect to $\mathsf{R}$.
Given a sequence $V_{1},\ldots,V_{n}$ of objects in $\mathbf{C}$ and an object $X$ in $\mathbf{M}$, the set of ways in which we may compose the universal arrows in the given representations to a universal arrow with domain $(V_{1},\ldots,V_{n};X)$ is canonically in bijection with the set of parenthesizations of the word $V_{1}\cdots V_{n}X$ into subwords of length at least two. Let $\diamond$ and $\diamond'$ be two such parenthesizations. We denote the codomains of the corresponding universal arrows by $\diamond(V_{1},\ldots,V_{n};X)$ and $\diamond'(V_{1},\ldots,V_{n};X)$, and the universal arrows themselves by $\pi_{\diamond(V_{1},\ldots,V_{n};X)}$ and $\pi_{\diamond'(V_{1},\ldots,V_{n};X)}$. The universality of $\pi_{\diamond(V_{1},\ldots,V_{n};X)}$ entails the existence of a unique arrow $\alpha_{\overrightarrow{V},X}^{\diamond,\diamond'}\from \diamond(V_{1},\ldots,V_{n};X)\to \diamond'(V_{1},\ldots,V_{n};X)$ satisfying $\alpha_{\overrightarrow{V},X}^{\diamond,\diamond'} \circ \pi_{\diamond(\overrightarrow{V},X)} = \pi_{\diamond'(\overrightarrow{V},X)}$, which is invertible, with inverse $\alpha_{\overrightarrow{V},X}^{\diamond',\diamond}$. Further, $\alpha_{\overrightarrow{V},X}^{\diamond,\diamond''} = \alpha_{\overrightarrow{V},X}^{\diamond',\diamond''} \circ \alpha_{\overrightarrow{V},X}^{\diamond,\diamond'}$. Thus, for any sequence $V_{1},\ldots,V_{n},X$ we obtain an indiscrete category whose objects are parenthesizations of $V_{1}\cdots V_{n}X$ and $\on{Hom}(\diamond,\diamond') = \setj{\alpha_{\overrightarrow{V},X}^{\diamond,\diamond'}}$. Further, $\alpha_{\overrightarrow{V},X}^{\diamond,\diamond'}$ is natural in $X$ and $V_{i}$ for all $i$.

Recall that given a representable multicategory $\mathbf{C}$ with a fixed representation, there is a monoidal category $\mathcal{C}$ defined by $\on{Ob}\mathcal{C} = \on{Ob}\mathbf{C}$, $\mathcal{C}(X,Y) = \mathbf{C}(X;Y)$, $X \otimes Y \defeq \otimes(X,Y)$, the codomain of the universal arrow $\pi_{\otimes(X,Y)}$ from $(X,Y)$ in the representation of $\mathbf{C}$, and whose associators and unitors are defined by the morphisms $\alpha_{U,V,W}^{(UV)W,U(VW)}$ described above, and similarly obtained morphisms for unitality. The indiscreteness of the category of morphisms $\alpha^{\diamond,\diamond'}$ associated to the parenthesizations of words in $\mathbf{C}$ entails the commutativity of the pentagon and triangle diagrams, establishing the coherence axioms and making $\mathcal{C}$ well-defined. Similarly, the indiscrete category described in the previous paragraph establishes the well-definedness of the $\mathcal{C}$-module category $\mathcal{M}$ of Definition~\ref{associatedrepresented}.

\begin{definition}\label{associatedrepresented}
 Let $\mathbf{M}$ be a representable multiactegory over a representable multicategory $\mathbf{C}$, with fixed representations for both. Let $\mathcal{C}$ be the monoidal category associated to the representation of $\mathbf{C}$. The $\mathcal{C}$-module category $\mathcal{M}$ associated to the fixed representation of $\mathbf{M}$ is defined by setting $\on{Ob}\mathbf{M} = \on{Ob}\mathcal{M}$, $\mathcal{M}(X,Y) = \mathbf{M}(X;Y)$, $V \triangleright X \defeq \triangleright(V,X)$, the codomain of the universal arrow $\pi_{\triangleright(V,X)}$, and setting the associator $(V \otimes W) \triangleright X \xiso V \triangleright (W \triangleright X)$ to be given by the arrow $\alpha_{V,W,X}^{(VW)X, V(WX)}$, and similarly for the unitors.
\end{definition}

Having established sufficient analogues of the results of~\cite{hermida00:repres,leinster04:higher}, we continue with analogues of~\cite[Section~3]{aguiar18:monad}. Similarly to that setting, we will denote the morphism
\[
  V_{1}\triangleright \cdots \triangleright V_{n} \triangleright T(X) \to T(V_{1}\triangleright \cdots \triangleright V_{n} \triangleright X)
\]
obtained by repeated application of the lax structure of $T$ simply by $\varphi$. We also will consider the coequalizer diagram
\begin{equation}\label{unicharacterize}
\begin{mytikzcd}[ampersand replacement=\&]
	{T(V_{1} \triangleright \cdots \triangleright V_{n} \triangleright T(X))} \&\& {T^{2}(V_{1} \triangleright \cdots \triangleright V_{n} \triangleright X)} \&\& {T(V_{1} \triangleright \cdots \triangleright V_{n} \triangleright X)}
	\arrow["{T\varphi}", from=1-1, to=1-3]
	\arrow["{T(V_{1}\triangleright\cdots \triangleright V_{n}\triangleright \nabla_{X})}", curve={height=-18pt}, from=1-1, to=1-5]
	\arrow["\mu", from=1-3, to=1-5]
\end{mytikzcd}
\end{equation}
which is reflexive, with $T(V_{1} \triangleright \cdots \triangleright V_{n} \triangleright X) \xrightarrow{T(V_{1} \triangleright \cdots \triangleright V_{n} \triangleright \eta_{X})} T(V_{1} \triangleright \cdots \triangleright V_{n} \triangleright T(X))$ a common section.

\begin{definition}
 We say that a morphism $f\from V_{1}\triangleright \cdots \triangleright V_{n} \triangleright X \to W$ is \emph{$n$-multilinear} if the following diagram commutes:
\[\begin{mytikzcd}[ampersand replacement=\&]
	{V_{1}\triangleright \cdots \triangleright V_{n} \triangleright T(X)} \&\& {T(V_{1}\triangleright \cdots \triangleright V_{n}\triangleright X)} \\
	\&\& {T(Y)} \\
	{V_{1}\triangleright \cdots \triangleright V_{n} \triangleright X} \&\& Y
	\arrow["\varphi", from=1-1, to=1-3]
	\arrow["{V_{1}\triangleright \cdots \triangleright V_{n} \triangleright \nabla_{X}}", from=1-1, to=3-1]
	\arrow["{T(f)}"', from=1-3, to=2-3]
	\arrow["{\nabla_{Y}}"', from=2-3, to=3-3]
	\arrow["f", from=3-1, to=3-3]
\end{mytikzcd}\]
\end{definition}

For a monoidal category $\mathcal{C}$, there is a representable multicategory $\mathbf{C}$ such that $\mathbf{C}(V_{1},\ldots,V_{n},W)$ is given by $\mathcal{C}(V_{1}\otimes (\cdots (V_{n-1}\otimes V_{n})\ldots),W)$, the domain being endowed with the rightmost parenthesization. An arrow $\pi \in \mathbf{C}(V_{1},\ldots,V_{n}; W)$ is universal if and only if it is an isomorphism $\pi\from V_{1} \otimes \cdots \otimes V_{n} \xiso W$.
Similarly, given a $\mathcal{C}$-module category, one obtains a representable multiactegory $\mathbf{M}$ over $\mathbf{C}$ satisfying $\mathbf{M}(V_{1},\ldots,V_{n};X;Y) = \mathcal{M}(V_{1}\triangleright \cdots \triangleright V_{n} \triangleright X, Y)$. A universal arrow $\pi \in \mathbf{M}(V_{1},\ldots,V_{n}; X; Y)$ is an isomorphism $\pi\from V_{1} \triangleright \cdots \triangleright X \xiso Y$.

\begin{lemma}
  For a $\mathcal{C}$-module category $\mathcal{M}$, the collection of multilinear morphisms of\, $\mathbf{M}$ is stable under multicomposition with morphisms of\, $\mathbf{C}$,
  forming a $\mathcal{C}$-multiactegory that we denote by $\mathbf{M}^{\varphi}$.
\end{lemma}

\begin{proof}
 Consider sequences $\overrightarrow{U}_{i}$ for $i=1,\ldots, n$, and morphisms $\otimes(U_{i}) \xrightarrow{g_{i}} V_{i}$, where $\otimes(U_{i})$ again is given the rightmost parenthesization. Let $f$ be a multilinear morphism of $\mathbf{M}$.

 The claim follows immediately from the commutativity of
\[\begin{mytikzcd}[ampersand replacement=\&]
	{\otimes(\overrightarrow{U}_{1})\triangleright \cdots \triangleright \otimes(\overrightarrow{U}_{n}) \triangleright T(X)} \&\& {T(\otimes(\overrightarrow{U}_{1})\triangleright \cdots \triangleright \otimes(\overrightarrow{U}_{n})\triangleright X)} \\
	{V_{1}\triangleright \cdots \triangleright V_{n} \triangleright T(X)} \&\& {T(V_{1}\triangleright \cdots \triangleright V_{n}\triangleright X)} \\
	\&\& {T(Y)} \\
	{V_{1}\triangleright \cdots \triangleright V_{n} \triangleright X} \&\& Y
	\arrow["\varphi", from=1-1, to=1-3]
	\arrow["{g_{1}\triangleright\cdots \triangleright g_{n} \triangleright T(X)}", from=1-1, to=2-1]
	\arrow["{T(g_{1}\triangleright\cdots \triangleright g_{n} \triangleright X)}", from=1-3, to=2-3]
	\arrow["\varphi", from=2-1, to=2-3]
	\arrow["{V_{1}\triangleright \cdots \triangleright V_{n} \triangleright \nabla_{X}}", from=2-1, to=4-1]
	\arrow["{T(f)}"', from=2-3, to=3-3]
	\arrow["{\nabla_{Y}}"', from=3-3, to=4-3]
	\arrow["f", from=4-1, to=4-3]
\end{mytikzcd}\]
where the top face commutes by naturality of the action of $\mathcal{C}$ in $\mathcal{M}$.
\end{proof}

\begin{lemma}\label{multiuni}
  A map \(f\) is \(n\)-multilinear if and only if the map \(\overline{f} \defeq \nabla_{Y} \circ T(f)\) coequalizes Diagram~\eqref{unicharacterize}.
  Further, \(f\) is a pre-universal arrow in \(\mathbf{M}^{\varphi}\) if and only if\, \(\overline{f}\) is the coequalizer of~\eqref{unicharacterize}.
\end{lemma}

\begin{proof}
Consider the following diagram:
\[\begin{mytikzcd}[ampersand replacement=\&]
	{T(V_{1}\triangleright \cdots \triangleright V_{n}\triangleright T(X))} \&\& {T^{2}(V_{1}\triangleright \cdots \triangleright V_{n}\triangleright X)} \\
	\&\& {T^{2}(Y)} \\
	{T(V_{1}\triangleright \cdots \triangleright V_{n}\triangleright X)} \&\& {T(Y)} \& {T(Y)} \\
	\&\& Y
	\arrow["{T\varphi}", from=1-1, to=1-3]
	\arrow["{T(V_{1}\triangleright \cdots\triangleright \nabla_{X})}"', from=1-1, to=3-1]
	\arrow["{T^{2}(f)}", from=1-3, to=2-3]
	\arrow[""{name=0, anchor=center, inner sep=0}, "\mu", from=1-3, to=3-1]
	\arrow["{\mu_{Y}}", from=2-3, to=3-3]
	\arrow[""{name=1, anchor=center, inner sep=0}, "{T(\nabla_{Y})}", from=2-3, to=3-4]
	\arrow["{T(f)}", from=3-1, to=3-3]
	\arrow["{\overline{f}}"', from=3-1, to=4-3]
	\arrow["{\nabla_{Y}}", from=3-3, to=4-3]
	\arrow["{\nabla_{Y}}", from=3-4, to=4-3]
	\arrow["{(1)}"{description}, draw=none, from=1-1, to=0]
	\arrow["{(2)}"{description}, shift left=5, draw=none, from=0, to=3-3]
	\arrow["{(3)}"{description, pos=0.4}, draw=none, from=1, to=4-3]
\end{mytikzcd}\]
Coequalizing the pair~\eqref{unicharacterize} is precisely coequalizing face $(1)$ in the above diagram. Thus, since face $(2)$ commutes, it is equivalent to $\nabla_{Y}$ coequalizing the square formed jointly by faces $(1)$ and $(2)$. Observe that, using commutativity of face $(3)$, this square postcomposed by $\nabla_{Y}$ gives precisely the pair of morphisms in $\mathbf{EM}(T)(T(V_{1}\triangleright \cdots \triangleright T(X)), Y)$ corresponding to the pair of morphisms in $\mathcal{M}(V_{1}\triangleright \cdots \triangleright T(X), Y)$, defining multilinearity of $f$, under $\mathbf{EM}(T)(T(\blank),\bblank) \cong \mathcal{M}(\blank,\bblank)$. Thus, both conditions in the first statement are equivalent to the commutativity of the diagram. The second statement is shown analogously to~\cite[Lemma~3.5]{aguiar18:monad}
\end{proof}

Finally, in order to establish Theorem~\ref{extendingalways} and Theorem~\ref{strongembeddingalways}, we need an analogue of~\cite[Proposition~3.13]{aguiar18:monad}:
\begin{proposition}\label{hillbility}
 The multicategory $\mathbf{M}^{\varphi}$ is representable.
\end{proposition}

\begin{proof}
  Let $(V_{1},\ldots,V_{n},Y)$ be a sequence of objects, with $V_{i} \in \mathcal{C}$ and $Y \in \mathcal{M}$.
  The coequalizer of Diagram~\eqref{unicharacterize} defines a pre-universal map in $\mathbf{M}^{\varphi}$ with domain $(V_{1},\ldots,V_{n};Y)$, by Lemma~\ref{multiuni}.
  It remains to show that a composition of pre-universal maps is pre-universal.
  Thus, consider sequences $\overrightarrow{U}_{1},\ldots,\overrightarrow{U}_{n}$ and $\overrightarrow{W}$, universal multimorphisms $u_{i}\from \overrightarrow{U}_{i} \to V_{i}$ of $\mathbf{C}$, and universal multimorphisms $x\from (\overrightarrow{W},X) \to Y$ and $y\from (\overrightarrow{V},Y) \to Z$.

  Behold the following diagram:
  \begin{equation}\label{BigLinton}
    \scalebox{0.9}{\begin{mytikzcd}[ampersand replacement=\&]
        \&\& {\overrightarrow{U}_{1} \triangleright \cdots \overrightarrow{U}_{n}\triangleright T^{2}(\overrightarrow{W} \triangleright X)} \\
        {\overrightarrow{U}_{1} \triangleright \cdots \overrightarrow{U}_{n}\triangleright T(\overrightarrow{W} \triangleright T(X))} \&\&\&\& {\overrightarrow{U}_{1} \triangleright \cdots \overrightarrow{U}_{n}\triangleright T(\overrightarrow{W} \triangleright X)} \\
        \&\& {T^{2}(\overrightarrow{U}_{1} \triangleright \cdots \overrightarrow{U}_{n}\triangleright \overrightarrow{W} \triangleright X)} \\
        {T(\overrightarrow{U}_{1} \triangleright \cdots \overrightarrow{U}_{n}\triangleright \overrightarrow{W} \triangleright T(X))} \&\&\&\& {T(\overrightarrow{U}_{1} \triangleright \cdots \overrightarrow{U}_{n}\triangleright \overrightarrow{W} \triangleright X)}
        \arrow["{U_{1} \triangleright \cdots U_{n}\triangleright \mu_{W\triangleright X}}"{pos=0.6}, from=1-3, to=2-5]
        \arrow["{\overrightarrow{U}_{1} \triangleright \cdots \overrightarrow{U}_{n}\triangleright T(\varphi)}"{pos=0.4}, from=2-1, to=1-3]
        \arrow["{\overrightarrow{U}_{1} \triangleright \cdots \overrightarrow{U}_{n}\triangleright T(\overrightarrow{W} \triangleright \nabla_{X})}", from=2-1, to=2-5]
        \arrow["\varphi"', from=2-1, to=4-1]
        \arrow["\varphi"', from=2-5, to=4-5]
        \arrow["{\mu_{U_{1}\triangleright \cdots \triangleright X}}", from=3-3, to=4-5]
        \arrow["{T\varphi}", from=4-1, to=3-3]
        \arrow["{T(\overrightarrow{U}_{1} \triangleright \cdots \overrightarrow{U}_{n}\triangleright \overrightarrow{W} \triangleright \nabla_{X})}"', from=4-1, to=4-5]
      \end{mytikzcd}}
  \end{equation}
  Similarly to the proof of~\cite[Proposition~3.13]{aguiar18:monad}, this is a morphism from the top triangle-shaped diagram to the bottom one, where the top is the action of $(\overrightarrow{U_{1}},\ldots,\overrightarrow{U_{n}})$ on the Linton pair for $(\overrightarrow{W},X)$, while the bottom is the Linton pair for $(\overrightarrow{U_{1}},\ldots,\overrightarrow{U_{n}},\overrightarrow{W};X)$.

  Now, let $g \in \mathbf{M}^{\varphi}(\overrightarrow{U}_{1},\ldots,\overrightarrow{U}_{n},\overrightarrow{W};X;M)$.
  Consider the diagram
\[\begin{mytikzcd}[ampersand replacement=\&]
	{\overrightarrow{U}_{1} \triangleright \cdots \overrightarrow{U}_{n}\triangleright T(\overrightarrow{W} \triangleright X)} \&\& {\overrightarrow{U}_{1} \triangleright \cdots \overrightarrow{U}_{n}\triangleright Y} \&\& {V_{1} \triangleright \cdots V_{n}\triangleright Y} \\
	\\
	{T(\overrightarrow{U}_{1} \triangleright \cdots \overrightarrow{U}_{n}\triangleright \overrightarrow{W} \triangleright X)} \&\& { M} \&\& Z
	\arrow["{U_{1}\triangleright\cdots \triangleright \overline{x}}", from=1-1, to=1-3]
	\arrow["\varphi"', from=1-1, to=3-1]
	\arrow["{u_{1} \triangleright \cdots u_{n}\triangleright Y}", from=1-3, to=1-5]
	\arrow["\simeq"', from=1-3, to=1-5]
	\arrow["{\exists! \widetilde{g}}", dashed, from=1-3, to=3-3]
	\arrow["y", from=1-5, to=3-5]
	\arrow["{\overline{g}}", from=3-1, to=3-3]
	\arrow["{\exists! \widehat{g}}", dotted, from=3-5, to=3-3]
\end{mytikzcd}\]
  where $\widetilde{g}$ exists since, by Lemma~\ref{multiuni}, $\overline{g}$ coequalizes the bottom Linton pair in Diagram~\ref{BigLinton} above, and ${u_{1}\triangleright\cdots \triangleright \overline{x}}$ is the coequalizer of the top Linton pair in the same diagram. Similarly to~\cite[Lemma~3.11]{aguiar18:monad} one shows that $\widetilde{g}$ is multilinear, and thus so is $u_{1}^{-1} \triangleright \cdots \triangleright u_{n}^{-1} \triangleright Y$, since $u_{i}$ is invertible for all $i$, being a pre-universal, and hence universal, arrow of $\mathbf{C}$. This proves the existence of $\widehat{g}$. Thus we may factorize the given multilinear morphism $g$ through $y \circ (u_{1},\ldots, u_{n},x)$, and this factorization is unique by uniqueness of $\widetilde{g}$ and of $\widehat{g}$, which establishes the pre-universality of the composition $y \circ (u_{1},\ldots, u_{n},x)$.
\end{proof}

\subsection{Adjoint module monads}

\begin{proposition}\label{EMMonadCorrespondence}
  Let \(\mathcal{M}\) be a \(\mathcal{C}\)-module category,
  \(T\from \mathcal{M} \to \mathcal{M}\) an oplax \(\mathcal{C}\)-module monad,
  and \(G\from \mathcal{M} \to \mathcal{M}\) a right adjoint to \(T\).
  Then the isomorphism \(\mathbf{EM}(T) \cong \mathbf{EM}(G)\) of Proposition~\ref{pyramidscheme} is a strict \(\mathcal{C}\)-module isomorphism,
  where \(\mathbf{EM}(T)\) is endowed with the \(\mathcal{C}\)-module structure of Proposition~\ref{oplaxEMcorrespondence},
  and similarly for \(\mathbf{EM}(G)\).
\end{proposition}
\begin{proof}
  Note that,
  by Proposition~\ref{prop:module-adjunctions},
  the comonad \(G\from \cat{M}\to \cat{M}\) comes equipped with the following lax \(\cat{C}\)-module structure,
  for all \(V \in \cat{C}\) and \(M \in \cat{M}\):
  \[
    {(G_{\mathsf{a}})}_{V, M} \from
    V \lact G(M)
    \xrightarrow{\eta_{V \lact G(M)}} GT(V \lact G(M))
    \xrightarrow{G{(T_{\mathsf{a}})}_{V, G(M)}} G(V \lact TG(M))
    \xrightarrow{G(V \lact \varepsilon_M)} G(V \lact M).
  \]
  Further, the isomorphism \(L\from \mathbf{EM}(T) \xiso \mathbf{EM}(G)\) of Proposition~\ref{pyramidscheme} is given by
  \[
    \big(A,\ \nabla_A\from T(A) \to A\big) \mapsto \big(A,\ A \xrightarrow{\ \eta_A\ } GT(A) \xrightarrow{\,G\nabla_A\,} G(A)\big).
  \]
  We have to prove that,
  for all \(V \in \cat{C}\) and \((A, \nabla_A) \in \mathbf{EM}(T)\),
  one has \(L(V \lact (A, \nabla_A)) = V \lact L((A, \nabla_A))\).
  This is equivalent to the equality of the following two morphisms:
  \[
    \big(
    V \lact A,\
    V \lact A
    \xrightarrow{\eta_{V \lact A}} GT(V \lact A)
    \xrightarrow{G {(T_{\mathsf{a}})}_{V, A}} G(V \lact T(A))
    \xrightarrow{G(V \lact \nabla_A)} G(V \lact A)
    \big);
  \]
  \begin{align*}
    \big(
    V \lact A,\
    V \lact A
    &\xrightarrow{V \lact \eta_A} V \lact GT(A)
      \xrightarrow{V \lact G \nabla_A} V \lact G(A)
      \xrightarrow{\eta_{V \lact G(A)}} GT(V \lact G(A)) \\
    &\xrightarrow{G{(T_{\mathsf{a}})}_{V,G(A)}} G(V \lact TG(A))
      \xrightarrow{G(V \lact \varepsilon_A)} G(V \lact A)
      \big).
  \end{align*}
  This is evidenced by the following commutative diagram:
  \[
    \begin{mytikzcd}[ampersand replacement=\&]
      {V \lact A} \& {GT(V\lact A)} \&\& {G(V\lact T(A))} \\
      {V\lact GT(A)} \& {GT(V \lact GT(A))} \& {G(V \lact TGT(A))} \& {G(V\lact T(A))} \\
      {V\lact G(A)} \& {GT(V \lact G(A))} \& {G(V\lact TG(A))} \& {G(V \lact A)}
      \arrow["{{\eta_{V\lact A}}}", from=1-1, to=1-2]
      \arrow["{{V\lact \eta_A}}"', from=1-1, to=2-1]
      \arrow["{{G {(T_{\mathsf{a}})}_{V,A}}}", from=1-2, to=1-4]
      \arrow["{{GT(V \lact \eta_A)}}", from=1-2, to=2-2]
      \arrow[""{name=0, anchor=center, inner sep=0}, "{{G(V \lact T\eta_A)}}"', from=1-4, to=2-3]
      \arrow[Rightarrow, no head, from=1-4, to=2-4]
      \arrow["{{\eta_{V \lact GT(A)}}}"', from=2-1, to=2-2]
      \arrow["{{V \lact G \nabla_A}}"', from=2-1, to=3-1]
      \arrow["{{GT_{\mathsf{a}}}}"', from=2-2, to=2-3]
      \arrow["{{GT(V \lact G \nabla_A)}}", from=2-2, to=3-2]
      \arrow["{{G(V \lact \varepsilon_{T(A)})}}"', from=2-3, to=2-4]
      \arrow["{{G(V \lact TG \nabla_A)}}", from=2-3, to=3-3]
      \arrow["{{G(V \lact \nabla_A)}}", from=2-4, to=3-4]
      \arrow["{{\eta_{V \lact G(A)}}}"', from=3-1, to=3-2]
      \arrow["{{G{(T_{\mathsf{a}})}_{V,G(A)}}}"', from=3-2, to=3-3]
      \arrow["{{G(V \lact \varepsilon_{A})}}"', from=3-3, to=3-4]
      \arrow["{\mathsf{adj}}"{description}, draw=none, from=0, to=2-4]
    \end{mytikzcd}
  \]
\end{proof}

In an analogous way to Proposition~\ref{EMMonadCorrespondence},
one obtains the following result.

\begin{proposition}\label{KleisliMonadCorrespondence}
  Let \(\mathcal{M}\) be a \(\mathcal{C}\)-module category,
  \(T\from \mathcal{M}\to \mathcal{M}\) a lax \(\mathcal{C}\)-module monad,
  and \(K\from \mathcal{M} \to \mathcal{M}\) a left adjoint to \(T\).
  Then the equivalence \(\mathbf{Kl}(T) \simeq \mathbf{Kl}(G)\) of Proposition~\ref{einerKleinerKleisliÄquivalenz} is a \(\mathcal{C}\)-module equivalence,
  where \(\mathbf{Kl}(T)\) is endowed with the \(\mathcal{C}\)-module structure of Proposition~\ref{laxKlCorrespondence},
  and similarly for \(\mathbf{Kl}(K)\).
\end{proposition}

\subsection{Beck's monadicity theorems}

We recall Beck's monadicity theorem,
as well as specialisations of it to abelian categories,
which will be useful in our further investigations into the module structure of the Eilenberg--Moore category of a (co)monad.

\begin{theorem}[Beck's monadicity theorem]\label{BigBeck}
  Let \(\stdadj\) be an adjunction.
  Then the functor \(\Omega\from \mathcal{B} \to \mathbf{EM}(UF)\) is an equivalence if and only if\, \(U\) creates coequalizers of\, \(U\)-split pairs.
\end{theorem}

We refer to~\cite[\S~VI.7]{MacL} for a definition of the notion of \(U\)-split pairs.
Theorem~\ref{BigBeck} is a very precise result---we are essentially only interested in sufficient conditions for the comparison functor to be an equivalence.
Thus, we state a simpler set of sufficient conditions.

We say that a functor $U$ \emph{reflects zero objects} if $U(X) \cong 0$ implies that $X \cong 0$.

\begin{theorem}\label{thm:SmallBeck}
  Let \(\adj{F}{U}{\cat{A}}{\cat{B}} \) be an adjunction between abelian categories.
  If\, \(U\) is right exact and it reflects zero objects,
  then \(K \from \cat{B} \to \mathbf{EM}(UF)\) is an equivalence.
  If\, \(F\) is left exact and it reflects zero objects,
  then \(K\from \cat{A} \to \mathbf{EM}(FU)\) is an equivalence.
\end{theorem}

That Theorem~\ref{thm:SmallBeck} follows from Beck's monadicity theorem is explained in~\cite[Section~4.1]{BZBJ}.

\subsection{Semisimple monads}\label{sec:semisimple-monads}

\begin{definition}\label{def:cauchy-completion}
  The \emph{Cauchy completion} \(\cat{C}^{\mathrm{Cy}}\) of a (small) category \(\cat{C}\)
  is the full subcategory of \([\cat{C}^{\op}, \kVect]\) comprising all \emph{tiny objects}%
  —functors \(F\from \cat{C}^{\op} \to \kVect\) such that \([\cat{C}^{\op}, \kVect](F, \blank)\) preserves small colimits.
\end{definition}

There are many equivalent formulations of Definition~\ref{def:cauchy-completion} in the case of\, \(\Bbbk\)-linear categories.
For example, a functor is in the Cauchy completion
if and only if it is a retract of a representable functor,
if and only if it is a direct summand of a direct sum of representable functors;
see~\cite{BoDe,prest09:purit}.
Further, this coincides with the representation theoretic definition of the Cauchy completion as the Karoubi envelope of the additive envelope of a category,
see~\cite[Corollary~4.22]{lack22:flat}.

\begin{lemma}
  The Eilenberg--Moore category of a monad on a Cauchy complete category is Cauchy complete.
\end{lemma}
\begin{proof}
  For a monad \(T\),
  this follows immediately from the fact that the forgetful functor
  \(R_{\mathbf{EM}(T)}\) creates limits in \(\mathbf{EM}(T)\).
\end{proof}

\begin{proposition}\label{prop:SemisimpleEM}
  Let \(T\from \mathcal{A} \to \mathcal{A}\) be a monad on an abelian category \(\mathcal{A}\).
  The category \(\mathbf{EM}(T)\) is semisimple if and only if\,
  \({\mathbf{Kl}(T)}^{\mathrm{Cy}}\) is so.
  In that case, the extension \(\iota^{\mathrm{Cy}}\from {\mathbf{Kl}(T)}^{\mathrm{Cy}} \to \mathbf{EM}(T)\) of the canonical embedding \(\iota\from \mathbf{Kl}(T) \to \mathbf{EM}(T)\)
  to the Cauchy completion is an equivalence of categories.
\end{proposition}
\begin{proof}
  Since \(\mathcal{A}\) is Cauchy complete, the adjunction
  \[
    \begin{mytikzcd}[ampersand replacement=\&,sep=small]
      {\mathcal{A}} \&\& {\mathbf{Kl}(T)}
      \arrow[""{name=0, anchor=center, inner sep=0}, "{L_{\mathbf{Kl}(T)}}", shift left=2, from=1-1, to=1-3]
      \arrow[""{name=1, anchor=center, inner sep=0}, "{R_{\mathbf{Kl}(T)}}", shift left, from=1-3, to=1-1]
      \arrow["\dashv"{anchor=center, rotate=-90}, draw=none, from=0, to=1]
    \end{mytikzcd}
  \]
  extends to an adjunction
  \[
    \begin{mytikzcd}[ampersand replacement=\&,sep=small]
      {\mathcal{A}} \&\& {{\mathbf{Kl}(T)}^{\mathrm{Cy}},}
      \arrow[""{name=0, anchor=center, inner sep=0}, "{L_{\mathbf{Kl}(T)}^{\mathrm{Cy}}}", shift left=2, from=1-1, to=1-3]
      \arrow[""{name=1, anchor=center, inner sep=0}, "{R_{\mathbf{Kl}(T)}^{\mathrm{Cy}}}", shift left, from=1-3, to=1-1]
      \arrow["\dashv"{anchor=center, rotate=-90}, draw=none, from=0, to=1]
    \end{mytikzcd}
  \]
  such that \(R_{\mathbf{Kl}(T)}^{\mathrm{Cy}} \circ L_{\mathbf{Kl}(T)}^{\mathrm{Cy}} = T\).

  It is easy to see that the resulting comparison functor \(K_{{\mathbf{Kl}(T)}^{\mathrm{Cy}}}\from {\mathbf{Kl}(T)}^{\mathrm{Cy}} \to \mathbf{EM}(T)\) is naturally isomorphic to the extension \(\iota^{\mathrm{Cy}}\) of \(\iota\).
  Further, since \(\iota\) is full and faithful, so is \(\iota^{\mathrm{Cy}}\).
  Thus, \(\iota^{\mathrm{Cy}}\) reflects zero objects, and hence so does \(R_{\mathbf{Kl}(T)} = R_{\mathbf{EM}(T)} \circ \iota\), and by extension so does \(R_{\mathbf{Kl}(T)}^{\mathrm{Cy}}\).

  Assume that \(\mathbf{EM}(T)\) is semisimple.
  The monomorphisms and epimorphisms in \(\mathbf{EM}(T)\) are split by Proposition~\ref{toomuchTV}, and thus they are reflected under the full, faithful functor \(\iota^{\mathrm{Cy}}\).
  Thus, \({\mathbf{Kl}(T)}^{\mathrm{Cy}}\) is an abelian category, and \(\iota^{\mathrm{Cy}}\) is an exact functor.
  In particular \(R_{\mathbf{Kl}(T)}^{\mathrm{Cy}} \simeq R_{\mathbf{EM}(T)} \circ \iota^{\mathrm{Cy}}\) is exact and faithful; thus, it reflects zero objects.
  Thus, by Theorem~\ref{thm:SmallBeck}, \(K_{{\mathbf{Kl}(T)}^{\mathrm{Cy}}}\) is an equivalence---%
  in particular, \({\mathbf{Kl}(T)}^{\mathrm{Cy}}\) is semisimple.

  Assume now that \({\mathbf{Kl}(T)}^{\mathrm{Cy}}\) is semisimple.
  Then \(R_{\mathbf{Kl}(T)}^{\mathrm{Cy}}\) is exact by Proposition~\ref{toomuchTV},
  and so it again satisfies the assumptions of Theorem~\ref{thm:SmallBeck}.
  Thus, \(K_{{\mathbf{Kl}(T)}^{\mathrm{Cy}}}\from {\mathbf{Kl}(T)}^{\mathrm{Cy}} \xiso \mathbf{EM}(T)\) is an equivalence;
  in particular, $\mathbf{EM}(T)$ is semisimple.
\end{proof}

\begin{definition}\label{def:semisimple-monad}
  We say that a monad \(T\from \mathcal{A} \to \mathcal{A}\) on an abelian category \(\mathcal{A}\) is \emph{semisimple}
  if it satisfies the equivalent conditions of Proposition~\ref{prop:SemisimpleEM}. Since the opposite of an abelian category is abelian and the opposite of a semisimple category is semisimple, an analogous result to Proposition~\ref{prop:SemisimpleEM} holds for comonads, and we define semisimple comonads analogously.
\end{definition}

\begin{proposition}\label{extendingsemisimple}
  Let \(T\from \mathcal{M} \to \mathcal{M}\) be a lax \(\mathcal{C}\)-module semisimple monad on an abelian \(\mathcal{C}\)-module category.
  Then \(T\) is extendable.
\end{proposition}
\begin{proof}
  By Proposition~\ref{prop:SemisimpleEM}, the functor \(\iota^{\mathrm{Cy}}\from {\mathbf{Kl}(T)}^{\mathrm{Cy}} \to \mathbf{EM}(T)\) is an equivalence.
  Further, following Proposition~\ref{cocompletingnonsense}, \({\mathbf{Kl}(T)}^{\mathrm{Cy}}\) has a canonical structure of a \(\mathcal{C}\)-module category such that the inclusion \(\mathbf{Kl}(T) \hookrightarrow {\mathbf{Kl}(T)}^{\mathrm{Cy}}\) is a strong \(\mathcal{C}\)-module functor.
  Transporting the \(\mathcal{C}\)-module structure along \(\iota^{\mathrm{Cy}}\) endows \(\iota\) with the structure of a strong \(\mathcal{C}\)-module functor, proving the claim.
\end{proof}

\begin{lemma}\label{semisimpleleftandright}
  Let \(T\from \mathcal{A} \to \mathcal{A}\) be a monad on an abelian category, and let \(K\from \mathcal{M} \to \mathcal{M}\) be a left adjoint to \(T\).
  Then \(T\) is semisimple if and only if \(K\) is semisimple.
  In that case, there is an equivalence \(\mathbf{EM}(T) \simeq \mathbf{EM}(K)\).
\end{lemma}
\begin{proof}
  By Proposition~\ref{prop:SemisimpleEM}, \(T\) is semisimple if and only if so is \({\mathbf{Kl}(T)}^{\mathrm{Cy}}\).
  Moreover, the equivalence \(\mathbf{Kl}(T) \simeq \mathbf{Kl}(K)\) of Proposition~\ref{einerKleinerKleisliÄquivalenz} extends to an equivalence \({\mathbf{Kl}(T)}^{\mathrm{Cy}} \simeq {\mathbf{Kl}(K)}^{\mathrm{Cy}}\).
  Thus, \({\mathbf{Kl}(T)}^{\mathrm{Cy}}\) is semisimple if and only if \({\mathbf{Kl}(K)}^{\mathrm{Cy}}\) is so,
  which is the case if and only if \(K\) is semisimple.
  This establishes the first claim.

  The latter claim follows from the equivalences
  \[
    \mathbf{EM}(T) \simeq {\mathbf{Kl}(T)}^{\mathrm{Cy}} \simeq {\mathbf{Kl}(K)}^{\mathrm{Cy}} \simeq \mathbf{EM}(K).
  \]
\end{proof}

\begin{proposition}\label{ModuleEMSemisimple}
  Let \(T\from \mathcal{M} \to \mathcal{M}\) be a semisimple lax \(\mathcal{C}\)-module monad on an abelian \(\mathcal{C}\)-module category \(\mathcal{M}\),
  and let \(K\from \mathcal{M} \to \mathcal{M}\) be a left adjoint to \(T\).
  The equivalence \(\mathbf{EM}(T) \simeq \mathbf{EM}(K)\) of Lemma~\ref{semisimpleleftandright} is a \(\mathcal{C}\)-module equivalence,
  where \(\mathbf{EM}(T), \mathbf{EM}(K)\) are endowed with \(\mathcal{C}\)-module structures extended from \(\mathbf{Kl}(T),\mathbf{Kl}(K)\), respectively.
\end{proposition}
\begin{proof}
  Following the proof of Proposition~\ref{extendingsemisimple},
  the functor \(\iota^{\mathrm{Cy}}\from {\mathbf{Kl}(T)}^{\mathrm{Cy}} \xiso \mathbf{EM}(T)\) is a \(\mathcal{C}\)-module equivalence,
  and similarly for \(\mathbf{EM}(K)\) and \({\mathbf{Kl}(K)}^{\mathrm{Cy}}\).
  Further, from Proposition~\ref{cocompletingnonsense} we obtain a \(\mathcal{C}\)-module equivalence
  \({\mathbf{Kl}(T)}^{\mathrm{Cy}} \simeq {\mathbf{Kl}(K)}^{\mathrm{Cy}}\)
  from the \(\mathcal{C}\)-module equivalence \(\mathbf{Kl}(T) \simeq \mathbf{Kl}(K)\) of Proposition~\ref{KleisliMonadCorrespondence}.

  The result follows by composing the above-established \(\mathcal{C}\)-module equivalences as follows:
  \[
    \mathbf{EM}(T) \simeq {\mathbf{Kl}(T)}^{\mathrm{Cy}} \simeq {\mathbf{Kl}(K)}^{\mathrm{Cy}} \simeq \mathbf{EM}(K).
  \]
\end{proof}

\section{Internal projective and injective objects}\label{sec:internal-projectives-and-injectives}

\begin{definition}
  Let $\mathcal{C}$ be an abelian monoidal category and let $\mathcal{M}$ be an abelian $\mathcal{C}$-module category. An object $M \in \mathcal{M}$ is said to be \emph{$\mathcal{C}$-projective} if, for any projective object $P \in \mathcal{C}$, the object $P \triangleright M \in \mathcal{M}$ is projective.

  Analogously, $M$ is \emph{$\mathcal{C}$-injective} if, for any injective object $I \in \mathcal{C}$, the object $I \triangleright M \in \mathcal{M}$ is injective.
\end{definition}

\begin{definition}
  Let $\mathcal{C}$ be an abelian monoidal category with enough projectives and let $\mathcal{M}$ be an abelian $\mathcal{C}$-module category with enough projectives. We say that a \(\mathcal{C}\)-projective object $X$ is a \emph{$\mathcal{C}$-projective $\mathcal{C}$-generator} if any projective object $Q$ of $\mathcal{M}$ is a direct summand of an object of the form $P \triangleright X$, for a projective object $P$ of $\mathcal{C}$.

  Analogously, if\, $\mathcal{C}$ and \(\cat{M}\) are instead assumed to have enough injectives, we say that a $\mathcal{C}$-injective object $X$ is a \emph{\(\mathcal{C}\)-injective \(\mathcal{C}\)-cogenerator} if any injective object $J$ of $\mathcal{M}$ is a direct summand of an object of the form $I \triangleright X$, for an injective object $I$ of $\mathcal{C}$.
\end{definition}

\begin{proposition}\label{Cprojectives}
  Assume that $\mathcal{C}$ and \(\cat{M}\) have enough projectives. Let $X$ be a closed object of $\mathcal{M}$.
  Then $\hom{X,\blank}$ is right exact if and only if\, $X$ is $\mathcal{C}$-projective.
\end{proposition}
\begin{proof}
  If $\hom{X, \blank}$ is right exact and $P$ is a projective object of\, $\mathcal{C}$,
  then $\mathcal{M}(P \lact X, \blank) \cong \mathcal{C}(P,\hom{X, \blank})$,
  which is a composite of right exact functors, hence a right exact functor;
  thus, $P \lact X$ is projective.

  Assume that $X$ is $\mathcal{C}$-projective.
  Let \(\colim_{j} Y_{j}\) be a finite colimit in \(\mathcal{M}\),
  and $V \in \mathcal{C}$.
  Since $\mathcal{C}$ has enough projectives, one has $V \cong \colim_{i} Q_{i}$, realizing $V$ as the colimit of a finite diagram of projectives.
  We then have
  \begin{align*}
    \mathcal{C}(V,\, &\hom{X,\colim\displaylimits_{j}Y_{j}})
          \cong \mathcal{C}(\colim\displaylimits_{i}Q_{i},\hom{X,\colim\displaylimits_{j}Y_{j}})
          \cong \lim_{i}\mathcal{C}(Q_{i},\hom{X,\colim\displaylimits_{j}Y_{j}}) \\
        & \cong \lim_{i} \mathcal{M}(Q_{i} \triangleright X, \colim\displaylimits_{j} Y_{j})
          \cong \lim_{i}\colim\displaylimits_{j} \mathcal{M}(Q_{i} \triangleright X, Y_{j})
          \cong \lim_{i}\colim\displaylimits_{j} \mathcal{C}(Q_{i}, \hom{X,Y_{j}}) \\
        & \cong \lim_{i} \mathcal{C}(Q_{i}, \colim\displaylimits_{j} \hom{X,Y_{j}})
          \cong \mathcal{C}(\colim\displaylimits_{i} Q_{i}, \colim\displaylimits_{j} \hom{X,Y_{j}})
          \cong \mathcal{C}(V,\colim\displaylimits_{j} \hom{X,Y_{j}}).
  \end{align*}
\end{proof}

By oppositizing Proposition~\ref{Cprojectives}, we obtain the following result.
\begin{proposition}\label{Cinjectives}
  Assume that $\mathcal{C}$ and \(\cat{M}\) have enough injectives.
  Let $X$ be a coclosed object of $\mathcal{M}$. Then $\cohom{X,\blank}$ is left exact if and only if\, $X$ is $\mathcal{C}$-injective.
\end{proposition}

\begin{proposition}\label{ProjReflect}
  Assume that $\mathcal{C}$ and \(\cat{M}\) have enough projectives.
  Let $X$ be a closed $\mathcal{C}$-projective object of $\mathcal{M}$.
  If\, $X$ is a $\mathcal{C}$-generator, then $\hom{X, \blank}$ reflects zero objects.

  Additionally, if\, $\proj{\mathcal{C}}$ is Krull--Schmidt and every indecomposable projective object of\, $\mathcal{C}$ is the projective cover of a simple object, then $\hom{X, \blank}$ reflecting zero objects implies that $X$ is a $\mathcal{C}$-generator.
\end{proposition}
\begin{proof}
  Assume that $X$ is a $\mathcal{C}$-generator. Let $N$ be a non-zero object of $\mathcal{M}$ and let $Q \twoheadrightarrow N$ be an epimorphism from a projective object in $\mathcal{M}$. Let $P$ be a projective object of\, $\mathcal{C}$ such that $Q$ is a direct summand of $P \triangleright X$. Then $\mathcal{M}(P \triangleright X, N)$ is non-zero, since $\mathcal{M}(Q,N)$ is a direct summand thereof. Thus
  \[
    0 \neq \mathcal{M}(P \triangleright X, N) \cong \mathcal{C}(P,\hom{X,N}),
  \]
  showing that $\hom{X,N}$ is not zero.

  For the latter statement, let $Q'$ be an indecomposable projective object of $\mathcal{M}$, and let $S$ be its simple top. Since $\hom{X,S} \neq 0$, there is some $V \in \mathcal{M}$ such that $\mathcal{C}(V,\hom{X,S}) \neq 0$. Let $P' \twoheadrightarrow V$ be an epimorphism from a projective object in $\mathcal{C}$. Then
  \[
    0\neq \mathcal{C}(V, \hom{X,S}) \simeq \mathcal{C}(P',\hom{X,S}) \simeq \mathcal{C}(P' \triangleright X, S).
  \]
  Since $X$ is $\mathcal{C}$-projective, the object $P' \triangleright X$ is projective. Thus, $\mathcal{C}(P' \triangleright X, S)$ being non-zero implies that $Q'$ is a direct summand of $P' \triangleright X$.
\end{proof}

Again, oppositization yields a similar variant in terms of coclosed and $\mathcal{C}$-injective objects:

\begin{proposition}\label{InjReflect}
  Assume that $\mathcal{C,M}$ have enough injectives. Let $X$ be a coclosed $\mathcal{C}$-injective object of $\mathcal{M}$.
  If\, $X$ is a $\mathcal{C}$-cogenerator, then $\cohom{X, \blank}$ reflects zero objects.

  Additionally, if\, $\inj{\mathcal{C}}$ is Krull--Schmidt and every indecomposable projective object of\, $\mathcal{C}$ is the injective hull of a simple object, then $\cohom{X, \blank}$ reflecting zero objects implies that $X$ is a $\mathcal{C}$-cogenerator.
\end{proposition}

Finally, we give a brief account of internally projective and injective objects in a semisimple module category.

\begin{proposition}
  Let \(\mathcal{C}\) be a monoidal category with enough projectives, and let \(\mathcal{M}\) be a semisimple \(\mathcal{C}\)-module category.
  Since \(\mathcal{M}\) is semisimple, any object of \(\mathcal{M}\) is \(\mathcal{C}\)-projective and \(\mathcal{C}\)-injective.
  An object \(X \in \mathcal{M}\) is a \(\mathcal{C}\)-generator if and only if any simple object \(S \in \mathcal{M}\) is a direct summand of an object of the form \(P \triangleright X\), for some \(P \in \proj{\mathcal{C}}\).
\end{proposition}
\begin{proof}
  Recall that, in a semisimple abelian category, every object is projective.
  In particular, \(P \lact M\) is always projective, for \(M \in \cat{M}\) and \(P \in \proj{\cat{C}}\),
  hence every object in \(\cat{M}\) is \(\cat{C}\)-projective.

  Now, let \(X \in \cat{M}\) be a \(\cat{C}\)-generator.
  Since a simple object \(S \in \cat{M}\) is projective,
  it immediately follows
  from \(X\) being a \(\cat{C}\)-generator
  that \(S\) is a direct summand of \(P \lact X\), for some \(P \in \proj{\cat{C}}\).

  Lastly, assume that every simple object is a direct summand of \(P \lact X\),
  for some \(P \in \cat{C}\),
  and let \(Q \in \proj{\cat{M}}\).
  Since \(\cat{M}\) is semisimple, we have \(Q \cong \oplus_i S_i\) for simples \(S_i\) in \(\cat{M}\).
  Since \(S_i\) is a direct summand of \(P_i \lact X\) for some $P_{i} \in \mathcal{C}$, one calculates
  \[
    Q \,\cong\, \oplus_i S_i \,\subseteq_{\oplus}\, \oplus_i (P_i \lact X) \,\cong\, (\oplus_i P_i) \lact X \,\eqdef\, P \lact X.
  \]
\end{proof}

\section{Reconstruction in terms of lax module endofunctors}\label{sec:lax-module-reconstruction}

\begin{theorem}\label{mainnonsenseprojective}
  Let $\mathcal{C}$ be a monoidal abelian category with enough projectives, $\mathcal{M}$ an abelian $\mathcal{C}$-module category with enough projectives, and assume that $X \in \mathcal{M}$ is a closed $\mathcal{C}$-projective $\mathcal{C}$-generator.

  Then there is an equivalence of\, $\mathcal{C}$-module categories
  \[
    \mathcal{M} \simeq \mathbf{EM}(\hom{X, \blank \lact X}),
  \]
  where $\mathbf{EM}(\hom{X, \blank \lact X})$ is endowed with the extended $\mathcal{C}$-module structure of Theorem~\ref{extendingalways}.
\end{theorem}
\begin{proof}
  Following Example~\ref{ex:regularend}, the functor $\blank \triangleright X$ is a strong $\mathcal{C}$-module functor.
  By Proposition~\ref{prop:special-doctrinal-module-adjunctions}, its right adjoint $\hom{X,\blank}$ is a lax $\mathcal{C}$-module functor.
  Thus, the resulting monad $\hom{X, \blank \triangleright X}$ is a right exact lax $\mathcal{C}$-module monad, and, by Proposition~\ref{strongcomparisons}, the comparison functor $\mathbf{Kl}(\hom{X, \blank \triangleright X}) \to \mathcal{M}$ is a strong $\mathcal{C}$-module functor.
  Furthermore, due to Propositions~\ref{Cprojectives} and~\ref{ProjReflect}, the functor $\hom{X, \blank}$ is right exact and reflects zero objects.
  Applying Theorem~\ref{thm:SmallBeck}, we find that the comparison functor $\mathcal{M} \to \mathbf{EM}(\hom{X, \blank \triangleright X})$ is an equivalence.
  Transporting the $\mathcal{C}$-module structure along this equivalence, we obtain an extended $\mathcal{C}$-module structure on $\mathbf{EM}(T)$, which is necessarily that of Theorem~\ref{extendingalways}, by Theorem~\ref{thm:one-module-structure-on-EM}, establishing the result.
\end{proof}

Using Theorem~\ref{extendingalways}, we can also formulate a converse:

\begin{theorem}\label{projcorrespondence}
  Let $\mathcal{C}$ be a monoidal abelian category with enough projectives, and let $T$ be a right exact lax $\mathcal{C}$-module monad on $\mathcal{C}$. Then $T(\mathbb{1}_{\mathcal{C}})$ is a closed $\mathcal{C}$-projective $\mathcal{C}$-generator in $\mathbf{EM}(T)$, with the $\mathcal{C}$-module category structure on the latter given by Theorem~\ref{extendingalways}.
  There is a bijection
 \[
 \begin{aligned}
  \setj{\,(\mathcal{M},X) \text{ as in Theorem~\ref{mainnonsenseprojective}}\,}/(\mathcal{M} \simeq \mathcal{N}) &\xleftrightarrow{\ \simeq\,} \setj{\,\text{Lax }\mathcal{C}\text{-module monads on } \mathcal{C}\,}/(\mathbf{EM}(T) \simeq \mathbf{EM}(S)) \\
  (\mathcal{M},X) &\longmapsto \hom{X,-\triangleright X} \\
  (\mathbf{EM}(T), T(\mathbb{1})) &\longmapsfrom T
 \end{aligned}
 \]
\end{theorem}

\begin{proof}
 The fact that $T(\mathbb{1})$ is closed follows immediately from the isomorphisms
 \begin{equation}\label{someconverse}
   \mathbf{EM}(T)(\blank \triangleright T(\mathbb{1}_{\mathcal{C}}), \bblank)
   = \mathbf{EM}(T)(T(\blank \otimes \mathbb{1}_{\mathcal{C}}), \bblank)
   \simeq \mathcal{C}(\blank \otimes \mathbb{1}_{\mathcal{C}}, R_{\mathbf{EM}(T)}(\bblank))
   \simeq \mathcal{C}(\blank, R_{\mathbf{EM}(T)}(\bblank))
 \end{equation}
 showing that $\hom{T(\mathbb{1}_{\mathcal{C}}),-} \simeq R_{\mathbf{EM}(T)}$. The first equality uses the fact that $T(\mathbb{1}_{\mathcal{C}}) \in \mathbf{Kl}(T)$, and thus $\blank \triangleright T(\mathbb{1}_{\mathcal{C}}) = T(\blank \otimes \mathbb{1}_{\mathcal{C}})$, since $\mathbf{Kl}(T)$ endowed with the $\mathcal{C}$-module structure described in Proposition~\ref{laxKlCorrespondence} is a $\mathcal{C}$-module subcategory of $\mathbf{EM}(T)$, by Theorem~\ref{strongembeddingalways}. For $P \in \proj{\cat{C}}$ projective, $P \triangleright T(\mathbb{1}) \simeq T(P)$ is projective by Proposition~\ref{finitemonads}, and, by the same result, $T(\mathbb{1})$ is a $\mathcal{C}$-projective generator, and $\mathbf{EM}(T)$ has enough projectives. To verify the latter claim, it suffices to show that $\mathcal{M} \simeq \mathbf{EM}(\hom{X,-\triangleright X})$, which is the claim of Theorem~\ref{mainnonsenseprojective}, and that $\mathbf{EM}(T) \simeq \mathbf{EM}(\hom{T(\mathbb{1}),-\triangleright T(\mathbb{1})})$. But since $\hom{T(\mathbb{1}_{\mathcal{C}}),-} \simeq R_{\mathbf{EM}(T)}$, we in fact have $-\triangleright T(\mathbb{1}) \simeq L_{\mathbf{EM}(T)}$ and hence $T \simeq \hom{T(\mathbb{1}),-\triangleright T(\mathbb{1})}$, which in particular implies equivalence of Eilenberg--Moore categories.
\end{proof}

Using Proposition~\ref{Cinjectives} and Proposition~\ref{InjReflect} analogously to the proof above, we find the dual statements:

\begin{theorem}\label{mainnonsenseinjective}
  Let $\mathcal{C}$ be a monoidal abelian category with enough injectives, $\mathcal{M}$ an abelian $\mathcal{C}$-module category with enough injectives, and assume that $X \in \mathcal{M}$ be a coclosed $\mathcal{C}$-injective $\mathcal{C}$-cogenerator.

  Then there is an equivalence of\, $\mathcal{C}$-module categories
  \[
    \mathcal{M} \simeq \mathbf{EM}(\cohom{M, \blank \lact M}),
  \]
  where $\mathbf{EM}(\cohom{M, \blank \lact M})$ is endowed with the extended $\mathcal{C}$-module structure of Theorem~\ref{extendingalways}.
\end{theorem}

\begin{theorem}\label{injcorrespondence}
  Let $\mathcal{C}$ be a monoidal abelian category with enough injectives,
  and $C$ a left exact oplax $\mathcal{C}$-module comonad on $\mathcal{C}$.
  Then $\mathbf{EM}(C)$, endowed with the $\mathcal{C}$-module category structure of Theorem~\ref{extendingalways}, is a $\mathcal{C}$-module category,
  and $C(\mathbb{1}_{\mathcal{C}})$ is a coclosed $\mathcal{C}$-injective $\mathcal{C}$-generator.
  There is a bijection
  \begin{align*}
    \{\,(\mathcal{M},X)\ \text{as in Theorem~\ref{mainnonsenseinjective}}\,\}
    / (\mathcal{M} \simeq \mathcal{N})
    &\xleftrightarrow{\ \simeq\,}
    \left\{
      \begin{aligned}
        &\text{Left exact oplax}\ \cat{C}\text{-module}\\
        &\text{comonads on}\ \mathcal{C}
      \end{aligned}
      \right\}
      \!\Big/ (\mathbf{EM}(C) \simeq \mathbf{EM}(D)) \\
    (\mathcal{M},X) &\longmapsto \cohom{X, \blank \triangleright X} \\
    \big(\mathbf{EM}(C), C(\mathbb{1})\big) &\longmapsfrom C
  \end{align*}
\end{theorem}

Observe that Theorems~\ref{mainnonsenseprojective} and~\ref{mainnonsenseinjective} do not make any finiteness assumptions on $\mathcal{M}$ relative to $\mathcal{C}$---such finiteness conditions are often \emph{imposed} on $\mathcal{M}$ by the existence of a closed $\mathcal{C}$-projective $\mathcal{C}$-generator, respectively a coclosed $\mathcal{C}$-injective $\mathcal{C}$-cogenerator.

\begin{proposition}
  Let $\mathcal{C}$ be a finite abelian monoidal category, and let $\mathcal{M}$ be an abelian $\mathcal{C}$-module category such that there exists a closed $\mathcal{C}$-projective $\mathcal{C}$-generator $X \in \mathcal{M}$. Then $\mathcal{M}$ is finite abelian.
\end{proposition}
\begin{proof}
  By Theorem~\ref{mainnonsenseprojective}, we have a $\mathcal{C}$-module equivalence $\mathcal{M} \simeq \mathbf{EM}(\hom{X,\blank \triangleright X})$.
  Observe that $\blank \triangleright X$ is right exact, since it admits a right adjoint.
  Since $X$ is $\mathcal{C}$-projective, the functor $\hom{X, \blank}$ also is right exact.
  Thus $\hom{X, \blank \triangleright X}$ is a right exact monad on the finite abelian category $\mathcal{C}$, and hence---by Proposition~\ref{finitemonads}---its Eilenberg--Moore category $\mathbf{EM}(\hom{X, \blank \triangleright X})$ is finite abelian.
\end{proof}

Analogously, using Proposition~\ref{locomonads} rather than Proposition~\ref{finitemonads},
one obtains the following.

\begin{proposition}
  Let $\mathcal{C}$ be a locally finite abelian monoidal category, and let $\mathcal{M}$ be an abelian $\mathcal{C}$-module category with enough injectives such that there exists a coclosed $\mathbf{Ind}(\mathcal{C})$-injective $\mathbf{Ind}(\mathcal{C})$-generator $X \in \mathcal{M}$. Then the category $\mathbf{cpt}(\mathcal{M})$ of compact objects in $\mathcal{M}$ is locally finite abelian.
\end{proposition}

On the other hand, in the presence of suitable finiteness conditions, the characterizations of adjoint functors between finite and locally finite abelian categories give sufficient conditions for an object in a module category to be closed, respectively coclosed.

\begin{proposition}
  Let $\mathcal{C}$ be a finite abelian monoidal category and let $\mathcal{M}$ be a finite abelian $\mathcal{C}$-module category. An object $X \in \mathcal{M}$ is closed if and only if the functor $\blank \triangleright X$ is right exact.
\end{proposition}

\begin{proof}
  This is an immediate consequence of Proposition~\ref{finiteadjoints}.
\end{proof}

Similarly, using Proposition~\ref{locallyfiniteadjoints}, one shows the following.

\begin{proposition}
  Let $\mathcal{C}$ be a locally finite abelian monoidal category and let $\mathcal{M}$ be a locally finite abelian $\mathcal{C}$-module category. An object $X \in \mathbf{Ind}(\mathcal{M})$ is coclosed with respect to the induced $\mathbf{Ind}(\mathcal{C})$-module structure on $\mathbf{Ind}(\mathcal{M})$ if and only if\, $\blank \triangleright X\from \mathcal{C} \to \mathbf{Ind}(\mathcal{M})$ is left exact and the induced functor $\blank \triangleright X\from \mathbf{Ind}(\mathcal{C}) \to \mathbf{Ind}(\mathcal{M})$ is quasi-finite.
\end{proposition}

\subsection{The rigid case}

\begin{proposition}\label{prop:strong-module-monad-and-algebra-objects}
  Let \((\mathcal{C}, \otimes, \mathbb{1})\) be a monoidal category.
  There is a bijection between strong \(\mathcal{C}\)-module monads on \(\cat{C}\) and algebra objects in \(\mathcal{C}\).
\end{proposition}
\begin{proof}
  By Proposition~\ref{prop:bicategorical-yoneda-correspondence-monads-and-algebras},
  we have that there is an equivalence of left module categories
  \begin{align*}
    \lMod{\cat{C}}(\cat{C}, \cat{C}) \xiso \cat{C}^{\otimes\on{rev}}, \qquad
    T \mapsto T(\mathbb{1}), \qquad
    \blank \otimes A \longmapsfrom A.
  \end{align*}

  In particular, if \((A, m, u)\) is an algebra object in \(\cat{C}\),
  then \((\blank \otimes A,\, \blank \otimes m,\, \blank \otimes u)\) becomes a strong $\mathcal{C}$-module monad on \(\cat{C}\).
  The strong $\mathcal{C}$-module structure is given by the associator of \(\cat{C}\):
  \[
    {({(\blank \otimes A)}_{\mathsf{a}})}_V \defeq V \otimes (\blank \otimes A) \xiso (V \otimes \blank) \otimes A.
  \]

  On the other hand,
  if \((T, \mu, \eta)\) is a strong \(\cat{C}\)-module monad,
  then evaluating at the monoidal unit yields an algebra object in \(\cat{C}\),
  with multiplication \(\mu_{\mathbb{1}}\) and unit \(\eta_{\mathbb{1}}\).
\end{proof}

\begin{example}\label{ex:EM-is-a-modules}
  Let \(\cat{C}\) be a monoidal category,
  and \(A \in \cat{C}\) and algebra object.
  Then the Eilenberg--Moore category for \(A \otimes \blank\) is the category of left \(A\)-module objects,
  see e.g.~\cite[Example~1.2]{BV}.
  Likewise, \(\mathbf{EM}(\blank \otimes A)\) yields the category of right \(A\)-module objects.
\end{example}

\begin{theorem}\label{thm:main-result-rigid-projective-case}
  Let \(\mathcal{C}\) be a rigid monoidal abelian category with enough projectives,
  \(\mathcal{M}\) an abelian \(\mathcal{C}\)-module category,
  and assume that \(X \in \mathcal{M}\) is a closed \(\mathcal{C}\)-projective \(\mathcal{C}\)-generator.
  Then there is an algebra object \(A \in \mathcal{C}\) such that
  there is an equivalence of\, \(\mathcal{C}\)-module categories
  \[
    \mathrm{mod}_{\mathcal{C}}A \simeq \mathcal{M}.
  \]
\end{theorem}
\begin{proof}
  By Theorem~\ref{mainnonsenseprojective},
  we have \(\mathbf{EM}(\hom{X, \blank \lact X}) \simeq \cat{M}\)
  as \(\cat{C}\)-module categories.
  In particular, applying Proposition~\ref{prop:lax-is-strong-when-rigid} yields that
  the monad \(\hom{X, \blank \lact X}\from \cat{C}\to \cat{C}\) is a strong $\mathcal{C}$-module functor,
  as it is lax and \(\cat{C}\) is rigid.
  The claim follows by Proposition~\ref{prop:strong-module-monad-and-algebra-objects}
  and Example~\ref{ex:EM-is-a-modules}.
\end{proof}

\begin{theorem}\label{thm:main-result-rigid-injective-case}
  Let $\mathcal{C}$ be a rigid monoidal abelian category with enough injectives, $\mathcal{M}$ an abelian $\mathcal{C}$-module category, and assume that $X \in \mathcal{M}$ is a coclosed $\mathcal{C}$-injective $\mathcal{C}$-cogenerator.

  Then there is a coalgebra object $K \in \mathcal{C}$ such that there is an equivalence of $\mathcal{C}$-module categories $\on{comod}_{\mathcal{C}}A \simeq \mathcal{M}$.
\end{theorem}

\subsection{The case of tensor categories and finite tensor categories}

\begin{theorem}[{\cite[Theorem~7.10.1]{EGNO}}]
  Let \(\mathcal{C}\) be a finite multitensor category, and let \(\mathcal{M}\) be an abelian \(\mathcal{C}\)-module category such that
  \begin{enumerate}
    \item \(\blank \triangleright \bblank\) is exact in the first variable (as \(\mathcal{C}\) is rigid, it is always exact in the second variable);
    \item it admits a \(\mathcal{C}\)-projective \(\mathcal{C}\)-generator.
  \end{enumerate}
  Then there is an algebra object \(A\) in \(\mathcal{C}\) such that \(\mathcal{M} \simeq \mathrm{mod}_{\mathcal{C}}A\).
\end{theorem}

We generalize this statement to the locally finite setting:

\begin{theorem}
  Let \(\mathcal{C}\) be a multitensor category,
  and let \(\mathcal{M}\) be an abelian \(\mathcal{C}\)-module category
  such that the \(\mathbf{Ind}(\mathcal{C})\)-module category \(\mathbf{Ind}(\mathcal{M})\) admits a coclosed \(\mathbf{Ind}(\mathcal{C})\)-injective \(\mathbf{Ind}(\mathcal{C})\)-cogenerator $X$.
  Then there is a coalgebra object \(C\) in \(\mathbf{Ind}(\mathcal{C})\)
  such that \(\mathbf{Ind}(\mathcal{M}) \simeq \mathrm{Comod}_{\mathbf{Ind}(\mathcal{C})}C\).
  Thus, \(\mathcal{M}\) is the category of compact \(C\)-comodule objects.
\end{theorem}
\begin{proof}
  By Theorem~\ref{mainnonsenseinjective}, we have \( \mathbf{Ind}(\mathcal{M}) \simeq \mathbf{EM}(\cohom{X, \blank \lact X}) \)
  as $\mathbf{Ind}(\mathcal{C})$-module categories.
  Observe that the $\mathbf{Ind}(\mathcal{C})$-module category structure on both $\mathbf{Ind}(\mathcal{C})$ and on $\mathbf{Ind}(\mathcal{M})$ is the finitary extension of the respective $\mathcal{C}$-module category structures,
  following Proposition~\ref{cocompletingnonsense}.
  Similarly, $\blank \triangleright X\from \mathbf{Ind}(\mathcal{C}) \to \mathbf{Ind}(\mathcal{M})$ is the extension of $\blank \triangleright X\from \mathcal{C} \to \mathbf{Ind}(\mathcal{M})$.
  The left adjoint $\cohom{X,\blank}\from \mathbf{Ind}(\mathcal{M}) \to \mathbf{Ind}(\mathcal{C})$, being finitary and preserving compact objects by Lemma~\ref{leftadjointssmallobjects}, restricts to an oplax $\mathcal{C}$-module functor $\cohom{X,\blank}\from \mathcal{M} \to \mathcal{C}$.
  Since $\mathcal{C}$ is rigid, the restricted $\mathcal{C}$-module functor $\cohom{X,\blank}\from \mathcal{M} \to \mathcal{C}$ is in fact a strong $\mathcal{C}$-module functor.
  Thus, its finitary extension $\cohom{X,\blank}\from \mathbf{Ind}(\mathcal{M}) \to \mathbf{Ind}(\mathcal{C})$ is a strong $\mathbf{Ind}(\mathcal{C})$-module functor.
  Thus the comonad $\cohom{X, \blank \lact X}$ is a strong $\mathbf{Ind}(\mathcal{C})$-module monad, and, by Proposition~\ref{prop:strong-module-monad-and-algebra-objects}, it is of the form $\blank \otimes C$ for a coalgebra object in $\mathbf{Ind}(\mathcal{C})$.
  Following Example~\ref{ex:EM-is-a-modules}, one obtains $\mathbf{Ind}(\mathcal{M}) \simeq \mathbf{EM}(\cohom{X, \blank \lact X}) \simeq \mathrm{Comod}_{\mathbf{Ind}(\mathcal{C})}C$.

  From Proposition~\ref{locomonads}, we conclude that $\mathcal{M}$ consists of compact \(C\)-comodule objects of $\mathcal{C}$.
\end{proof}

\begin{corollary}\label{infinitecomodulecoalgebras}
  Let $H$ be a Hopf algebra, and let $\mathcal{C} = \tetramodfd[H]{}$, so that in particular we have a monoidal equivalence $\mathbf{Ind}(\mathcal{C}) \simeq \tetramod[H]{}$.
  Let \(\mathcal{M}\) be an abelian \(\mathcal{C}\)-module category such that the \(\mathbf{Ind}(\mathcal{C})\)-module category \(\mathbf{Ind}(\mathcal{M})\) admits a coclosed \(\mathbf{Ind}(\mathcal{C})\)-injective \(\mathbf{Ind}(\mathcal{C})\)-cogenerator.
  Then there is an $H$-comodule coalgebra $K$ such that there is an equivalence $\mathbf{Ind}(\mathcal{M}) \simeq \on{Comod}_{H}K$ of $\mathbf{Ind}(\mathcal{C})$-module categories, restricting to an equivalence $\mathcal{M} \simeq \on{comod}_{H}K$ of $\mathcal{C}$-module categories.
\end{corollary}

\section{An Eilenberg--Watts theorem for lax module monads}\label{Takyon}

\begin{definition}
  Let \(\mathcal{A}\) be an abelian category and suppose that \(S_{1},S_{2}\) are right exact monads on \(\mathcal{A}\).
  Viewing \(S_{1},S_{2}\) as algebra objects in \(\mathbf{Rex}(\mathcal{A,A})\),
  an \emph{\(S_{1}\)-\(S_{2}\)-biact functor} is an \(S_{1}\)-\(S_{2}\)-bimodule object in \(\mathbf{Rex}(\mathcal{A,A})\).
\end{definition}

In other words, a biact functor \(F\from \mathcal{A} \to \mathcal{A}\) is a triple \((F,F_{\mathsf{la}},F_{\mathsf{ra}})\),
consisting of a right exact endofunctor on \(\mathcal{A}\) together with transformations
\(F_{\mathsf{la}}\from S_{1}\circ F \nt F\)
and
\(F_{\mathsf{ra}}\from F \circ S_{2} \nt F\), satisfying the natural unitality and associativity axioms, and commuting with each other in the sense that
\begin{equation}\label{eq:comm-biact}
  F_{\mathsf{la}} \circ (S_{1}\circ F_{\mathsf{ra}}) = F_{\mathsf{ra}} \circ (F_{\mathsf{la}} \circ S_{2}) \from S_1 \circ F \circ S_2 \nt F,
\end{equation}
where, for example, \(S_1 \circ F_{\mathsf{ra}}\from S_1 \circ F \circ S_2 \to S_1 \circ F\) denotes the \emph{whiskering} of \(S_1\) and \(F_{\mathsf{ra}}\).

The category of \(S_{1}\)-\(S_{2}\)-biact functors shall be denoted by \(\biact{S_1}{S_2}\).

\begin{remark}\label{factorthrough}
  For an $S_{1}$-$S_{2}$-biact functor $F$, the left $S_{1}$-action $F_{\mathsf{la}}$ endows any object of the form $F(A)$, for $A \in \mathcal{A}$, with the structure of an $S_{1}$-module, via the action$\nabla_{F(A)} \defeq {(F_{\mathsf{la}})}_{A}$.
  Any morphism $f\from A \to B$ lifts to a map \(Ff\from F(A) \to F(B)\) of \(S_1\)-modules
  by naturality of $F_{\mathsf{la}}$.
  Thus, the functor $F$ factors through $R_{\mathbf{EM}(S_{1})}$; we write $F = R_{\mathbf{EM}(S_{1})} \circ \overline{F}$.
\end{remark}

\begin{definition}
  Given an $S_{1}$-$S_{2}$-biact functor $F$, define the functor $\overline{F} \circ_{S_2} \blank\from \mathbf{EM}(S_{2}) \to \mathbf{EM}(S_{1})$ as
  \[
    \overline{F} \circ_{S_2} \blank \defeq \on{coeq}\Big(\!\!\!
    \tikzexternaldisable
    \begin{tikzcd}[ampersand replacement=\&]
      {FS_{2}(\blank)} \&\& {F(\blank)}
      \arrow["{F(\nabla_{\blank})}", shift left=2, from=1-1, to=1-3]
      \arrow["{{(F_{\mathrm{ra}})}_{\blank}}"', shift right=2, from=1-1, to=1-3]
    \end{tikzcd}
    \tikzexternalenable
    \!\!\!\Big),
  \]
  where, for $X \in \mathbf{EM}(S_{2})$ we endow $F(X)$ and $FS_{2}(X)$ with $S_{1}$-module structures described in Remark~\ref{factorthrough}.
  By naturality of $F_{\mathsf{la}}$ and the commutativity condition of Equation~\ref{eq:comm-biact},
  both morphisms in the above coequalizer are morphisms in $\mathbf{EM}(S_{1})$, and so, since the coequalizers in $\mathbf{EM}(S_{1})$ are created by $R_{\mathbf{EM}(S_{1})}$, this coequalizer is a coequalizer in $\mathbf{EM}(S_{1})$.
  Functoriality follows from functoriality of colimits.
\end{definition}

\begin{proposition}\label{EilenbergWatts}
  Let $S_{1}$ and \(S_{2}\) be right exact monads on an abelian category $\mathcal{A}$.
  The following assignments extend to an equivalence of categories:
  \begin{align*}
    \mathbf{Rex}(\mathbf{EM}(S_{2}),\mathbf{EM}(S_{1})) &\longleftrightarrow \biact{S_{1}}{S_{2}} \\
    \Phi &\longmapsto \big(
        R_{\mathbf{EM}(S_{1})} \Phi L_{\mathbf{EM}(S_{2})},\ %
        R_{\mathbf{EM}(S_1)} \varepsilon_{\mathbf{EM}(S_{1})} \Phi L_{\mathbf{EM}(S_2)},\ %
        R_{\mathbf{EM}(S_{1})} \Phi \varepsilon_{\mathbf{EM}(S_{2})} L_{\mathbf{EM}(S_{2})}
    \big)\\
    \overline{F} \circ_{S_2} \blank &\longmapsfrom F,
  \end{align*}
  where \(\varepsilon_{\mathbf{EM}(S_1)}\) and \(\varepsilon_{\mathbf{EM}(S_2)}\) are the counit of the respective Eilenberg–Moore adjunctions.
\end{proposition}

\begin{proof}
  Since $L_{\mathbf{EM}(S_{2})}$ is right exact and $R_{\mathbf{EM}(S_{1})}$ is exact, the composite $R_{\mathbf{EM}(S_{1})}\circ \Phi \circ L_{\mathbf{EM}(S_{2})}$ is right exact. It is easy to verify that the transformations $R_{s} \circ \varepsilon_{\mathbf{EM}(S_{1})} \circ \Phi L_{S_{2}}$ and $R_{\mathbf{EM}(S_{1})}\Phi \circ \varepsilon_{\mathbf{EM}(S_{2})} \circ L_{\mathbf{EM}(S_{2})}$ endow $R_{\mathbf{EM}(S_{1})}\circ \Phi \circ L_{\mathbf{EM}(S_{2})}$ with the structure of a biact functor, and that for a natural transformation $t\from \Phi \nt \Phi'$, the resulting transformation $R_{\mathbf{EM}(S_{1})} \circ t \circ L_{\mathbf{EM}(S_{2})}$ is a morphism of biact functors.
  This proves the functoriality of the assignment $\mathbf{Rex}(\mathbf{EM}(S_{2}),\mathbf{EM}(S_{1})) \to \biact{S_{1}}{S_{2}}$.
  To see that the converse assignment extends to a functor, observe that given a morphism $f\from F \nt G$ of biact functors,
  we obtain a morphism of forks in $\mathbf{EM}(S_{1})$:
\[\begin{mytikzcd}[ampersand replacement=\&]
	{FS_{2}(X)} \&\& {F(X)} \\
	\\
	{GS_{2}(X)} \&\& {G(X)}
	\arrow["{F(\nabla_{X})}", shift left=2, from=1-1, to=1-3]
	\arrow["{{(F_{\mathrm{ra}})}_{X}}"', shift right=2, from=1-1, to=1-3]
	\arrow["{f_{S_{2}(X)}}"', from=1-1, to=3-1]
	\arrow["{f_{X}}", from=1-3, to=3-3]
	\arrow["{G(\nabla_{X})}", shift left=2, from=3-1, to=3-3]
	\arrow["{{(G_{\mathrm{ra}})}_{X}}"', shift right=2, from=3-1, to=3-3]
\end{mytikzcd}\]
This induces a morphism of coequalizers $\overline{F} \circ_{S_2} \blank \to \overline{G} \circ_{S_2} \blank$.

We now prove that these functors define mutually quasi-inverse equivalences. First, observe that
\(\overline{R_{\mathbf{EM}(S_{1})}\circ \Phi \circ L_{\mathbf{EM}(S_{2})}} = \Phi \circ L_{\mathbf{EM}(S_{2})} \from \cat{C} \to \mathbf{EM}(S_1)\),
since \(R_{\mathbf{EM}(S_{1})}\) is faithful and injective on objects, so it is a monomorphism in the \(1\)-category \(\mathbf{Cat}_{\Bbbk}\).
For \(X \in \mathbf{EM}(S_2)\), we have
\[
  (\Phi \circ L_{\mathbf{EM}(S_{2})}) \circ_{S_2} X = \on{coeq}\Big(
  \tikzexternaldisable
  \begin{tikzcd}[ampersand replacement=\&]
    {\Phi S_{2}^{2}(X)} \&\& {\Phi S_{2}(X)}
    \arrow["{\Phi\mu^{S_{2}}_{X}}", shift left=2, from=1-1, to=1-3]
    \arrow["{\Phi S_{2}(\nabla_{X})}"', shift right=2, from=1-1, to=1-3]
  \end{tikzcd}
  \tikzexternalenable
  \Big),
\]
which, since $\Phi$ is right exact, is isomorphic to
\[
  \Phi\Big(\on{coeq}\Big(
  \tikzexternaldisable
  \begin{tikzcd}[ampersand replacement=\&]
    {S_{2}^{2}(X)} \&\& {S_{2}(X)}
    \arrow["{\mu^{S_{2}}_{X}}", shift left=2, from=1-1, to=1-3]
    \arrow["{S_{2}(\nabla_{X})}"', shift right=2, from=1-1, to=1-3]
  \end{tikzcd}
  \tikzexternalenable
  \Big)\Big),
\]
and, via the action \(\nabla_X\from T(X) \to X\), the latter coequalizer is canonically isomorphic to $X$.
Thus, we obtain the following isomorphism, which is clearly natural in \(\Phi\):
\[
  (\Phi \circ L_{\mathbf{EM}(S_{2})}) \circ_{S_2} X \xrightarrow[\sim]{\Phi(\widetilde{\nabla}_{X})} \Phi(X).
\]

In the other direction, let $F$ be a biact funtcor and $X \in \mathcal{A}$.
We have
\[
  \Big(R_{\mathbf{EM}(S_{1})} \circ (\overline{F} \circ_{S_2} \blank) \circ L_{\mathbf{EM}(S_{2})}\Big)(X) =
  \on{coeq}\Big(
  \tikzexternaldisable
  \begin{tikzcd}[ampersand replacement=\&]
    {FS_{2}^{2}(X)} \&\& {FS_{2}(X)}
    \arrow["{{(F_{\mathsf{ra}})}_{S_{2}(X)}}", shift left=2, from=1-1, to=1-3]
    \arrow["{F(\nabla_{S_{2}(X)})=F(\mu_{X})}"', shift right=2, from=1-1, to=1-3]
  \end{tikzcd}
  \tikzexternalenable
  \Big)
\]
Observe that $F$, being a right $S_{2}$-module object, is presented as a coequalizer in the category of such objects, which additionally splits in the underlying monoidal category $\mathbf{Rex}(\mathcal{A,A})$, as indicated in the following diagram:
\[\begin{mytikzcd}[ampersand replacement=\&]
	{FS_{2}^{2}} \&\& {FS_{2}} \& {\on{coeq}(F_{\mathsf{ra}}\circ S_{2}, F\circ \mu_{X})} \\
	\&\&\& F
	\arrow["{F_{\mathsf{ra}}\circ S_{2}}", shift left=2, from=1-1, to=1-3]
	\arrow["{F\circ \mu_{X}}"', shift right=2, from=1-1, to=1-3]
	\arrow[two heads, from=1-3, to=1-4]
	\arrow["{F_{\mathsf{ra}}}"{description}, from=1-3, to=2-4]
	\arrow["{\exists! \widetilde{F_{\mathsf{ra}}}}", dashed, from=1-4, to=2-4]
	\arrow["\simeq"', dashed, from=1-4, to=2-4]
	\arrow["{F\circ \eta^{S_{2}}}", curve={height=-12pt}, dotted, from=2-4, to=1-3]
\end{mytikzcd}\]
Since $F_{\mathsf{ra}}$ is a morphism of left $S_{1}$-module objects, we find that $\widetilde{F}_{\mathsf{ra}}$ defines an isomorphism of biact functors $R_{\mathbf{EM}(S_{1})} \circ (\overline{F} \circ_{S_2} \blank) \circ L_{\mathbf{EM}(S_{2})} \cong F$ that is natural in $F$.
\end{proof}

Similarly to Subsection~\ref{LintinTime}, let us now assume $\mathcal{C}$ and $\mathcal{M}$ to be abelian, $\blank \triangleright \bblank$ to be right exact in both variables and $S,T$ to be right exact lax $\mathcal{C}$-module monads on $\mathcal{M}$.

\begin{definition}
  Viewing $S$ and \(T\) as algebra objects in the category $\mathbf{RexLax}\mathcal{C}\text{-}\mathbf{Mod}(\mathcal{M,M})$ of right exact lax $\mathcal{C}$-module endofunctors of $\mathcal{M}$, we refer to $S$-$T$-bimodule objects in that category as \emph{lax $\mathcal{C}$-module $S$-$T$-biact functors}.
  In other words, a lax $\mathcal{C}$-module biact functor is a triple $(F,F_{\mathsf{la}},F_{\mathsf{ra}})$, consisting of a right exact lax $\mathcal{C}$-module endofunctor $F$ of $\mathcal{M}$ together with $\mathcal{C}$-module transformations $F_{\mathsf{la}}\from S\circ F \nt F$ and $F_{\mathsf{ra}}\from F \circ T \nt F$,
  satisfying the natural unitality and associativity axioms,
  and commuting with each other in the sense that
  $F_{\mathsf{la}} \circ (S \circ F_{\mathsf{ra}}) = F_{\mathsf{ra}} \circ (F_{\mathsf{la}} \circ T)$.
  We denote the category of $S$-$T$-biact functors by $\biactC[\cat{C}]{S}{T}$.
\end{definition}

\begin{theorem}\label{ultraEilenbergWatts}
  The equivalence \(\mathbf{Rex}(\mathbf{EM}(T),\mathbf{EM}(S)) \simeq \biact{S}{T}\)
  of Proposition~\ref{EilenbergWatts} restricts to a faithful functor
  \[
    \mathcal{C}\mathsf{EW}\from \mathbf{LaxRex}(\mathbf{EM}(T), \mathbf{EM}(S)) \to \biactC[\cat{C}]{S}{T},
  \]
  such that the following diagram commutes:
  \begin{equation}\label{mikey}
    \tikzexternaldisable
    \begin{mytikzcd}[ampersand replacement=\&]
      {\mathbf{LaxRex}(\mathbf{EM}(T), \mathbf{EM}(S))} \&\& {\biactC[\cat{C}]{S}{T}} \\
      {\mathbf{Rex}(\mathbf{EM}(T), \mathbf{EM}(S))} \&\& {\biact{S}{T}}
      \arrow["\mathcal{C}\mathsf{EW}", from=1-1, to=1-3]
      \arrow[from=1-1, to=2-1]
      \arrow[from=1-3, to=2-3]
      \arrow["\simeq \text{ by Proposition~\ref{EilenbergWatts}}", from=2-1, to=2-3]
    \end{mytikzcd}
    \tikzexternalenable
  \end{equation}
  The vertical arrows are the respective forgetful functors, and
  the categories $\mathbf{EM}(T)$ and $\mathbf{EM}(S)$ are endowed with the $\mathcal{C}$-module structures of Theorem~\ref{extendingalways}.
\end{theorem}

\begin{proof}
  By Theorem~\ref{strongembeddingalways}, similarly to~\cite[Proposition~3.10]{aguiar18:monad}, $L_{\mathbf{EM}(T)}\from \mathcal{M} \to \mathbf{EM}(T)$ is a strong $\mathcal{C}$-module functor, and so is $L_{\mathbf{EM}(S)}$.
  Thus, by Porism~\ref{porism:doctrinal-module-adjunctions}, the functor $R_{\mathbf{EM}(S)}$ is a lax $\mathcal{C}$-module functor,
  and hence for a lax $\mathcal{C}$-module functor $\Phi\from \mathbf{EM}(T) \to \mathbf{EM}(S)$, the composite $R_{\mathbf{EM}(S)} \circ \Phi \circ L_{\mathbf{EM}(T)}$ is a lax $\mathcal{C}$-module functor.
  Further, the $S$-$T$-biact structure on $R_{\mathbf{EM}(S)} \circ \Phi \circ L_{\mathbf{EM}(T)}$ given in Proposition~\ref{EilenbergWatts} clearly assembles to a lax $\mathcal{C}$-module biact functor, and this extends also to $\mathcal{C}$-module biact transformations arising from $\mathcal{C}$-module transformations.

 This defines a functor $\mathcal{C}\mathsf{EW}$, such that Diagram~\eqref{mikey} commutes. It is faithful, since so are the remaining functors in that diagram.
\end{proof}

\begin{lemma}\label{phigter}
  Let $\Phi \in \mathbf{Rex}(\mathbf{EM}(T),\mathbf{EM}(S))$ and for all \(X \in \cat{M}\) suppose that
 \[
   {({(\Phi T)}_{\mathsf{a}})}_{X}\from V \triangleright_{\mathcal{M}} \Phi T(X) \to \Phi T(V \triangleright_{\mathcal{M}} X)
 \]
 is a lax $\mathcal{C}$-module biact functor structure on the biact functor $R_{\mathbf{EM}(S)} \circ \Phi \circ L_{\mathbf{EM}(T)}$.
 Then we have ${(\Phi T)}_{\mathsf{a}} = \mathcal{C}\mathsf{EW}({(\Phi_{\mathsf{a}})}_{T(\blank)})$, where ${(\Phi_{\mathsf{a}})}_{T(\blank)}\from V \triangleright_{\mathbf{EM}(S)} \Phi(T(\blank)) \to \Phi(V \triangleright_{\mathbf{EM}(T)} T(\blank)) = \Phi(T(V\triangleright_{\mathcal{M}}X))$ is the unique transformation making the diagram
\begin{equation}\label{phiaeio}
 \begin{mytikzcd}[ampersand replacement=\&]
	{V \lact_{\mathbf{EM}(S)} \Phi T-} \\
	{S(V\triangleright_{\mathcal{M}}\Phi T-)} \& {S\Phi T V\triangleright_{\mathcal{M}} -} \& {\Phi T V\triangleright_{\mathcal{M}} -}
	\arrow["{\exists ! \Phi_{\mathsf{a}}}", from=1-1, to=2-3]
	\arrow[two heads, from=2-1, to=1-1]
	\arrow["{S{(\Phi T)}_{\mathsf{a}}}"', from=2-1, to=2-2]
	\arrow["{\Phi_{\mathsf{la}} TV\triangleright-}"', from=2-2, to=2-3]
\end{mytikzcd}
\end{equation}
commute.
\end{lemma}

\begin{proof}
  First, in order to verify that \(\Phi_{\mathsf{a}}\) is well-defined,
  we observe that \((\Phi_{\mathsf{la}}TV\triangleright_{\mathcal{M}}\blank)\circ S{(\Phi T)}_{\mathsf{a}}\) coequalizes \(S({(S_{\mathsf{a}})}_{V})\Phi T\) and \(SV\triangleright_{\mathcal{M}} \Phi_{\mathsf{la}} T\) by Figure~\ref{fig:eilenberg-watts-coequaliser-1},
  where \((1)\) follows by \({(\Phi T)}_{\mathsf{la}}\) being a left \(S\)-act structure,
  and \((2)\) by it being a \(\cat{C}\)-module transformation.
  \begin{figure}[htbp]
    \tikzsetnextfilename{fig-eilenberg-watts-coequaliser-1}

    \caption{}%
    \label{fig:eilenberg-watts-coequaliser-1}
  \end{figure}

  By the definition of the functor \(\mathcal{C}\mathsf{EW}\), we have \(\mathcal{C}\mathsf{EW}({(\Phi_{\mathsf{a}})}_{T(\blank)}) = {(\Phi_{\mathsf{a}})}_{T(\blank)} \circ {({(R_{\mathbf{EM}(S)})}_{\mathsf{a}})}_{V,\Phi T(\blank)}\),
  where \({(R_{\mathbf{EM}(S)})}_{\mathsf{a}}\) is the lax \(\mathcal{C}\)-module functor structure on \(R_{\mathbf{EM}(S)}\), coming from Porism~\ref{porism:doctrinal-module-adjunctions}.
  More generally, the map \({({(R_{\mathbf{EM}(S)})}_{\mathsf{a}})}_{V,X}\) is given by the composite
  \[
    V \triangleright_{\mathcal{M}} X \xrightarrow{\eta_{V \triangleright_{\mathcal{M}} X}} S(V \triangleright_{\mathcal{M}} X) \twoheadrightarrow V \triangleright_{\mathbf{EM}(S)} X,
  \]
  with the latter map being the projection onto the coequalizer:
  \[
    \begin{mytikzcd}[ampersand replacement=\&]
      {S(V \lact X)} \& {\mathrm{coeq}\Big( T(V \lact T^2(X))} \& {T(V \lact  T(X))\Big)} \\
      {V \lact X} \& {\mathrm{coeq}\Big( T(V \lact T^2(X))} \& {T(V) \lact  T(X)\Big)}
      \arrow[Rightarrow, no head, from=1-1, to=1-2]
      \arrow[two heads, from=1-1, to=2-1]
      \arrow[shift left=2, from=1-2, to=1-3]
      \arrow[shift right=2, from=1-2, to=1-3]
      \arrow["{T(V \lact \nabla_X)}", from=1-2, to=2-2]
      \arrow["{T(V  \lact \nabla_X)}", from=1-3, to=2-3]
      \arrow[Rightarrow, no head, from=2-1, to=2-2]
      \arrow[shift left=2, from=2-2, to=2-3]
      \arrow[shift right=2, from=2-2, to=2-3]
    \end{mytikzcd}
  \]
  Thus, we find that \(\mathcal{C}\mathsf{EW}({(\Phi_{\mathsf{a}})}_{T(\blank)})\) is given by the outer path of the commutative diagram
  \[\begin{mytikzcd}[ampersand replacement=\&]
      \& {V \triangleright_{\mathbf{EM}(S)} \Phi T\blank} \\
      {S(V\triangleright_{\mathcal{M}}\Phi T\blank)} \& {S(V\triangleright_{\mathcal{M}}\Phi T\blank)} \& {S\Phi T V\triangleright_{\mathcal{M}} \blank} \& {\Phi T V\triangleright_{\mathcal{M}} \blank}
      \arrow["{\exists! \Phi_{\mathsf{a}}}", from=1-2, to=2-4]
      \arrow["{\eta_{V\triangleright_{\mathcal{M}}\Phi T\blank}}", from=2-1, to=2-2]
      \arrow[two heads, from=2-2, to=1-2]
      \arrow["{S{(\Phi T)}_{\mathsf{a}}}"', from=2-2, to=2-3]
      \arrow["{\Phi_{\mathsf{la}} TV\triangleright\blank}"', from=2-3, to=2-4]
    \end{mytikzcd}\]
  which, by the unitality of the \(S\)-act structure \({(\Phi T)}_{\mathsf{la}}\), equals ${(\Phi T)}_{\mathsf{a}}$.
\end{proof}

For $X \in \mathbf{EM}(T)$, let ${(\Phi_{\mathsf{a}})}_{X}$ be the unique map making the following diagram commute:
\begin{equation}\label{tekken}
\begin{mytikzcd}[ampersand replacement=\&]
	{V \triangleright_{\mathbf{EM}(S)} \Phi T(X)} \&\& {V \triangleright_{\mathbf{EM}(S)} \Phi(X)} \\
	{\Phi T(V \triangleright X)} \\
	{\Phi(V \triangleright_{\mathbf{EM}(T)} T(X))} \&\& {\Phi(V \triangleright_{\mathbf{EM}(T)} X)}
	\arrow["{V \triangleright_{\mathbf{EM}(S)}\Phi\nabla_{X}}"', two heads, from=1-1, to=1-3]
	\arrow["{{(\Phi_{\mathsf{a}})}_{T(X)}}", from=1-1, to=2-1]
	\arrow["{\exists! {(\Phi_{\mathsf{a}})}_{X}}"', dashed, from=1-3, to=3-3]
	\arrow["{=}", from=2-1, to=3-1]
	\arrow["{\Phi(V \triangleright_{\mathbf{EM}(T)} \nabla_{X})}", two heads, from=3-1, to=3-3]
\end{mytikzcd}
\end{equation}
By the uniqueness in the defining property of ${(\Phi_{\mathsf{a}})}_{X}$, this defines a natural transformation
\[
  \Phi_{\mathsf{a}}\from \blank \triangleright_{\mathbf{EM}(S)} \Phi(\bblank) \nt \Phi(\blank \triangleright_{\mathbf{EM}(T)} \bblank).
\]

\begin{lemma}\label{bigdiagrams}
  For\, $\Phi \in \mathbf{Rex}(\mathbf{EM}(T),\mathbf{EM}(S))$, the map $\Phi_{\mathsf{a}}\from \blank \triangleright_{\mathbf{EM}(S)} \Phi(\bblank) \nt \Phi(\blank \triangleright_{\mathbf{EM}(S)} \bblank)$, extended from the morphisms ${(\Phi_{\mathsf{a}})}_{T(X)}$ via Diagram~\eqref{tekken}, defines a lax $\mathcal{C}$-module structure on $\Phi$.
\end{lemma}

\begin{proof}
  In order to simplify the notation, we write $\triangleright$ for $\triangleright_{\mathcal{M}}$ and $\blacktriangleright$ for $\triangleright_{\mathbf{EM}(T)}$ and $\triangleright_{\mathbf{EM}(S)}$, since each occurrence of either of these symbols is unambiguous with respect to which is used. We also write \(V \lact\) for \(V \lact \blank\), and suppress the horizontal composition symbols in $\mathbf{Cat}_{\Bbbk}$, replacing them simply by concatenation. Hence, every occurrence of $\circ$ in the diagrams that follow is a vertical composition of natural transformations. For example,
\[
  S(W\lact)S(V\lact)\Phi T
  \xrightarrow{S(W\lact)(\Phi_{\mathsf{la}}T(V\lact)\circ S{(\Phi T)}_{\mathsf{a}, V})}
  S(W\lact)\Phi T(V \lact)
\]
represents the natural transformation
\[
  S(W \lact S(V \lact \Phi T(\blank)))
  \xrightarrow{S(W \lact S{(\Phi T)}_{\mathsf{a}; W, V})}
  S(W \lact S \Phi T (V \lact \blank))
  \xrightarrow{S(W \lact \Phi_{\mathsf{la}}T(V \lact \blank))}
  S(W \lact \Phi T (V \lact \blank)).
\]

Since $\Phi_{\mathsf{a}}$ is uniquely determined by ${(\Phi_{\mathsf{a}})}_{T(-)}$, it suffices to show that the diagram
\[
  \begin{mytikzcd}[ampersand replacement=\&,cramped]
    {W\blacktriangleright \Phi(V\blacktriangleright T)} \&\&\& {\Phi(W\blacktriangleright V\blacktriangleright T)} \\
    {W\blacktriangleright V \blacktriangleright \Phi T} \& {W\blacktriangleright\Phi T (V\triangleright )} \& {\Phi T(W\triangleright)(V\triangleright)} \\
    {W\otimes V\blacktriangleright\Phi T} \& {\Phi T (W\otimes V)\triangleright} \&\& {\Phi((W\otimes V)\blacktriangleright T)}
    \arrow["{\Phi_{\mathsf{a}; W, V \blact T}}", from=1-1, to=1-4]
    \arrow["{\Phi(\alpha_2)}", from=1-4, to=3-4]
    \arrow["{W \blact \Phi_{\mathsf{a};V,T}}", from=2-1, to=1-1]
    \arrow["{{\alpha_{1}}}", from=2-1, to=2-2]
    \arrow["{{\alpha_{2}}}"', from=2-1, to=3-1]
    \arrow[Rightarrow, no head, from=2-2, to=1-1]
    \arrow["{{\alpha_{3}}}", from=2-2, to=2-3]
    \arrow[Rightarrow, no head, from=2-3, to=1-4]
    \arrow["{{\alpha_{5}}}"{description}, from=2-3, to=3-2]
    \arrow["{{\alpha_{4}}}", from=3-1, to=3-2]
    \arrow[Rightarrow, no head, from=3-2, to=3-4]
  \end{mytikzcd}
\]
commutes, where
$\alpha_{1} \defeq W \blacktriangleright \Phi_{\mathsf{a},V}$, $\alpha_{3} \defeq \Phi_{\mathsf{a},W}(V\triangleright)$, $\alpha_{4} \defeq \Phi_{\mathsf{a},W\otimes V}$ and $\alpha_{5} \defeq \Phi T(\triangleright_{\mathsf{a},W,V})$. The morphism $\alpha_{2}$ is obtained from the following diagram, in which every marked epimorphism is the coequalizer of the pair preceding it, every dashed arrow is an induced morphism between coequalizers coming from a morphism of diagrams, and the dotted arrows come from the universal property of coequalizers:
\[
  \begin{mytikzcd}[ampersand replacement=\&,cramped]
    {S(W\triangleright)SS(V\triangleright)S\Phi T} \&\&\& {S(W\triangleright )SS(V\triangleright)\Phi T} \& {S(W\triangleright)S(V\blacktriangleright \Phi T)} \\
    \\
    {S(W\triangleright )S(V\triangleright )S\Phi T} \&\&\& {S(W\triangleright )S(V\triangleright )\Phi T} \& {S(W\triangleright)(V \blacktriangleright \Phi T)} \\
    \&\&\&\& {W\blacktriangleright V \blacktriangleright \Phi T} \\
    {S(W\triangleright )(V\triangleright )S\Phi T} \&\&\& {S(W\triangleright )(V\triangleright )\Phi T} \& {(W,V)\blacktriangleright \Phi T} \\
    \\
    {S(W\otimes V\triangleright)S\Phi T} \&\&\& {(S(W\otimes V)\triangleright)\Phi T} \& {(W \otimes V) \blacktriangleright \Phi T}
    \arrow["{{S(W\triangleright)S(\mu \circ SS_{\mathsf{a},V})\Phi T}}", shift left=2, from=1-1, to=1-4]
    \arrow["{{S(W\triangleright)SS(V\triangleright)\Phi_{\mathsf{la}} T}}"', shift right=2, from=1-1, to=1-4]
    \arrow["{{(\mu \circ S_{\mathsf{a},W})S(V\triangleright )S\Phi T}}"', shift right=3, from=1-1, to=3-1]
    \arrow["{{S(W\triangleright)\mu (V\triangleright )S\Phi T}}", shift left=3, from=1-1, to=3-1]
    \arrow[two heads, from=1-4, to=1-5]
    \arrow["{{(\mu\circ S_{\mathsf{a,W}})S(V\triangleright )\Phi T}}"', shift right=3, from=1-4, to=3-4]
    \arrow["{{S(W\triangleright)\mu (V\triangleright)\Phi T}}", shift left=3, from=1-4, to=3-4]
    \arrow[shift right=3, dashed, from=1-5, to=3-5]
    \arrow[shift left=3, dashed, from=1-5, to=3-5]
    \arrow["{{S(W\triangleright)(\mu \circ SS_{\mathsf{a,V}})\Phi T}}", shift left=2, from=3-1, to=3-4]
    \arrow["{{S(W\triangleright)S(V\triangleright)\Phi_{\mathsf{la}}T}}"', shift right=2, from=3-1, to=3-4]
    \arrow["{{(\mu \circ S_{\mathsf{a},W})(V\lact)S\Phi T}}"', two heads, from=3-1, to=5-1]
    \arrow[two heads, from=3-4, to=3-5]
    \arrow["{{(\mu\circ S_{\mathsf{a},W})(V\lact)\Phi T}}"', two heads, from=3-4, to=5-4]
    \arrow[two heads, from=3-5, to=4-5]
    \arrow["{{\exists! \beta_{0}}}"', dotted, from=4-5, to=5-5]
    \arrow["\simeq", dotted, from=4-5, to=5-5]
    \arrow["{{\alpha_{2}}}", shift left=2, curve={height=-30pt}, dotted, from=4-5, to=7-5]
    \arrow["{{S(\mu \circ (S_{\mathsf{a},W}(V\triangleright)) \circ ((W\triangleright)S_{\mathsf{a},V}))\Phi T}}", shift left=2, from=5-1, to=5-4]
    \arrow["{{S(W\triangleright )(V\triangleright )\Phi_{\mathsf{la}} T}}"', shift right=2, from=5-1, to=5-4]
    \arrow["{{S(\triangleright_{\mathsf{a},W,V})S\Phi T}}"', from=5-1, to=7-1]
    \arrow[two heads, from=5-4, to=5-5]
    \arrow["{{S(\triangleright_{\mathsf{a},W,V})\Phi T}}", from=5-4, to=7-4]
    \arrow["{{\exists! \beta_{1}}}"', dotted, from=5-5, to=7-5]
    \arrow["\simeq", dotted, from=5-5, to=7-5]
    \arrow["{{S(\mu\circ S_{\mathsf{a},W\otimes V})\Phi T}}", shift left=2, from=7-1, to=7-4]
    \arrow["{{S(W\otimes V\triangleright)\Phi_{\mathsf{la}} T}}"', shift right=2, from=7-1, to=7-4]
    \arrow[two heads, from=7-4, to=7-5]
  \end{mytikzcd}
\]
The notation for the coequalizer $(W,V) \blacktriangleright \Phi T$ is to emphasize the connection with the multiactegorical approach of Section~\ref{multicategorical}.

We now have the diagram
\[\scalebox{0.89}{\begin{mytikzcd}[ampersand replacement=\&]
	{S(W\triangleright )S(V\triangleright )\Phi T} \&\&\& {S(W\triangleright)\Phi T(V\triangleright)} \\
	\& {S(W\triangleright)(V\blacktriangleright\Phi T)} \\
	{S(W\triangleright )(V\triangleright )\Phi T} \& {W\blacktriangleright V \blacktriangleright\Phi T} \&\& {W\blacktriangleright\Phi T(V\triangleright)} \& {\Phi T (W\triangleright)(V\triangleright)} \\
	\& {(W \otimes V)\blacktriangleright \Phi T} \\
	{S((W\otimes V)\triangleright )\Phi T} \&\&\&\& {\Phi T ((W\otimes V)\triangleright)}
	\arrow[""{name=0, anchor=center, inner sep=0}, "{S(W\triangleright)(\Phi_{\mathsf{la}}T(V\triangleright)\circ S{(\Phi T)}_{\mathsf{a},V})}", from=1-1, to=1-4]
	\arrow[two heads, from=1-1, to=2-2]
	\arrow[""{name=1, anchor=center, inner sep=0}, "{(\mu(W\lact) \circ S_{\mathsf{a},W}) (V\triangleright )\Phi T}"', from=1-1, to=3-1]
	\arrow[two heads, from=1-4, to=3-4]
	\arrow[""{name=2, anchor=center, inner sep=0}, "{((\Phi_{\mathsf{la}}(W\triangleright))\circ (S{(\Phi T)}_{\mathsf{a},W}))(V\triangleright)}", from=1-4, to=3-5]
	\arrow[""{name=3, anchor=center, inner sep=0}, "{\alpha'_{1}}"{description}, dashed, from=2-2, to=1-4]
	\arrow[two heads, from=2-2, to=3-2]
	\arrow["{S(\triangleright_{\mathsf{a},W,V})\Phi T}"', from=3-1, to=5-1]
	\arrow[""{name=4, anchor=center, inner sep=0}, "{\alpha_{1}}", from=3-2, to=3-4]
	\arrow[""{name=5, anchor=center, inner sep=0}, "{\alpha_{2}}", from=3-2, to=4-2]
	\arrow["{\alpha_{3}}", from=3-4, to=3-5]
	\arrow["{\alpha_{5}}"{description}, from=3-5, to=5-5]
	\arrow["{\alpha_{4}}"{description}, from=4-2, to=5-5]
	\arrow[two heads, from=5-1, to=4-2]
	\arrow[""{name=6, anchor=center, inner sep=0}, "{(\Phi_{\mathsf{la}}((W\otimes V)\triangleright))\circ (S{(\Phi T)}_{\mathsf{a},W\otimes V})}"', from=5-1, to=5-5]
	\arrow["{(2)}"{description}, draw=none, from=1, to=5]
	\arrow["{(1')}"{description}, draw=none, from=2-2, to=0]
	\arrow["{(1)}"{description}, draw=none, from=4, to=3]
	\arrow["{(3)}"{description}, draw=none, from=3-4, to=2]
	\arrow["{(4)}"{description}, draw=none, from=4-2, to=6]
\end{mytikzcd}}\]
whose labelled faces commute, and where the morphism decorated by the label of the face it is part of is defined as that making said face commute, via the universal property of coequalizers.
Our aim is to show the commutativity of the inner unlabelled face.
Since all of its morphisms are defined by the remaining inner (commutative) faces,
and we can reach \(W \blact V \blact \Phi T\) from \(S(W \lact)S(V \lact)\Phi T\) with two epimorphisms,
this is implied by the commutativity of the outer face,
which follows by the commutativity of the inner faces of the following diagram:
\[\scalebox{0.81}{\begin{mytikzcd}[ampersand replacement=\&]
	{S(W\triangleright )S(V\triangleright )\Phi T} \&\& {S(W\triangleright )S\Phi T(V\triangleright)} \&\&\&\& {S(W\triangleright)\Phi T(V\triangleright)} \\
	\\
	{SS(W\triangleright )(V\triangleright )\Phi T} \&\& {SS(W\triangleright )\Phi T(V\triangleright)} \&\& {SS\Phi T (W\triangleright)(V\triangleright)} \&\& {S\Phi T (W\triangleright)(V\triangleright)} \\
	\\
	{S(W\triangleright )(V\triangleright )\Phi T} \&\& {S(W\triangleright )\Phi T(V\triangleright)} \&\& {S\Phi T (W\triangleright)(V\triangleright)} \&\& {\Phi T (W\triangleright)(V\triangleright)} \\
	\\
	{S((W\otimes V)\triangleright )\Phi T} \&\&\&\& {S\Phi T ((W\otimes V)\triangleright)} \&\& {\Phi T ((W\otimes V)\triangleright)}
	\arrow[""{name=0, anchor=center, inner sep=0}, "{S(W\triangleright )S{(\Phi T)}_{\mathsf{a},V}}", from=1-1, to=1-3]
	\arrow["{SS_{\mathsf{a},W} (V\triangleright )\Phi T}"', from=1-1, to=3-1]
	\arrow[""{name=1, anchor=center, inner sep=0}, "{S(W\triangleright)\Phi_{\mathsf{la}}T(V\triangleright)}", from=1-3, to=1-7]
	\arrow["{SS_{\mathsf{a},W}\Phi T(V\triangleright)}"', from=1-3, to=3-3]
	\arrow["{S{(\Phi T)}_{\mathsf{a},W}(V\triangleright)}", from=1-7, to=3-7]
	\arrow[""{name=2, anchor=center, inner sep=0}, "{SS(W\triangleright){(\Phi T)}_{\mathsf{a},V}}", from=3-1, to=3-3]
	\arrow["{\mu (W \triangleright)(V\triangleright)\Phi T}"', from=3-1, to=5-1]
	\arrow[""{name=3, anchor=center, inner sep=0}, "{SS{(\Phi T)}_{\mathsf{a},W}(V \triangleright )}"', from=3-3, to=3-5]
	\arrow["{\mu (W\triangleright)\Phi T(V\triangleright)}"', from=3-3, to=5-3]
	\arrow[""{name=4, anchor=center, inner sep=0}, "{S\Phi_{\mathsf{la}}T(W\triangleright)(V\triangleright)}"', from=3-5, to=3-7]
	\arrow["{\mu\Phi T (W\triangleright)(V\triangleright)}"', from=3-5, to=5-5]
	\arrow["{\Phi_{\mathsf{la}}T(W\triangleright)(V\triangleright)}", from=3-7, to=5-7]
	\arrow[""{name=5, anchor=center, inner sep=0}, "{S(W\triangleright){(\Phi T)}_{\mathsf{a},V}}"', from=5-1, to=5-3]
	\arrow["{S(\triangleright_{\mathsf{a},W,V})\Phi T}"', from=5-1, to=7-1]
	\arrow[""{name=6, anchor=center, inner sep=0}, "{S{(\Phi T)}_{\mathsf{a},W}(V\triangleright)}"', from=5-3, to=5-5]
	\arrow[""{name=7, anchor=center, inner sep=0}, "{\Phi_{\mathsf{la}} T (W\triangleright)(V\triangleright)}"', from=5-5, to=5-7]
	\arrow["{S\Phi T(\triangleright_{\mathsf{a},W,V})}"', from=5-5, to=7-5]
	\arrow["{\Phi T(\triangleright_{\mathsf{a},W,V})}", from=5-7, to=7-7]
	\arrow[""{name=8, anchor=center, inner sep=0}, "{S{(\Phi T)}_{\mathsf{a},W\otimes V}}"', from=7-1, to=7-5]
	\arrow[""{name=9, anchor=center, inner sep=0}, "{\Phi_{\mathsf{la}}T((W\otimes V)\triangleright)}"', from=7-5, to=7-7]
	\arrow["{(1)}"{description}, shift right=5, draw=none, from=0, to=2]
	\arrow["{(2)}"{description}, draw=none, from=1, to=3-5]
	\arrow["{(3)}"{description}, shift right=5, draw=none, from=2, to=5]
	\arrow["{(4)}"{description}, shift right=5, draw=none, from=3, to=6]
	\arrow["{(5)}"{description}, draw=none, from=4, to=7]
	\arrow["{(6)}"{description}, draw=none, from=5-3, to=8]
	\arrow["{(7)}"{description}, draw=none, from=7, to=9]
\end{mytikzcd}}\]
where
\begin{itemize}
 \item faces (1), (3), and (4) commute by the interchange law in the $2$-category $\mathbf{Cat}_{\Bbbk}$;
 \item face (2) commutes by $\Phi_{\mathsf{la}}\from S \Phi T \nt \Phi T$ being a $\mathcal{C}$-module transformation;
 \item face (5) commutes by the associativity of action $\Phi_{\mathsf{la}}$;
 \item face (6) commutes by $\Phi T$ being a lax $\mathcal{C}$-module functor; and
 \item face (7) commutes by naturality of $\Phi_{\mathsf{la}}$.
\end{itemize}
\end{proof}

\begin{lemma}\label{jupiter}
 Let\, $\Phi, \Phi' \in \mathbf{LaxRex}(\mathbf{EM}(T),\mathbf{EM}(S))$ and $\phi\from\! R_{\mathbf{EM}(S)} \circ \Phi \circ L_{\mathbf{EM}(T)} \nt R_{\mathbf{EM}(S)} \circ \Phi' \circ L_{\mathbf{EM}(T)}$ be a lax $\mathcal{C}$-module biact transformation. Then $\varphi \defeq \phi \circ_{T} \blank \from \Phi \nt \Phi'$ is a $\mathcal{C}$-module transformation.
\end{lemma}

\begin{proof}
 First, observe that $\varphi$ is determined by $\varphi_{T(\blank)}$, as indicated by the diagram
\[\begin{mytikzcd}[ampersand replacement=\&]
	{\Phi T^{2}X} \& {\Phi TX} \& {\Phi X} \\
	{\Phi' T^{2}X} \& {\Phi' TX} \& {\Phi' X}
	\arrow["{\Phi\mu_{X}}"', shift right, from=1-1, to=1-2]
	\arrow["{\Phi T\nabla_{X}}", shift left, from=1-1, to=1-2]
	\arrow["{\varphi_{T(T(X))}}"', from=1-1, to=2-1]
	\arrow[two heads, from=1-2, to=1-3]
	\arrow["{\varphi_{T(X)}}", from=1-2, to=2-2]
	\arrow["{\exists! \varphi_{X}}", dashed, from=1-3, to=2-3]
	\arrow["{\Phi' T\nabla_{X}}", shift left, from=2-1, to=2-2]
	\arrow["{\Phi'\mu_{X}}"', shift right, from=2-1, to=2-2]
	\arrow[two heads, from=2-2, to=2-3]
\end{mytikzcd}\]
Thus, it suffices to show that for any $X \in \mathcal{M}$ and $V \in \mathcal{C}$, we have that
\begin{equation}\label{laxplax}
\Phi'_{\mathsf{a},V,T(X)} \circ (V \blacktriangleright\phi_{X}) = \phi_{V\triangleright X} \circ \Phi_{\mathsf{a},V,T(X)}.
\end{equation}
Consider the following diagram:
\[\begin{mytikzcd}[ampersand replacement=\&]
	{S(V\triangleright \Phi T)} \&\&\& {S\Phi TV\triangleright \blank} \\
	\& {V \blacktriangleright \Phi T} \&\&\& {\Phi TV\triangleright \blank} \\
	{S(V\triangleright \Phi' T)} \&\&\& {S\Phi' TV\triangleright \blank} \\
	\& {V \blacktriangleright \Phi' T} \&\&\& {\Phi' TV\triangleright \blank}
	\arrow["{S{(\Phi T)}_{\mathsf{a},V}}", from=1-1, to=1-4]
	\arrow[two heads, from=1-1, to=2-2]
	\arrow["{S(V\triangleright \phi)}"', from=1-1, to=3-1]
	\arrow["{\Phi_{\mathsf{la}}TV\triangleright \blank}", from=1-4, to=2-5]
	\arrow["{S\phi'V\triangleright\blank}"{pos=0.8}, from=1-4, to=3-4]
	\arrow[crossing over, "{\Phi_{\mathsf{a},V,T(\blank)}}"{pos=0.4}, from=2-2, to=2-5]
	\arrow["{\phi_{V\triangleright \blank}}", from=2-5, to=4-5]
	\arrow["{S{(\Phi' T)}_{\mathsf{a},V}}"{pos=0.7}, from=3-1, to=3-4]
	\arrow[crossing over, "{V \blacktriangleright\phi_{\blank}}"'{pos=0.3}, from=2-2, to=4-2]
	\arrow[two heads, from=3-1, to=4-2]
	\arrow["{\Phi'_{\mathsf{la}}TV\triangleright \blank}", from=3-4, to=4-5]
	\arrow["{\Phi'_{\mathsf{a},V,T(\blank)}}"', from=4-2, to=4-5]
\end{mytikzcd}\]
Its top and bottom faces commute by the definition of $\Phi_{\mathsf{a}}$, as seen from Diagram~\ref{phiaeio}; its left face commutes by the definition of $V \blacktriangleright \blank$; and its right face by $\phi$ being a biact transformation. The back face commutes by $\phi$ being a $\mathcal{C}$-module transformation.
The commutativity of the front face is precisely Equation~\eqref{laxplax},
and since the top-left projection map is an epimorphism, it suffices to verify its commutativity after precomposing with it.
This latter commutativity follows easily from the commutativity of the remaining faces.
\end{proof}

\begin{theorem}\label{TurboEilenbergWatts}
  The faithful functor
  \[
    \mathcal{C}\mathsf{EW}\from \mathbf{LaxRex}(\mathbf{EM}(T), \mathbf{EM}(S)) \to \biactC[\cat{C}]{S}{T}
  \]
  of Theorem~\ref{ultraEilenbergWatts} is an equivalence of categories.
\end{theorem}
\begin{proof}
  It is left to prove that \(\cat{C}\mathsf{EW}\) is full and essentially surjective.
  For the latter case, suppose that \(F \in \biactC[\cat{C}]{T}{S}\).
  Since the functor of Proposition~\ref{EilenbergWatts} is essentially surjective,
  we may assume that \(F \cong R_{\mathbf{EM}(S)} \circ \Phi \circ L_{\mathbf{EM}(T)}\),
  for some right exact \(\Phi\from \mathbf{EM}(T) \to \mathbf{EM}(S)\).
  By Lemma~\ref{bigdiagrams}, \((\Phi, \Phi_{\mathsf{a}}) \in  \mathbf{LaxRex}(\mathbf{EM}(T),\mathbf{EM}(S))\),
  and \(F \cong \mathcal{C}\mathsf{EW}((\Phi, \Phi_{\mathsf{a}}))\) due to Lemma~\ref{phigter},
  establishing essential surjectivity.
  Fullness follows from Lemma~\ref{jupiter}.
\end{proof}

\begin{definition}
  We say that two right exact lax \(\mathcal{C}\)-module monads \(T\) and \(T'\) on \(\mathcal{M}\) are \emph{Morita equivalent} if there are
  \(F \in \biactC[\cat{C}]{T}{T'}\) and
  \(G \in \biactC[\cat{C}]{T'}{T}\)
  such that \(G \circ_{T} F \cong \on{Id}_{\mathbf{EM}(T)}\)
  and \(F \circ_{T'} G \cong \on{Id}_{\mathbf{EM}(T')}\)
  as lax \(\mathcal{C}\)-module biact functors.
\end{definition}

\begin{proposition}\label{Moritabybimodules}
  Two right exact lax $\mathcal{C}$-module monads $T$ and \(T'\) on $\mathcal{M}$ are Morita equivalent if and only if there is a $\mathcal{C}$-module equivalence $\mathbf{EM}(T) \simeq \mathbf{EM}(T')$,
  where the Eilenberg--Moore categories are endowed with the extended Linton coequalizer structure of Theorem~\ref{extendingalways}.
\end{proposition}

\begin{proof}
  This is a direct consequence of Theorem~\ref{TurboEilenbergWatts}:
  an equivalence of categories is right exact.
\end{proof}

Combining Proposition~\ref{Moritabybimodules} with Theorem~\ref{injcorrespondence}, we find the following result:

\begin{theorem}\label{injectivebijection}
  Let $\mathcal{C}$ be a monoidal abelian category with enough injectives, and let $K$ be a left exact lax $\mathcal{C}$-module comonad on $\mathcal{C}$. There is a bijection
  \begin{align*}
    \{\,(\mathcal{M},X) \text{ as in Theorem~\ref{mainnonsenseinjective}}\,\}/(\mathcal{M} \simeq \mathcal{N})
    &\xleftrightarrow{\ \simeq\,}
      \{\,\text{Left exact oplax }\mathcal{C}\text{-module comonads on } \mathcal{C}\,\}/\simeq_{\text{Morita}} \\
    (\mathcal{M},X) &\longmapsto \cohom{X,-\triangleright X} \\
    (\mathbf{EM}(K), K(\mathbb{1})) &\longmapsfrom K
  \end{align*}
\end{theorem}

\section{Hopf trimodules}\label{sec:Hopf-trimodules}

Fix a field \(\Bbbk\) and a bialgebra \(B\) over it.
Write \(\cat{V} \defeq \bbcomod\) for the category of left \(B\)-comodules.
After introducing all of the relevant concepts and notation,
our goal in this section is to prove the following theorem.

\begin{theorem}\label{thm:hopf-trimodules-are-lax-monoidal-functors}
  There is a monoidal equivalence
  \[
    \Trimod \simeq \mathbf{LexfLax}\cat{V}\text{-}\mathbf{Mod}(\cat{V}, \cat{V})
  \]
  between the category \(\Trimod\) of Hopf trimodules, and
  the category \(\mathbf{LexfLax}\cat{V}\text{-}\mathbf{Mod}(\cat{V}, \cat{V})\)
  of left exact, finitary lax \(\cat{V}\)-module functors from \(\cat{V}\) to \(\cat{V}\).
\end{theorem}

\subsection{Bicomodules and the graphical calculus}\label{sec:bicomodules}

Before introducing Hopf trimodules,
let us first talk about bicomodules over \(B\)
in the sense of Section~\ref{sec:monoids}
as well as their string diagrammatics.
The latter will be our main tool to prove
Theorem~\ref{thm:hopf-trimodules-are-lax-monoidal-functors}.

Explicitly, a bicomodule over \(B\) comprises a vector space \(X\),
a left coaction \(\lambda \from X \to B \kotimes X\),
and a right coaction \(\rho \from X \to X \kotimes B\), such that
\[
  (B \kotimes \rho) \circ \lambda = (\lambda \kotimes B) \circ \rho.
\]

Our convention is to read string diagrams from bottom to top and left to right.
Vertical gluing of strings represents the tensor product,
while horizontal gluing amounts to composition of morphisms.
For a comprehensive account, we refer to~\cite{Sel}.
As a first example,
let us depict the compatibility condition of a bicomodule graphically:\\[-0.5em]
\begin{equation}\label{eq:bicomodule-condition}
  \begin{tikzpicture}[style=tikzfig]
    \begin{pgfonlayer}{nodelayer}
      \node [style=none] (0) at (-1.75, -1.25) {};
      \node [style=none] (1) at (-1.75, 1.5) {};
      \node [style=none] (2) at (-2.75, 1.5) {};
      \node [style=none] (3) at (-0.75, 1.5) {};
      \node [style=none] (4) at (-1.75, 0.5) {};
      \node [style=none] (5) at (-1.75, -0.5) {};
      \node [style=none] (6) at (1.75, -1.25) {};
      \node [style=none] (7) at (1.75, 1.5) {};
      \node [style=none] (8) at (2.75, 1.5) {};
      \node [style=none] (9) at (0.75, 1.5) {};
      \node [style=none] (10) at (1.75, 0.5) {};
      \node [style=none] (11) at (1.75, -0.5) {};
      \node [style=none] (12) at (0, 0) {\scriptsize $=$};
      \node [style=none] (13) at (-2.75, 1.75) {\scriptsize $B$};
      \node [style=none] (14) at (-1.75, 1.75) {\scriptsize $X$};
      \node [style=none] (15) at (-0.75, 1.75) {\scriptsize $B$};
      \node [style=none] (16) at (2.75, 1.75) {\scriptsize $B$};
      \node [style=none] (17) at (0.75, 1.75) {\scriptsize $B$};
      \node [style=none] (18) at (1.75, 1.75) {\scriptsize $X$};
      \node [style=none] (19) at (1.75, -1.5) {\scriptsize $X$};
      \node [style=none] (20) at (-1.75, -1.5) {\scriptsize $X$};
    \end{pgfonlayer}
    \begin{pgfonlayer}{edgelayer}
      \draw [style=morphism-edge](0.center) to (1.center);
      \draw [style=morphism-edge, in=270, out=180] (5.center) to (2.center);
      \draw [style=morphism-edge, in=0, out=-90] (3.center) to (4.center);
      \draw [style=morphism-edge](6.center) to (7.center);
      \draw [style=morphism-edge, in=-90, out=0] (11.center) to (8.center);
      \draw [style=morphism-edge, in=180, out=-90] (9.center) to (10.center);
    \end{pgfonlayer}
  \end{tikzpicture}
\end{equation}\\[-0.7em]

\begin{example}\label{ex:cotensor-product-of-bicomodules}
  Given two bicomodules \((X, \lambda^X, \rho^X)\) and \((Y, \lambda^Y, \rho^Y)\) over \(B\),
  one may form their \emph{cotensor product};
  this is the bicomodule \(X \cotens Y\) given by the equalizer
  \[
    \begin{mytikzcd}[ampersand replacement=\&]
      {X \cotens Y} \& {X \kotimes Y} \& {X \kotimes B \kotimes Y.}
      \arrow[hook, from=1-1, to=1-2]
      \arrow["{\rho^X \kotimes Y}", shift left=1, from=1-2, to=1-3]
      \arrow["{{\raisebox{-0.8em}{\(X \kotimes \lambda^Y\)}}}"', shift right=1, from=1-2, to=1-3]
    \end{mytikzcd}
  \]
\end{example}

Our graphical calculus has to differentiate between
the tensor product of two bicomodules over the ground field
and their cotensor product.
In particular, we have to indicate which additional transformations are possible with the latter.
For this, we annotate equalised actions in the cotensor product in grey.\\
\[

\]\\
The fact that we may interchange the appropriate left and right actions will be indicated thusly:\\
\[
  %
\]\\[-0.7em]

\subsection{From Hopf trimodules to lax module functors}\label{sec:hopf-trimodules-lax-mod-funs}

\begin{definition}\label{def:hopf-tri-module}
  A \emph{Hopf trimodule} over \(B\) is a bicomodule \(X\)
  together with an action \(\alpha\),
  such that \(\alpha\) is a left and right \(B\)-comodule morphism.
  We write \(X \in \Trimod\).
\end{definition}

Alternatively, Definition~\ref{def:hopf-tri-module} could impose the conditions
that \(\lambda\) and \(\rho\) are left and right \(B\)-module morphisms, respectively.
This equivalence is easily seen in a string diagrammatic reformulation,
as given in Figure~\ref{fig:hopf-trimodule}.
\begin{figure}[htbp]
  \tikzsetnextfilename{fig-hopf-trimodule}

  \caption{}%
  \label{fig:hopf-trimodule}
\end{figure}

\begin{definition}\label{def:bicomodule-interchange}
  Let \(X \in \Trimod\).
  For all \(M, N \in \bbcomod \), define the \emph{interchange morphism}
  \[
    \chi_{M, N} \from M \kotimes (X \cotens N) \to X \cotens (M \kotimes N)
  \]
  by\\
  \[
    %
  \]\\[-0.7em]
\end{definition}

\begin{remark}\label{rmk:interchange-unnessary-cotens}
  Observe in particular that\\
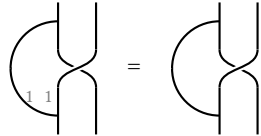
  \[
    %
  \]\\
  because the image of the coaction \(M \to B \kotimes M\) is contained in the cotensor product \( B \cotens M\).
\end{remark}

The interchange morphism in Definition~\ref{def:bicomodule-interchange}
is akin to the \emph{Yetter--Drinfeld braiding},
which endows the category of Yetter--Drinfeld modules with the structure of a braided monoidal category;
see~\cite[Theorem~7.2]{Yet}.

\begin{lemma}\label{lem:interchange-properties}
  The interchange morphism
  is well-defined and a left \(B\)-comodule morphism.
\end{lemma}
\begin{proof}
  To show well-definedness,
  we need to show that the image of \(\chi_{M, N}\) is contained in \(X \cotens (M \kotimes N)\),
  for all \(M, N \in \bbcomod \).
  This follows by the following calculation\\
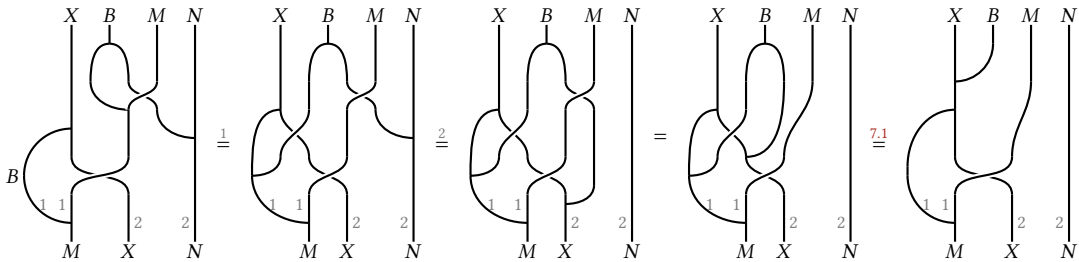
  \[
    \tikzexternaldisable

    \tikzexternalenable
  \]\\

  Graphically, the fact that \(\chi\) is a left \(B\)-comodule morphism means that\\
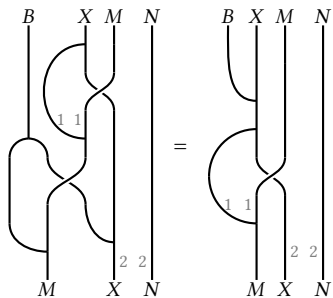
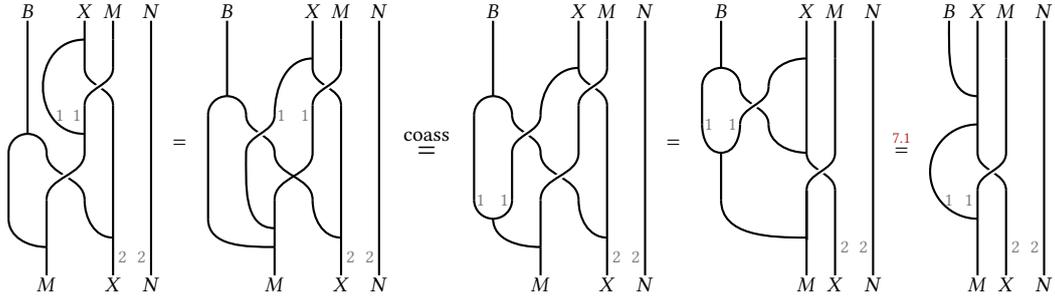
  \[
    %
    \tikzexternalenable
  \]\\
\end{proof}

\begin{lemma}\label{lem:interchange-is-braiding}
  The interchange morphism is a braiding; i.e.,
  it is natural in both variables, and the following diagrams commute,
  for all \(M, N, P \in \bbcomod \):
  \begin{equation} \label{eq:interchange-yb}
    \begin{mytikzcd}[ampersand replacement=\&]
      {M \kotimes (N \kotimes (X \cotens P))} \&\& {M \kotimes (X \cotens (N \kotimes P))} \\
      {(M \kotimes N) \kotimes (X \cotens P)} \& {X \cotens ((M \kotimes N) \kotimes P)} \& {X \cotens (M \kotimes (N \kotimes P))}
      \arrow["{M \kotimes \chi_{N, P}}", from=1-1, to=1-3]
      \arrow["{\chi_{M, N \kotimes P}}", from=1-3, to=2-3]
      \arrow["\alpha"', from=1-1, to=2-1]
      \arrow["{\chi_{M \kotimes N, P}}"', from=2-1, to=2-2]
      \arrow["\alpha"', from=2-2, to=2-3]
    \end{mytikzcd}
  \end{equation}
  \begin{equation} \label{eq:interchange-unital}
    \begin{mytikzcd}[ampersand replacement=\&]
      {\Bbbk \kotimes (X \cotens M)} \&\& {X \cotens (\Bbbk \kotimes M)} \\
      \& {X \cotens N}
      \arrow["{\chi_{\Bbbk,N}}", from=1-1, to=1-3]
      \arrow["{X \cotens l_N}", from=1-3, to=2-2]
      \arrow["{l_{X \cotens M}}"', from=1-1, to=2-2]
    \end{mytikzcd}
  \end{equation}
\end{lemma}
\begin{proof}
  It is immediate that \(\chi\) is natural in its second variable.
  To prove naturality in the first variable,
  let \(f \from M \to M'\) be a left comodule morphism.
  Then\\
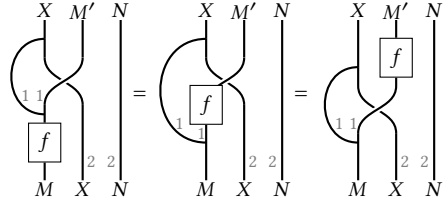
  \[

  \]\\
  where the first equality follows by \(f\) being a left comodule morphism,
  and the second one is the naturality of the braiding.

  Diagram~\eqref{eq:interchange-yb} commuting is equivalent to\\
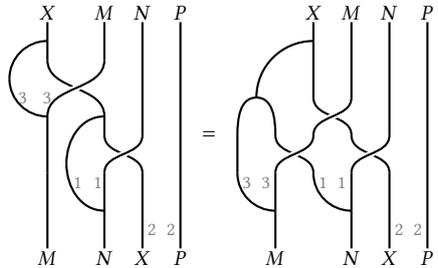
  \[
    %
  \]\\
  This is immediately seen to be true by associativity of the action on \(X\).
  Diagram~\eqref{eq:interchange-unital} follows by the unitality of the action on \(X\).
\end{proof}

In particular, Lemmas~\ref{lem:interchange-properties} and~\ref{lem:interchange-is-braiding}
taken together say that the well-defined natural transformation
\[
  \chi\from \blank \kotimes (X \cotens \bblank) \nt X \cotens (\blank \kotimes \bblank)
\]
satisfies Diagrams~\eqref{eq:interchange-yb} and~\eqref{eq:interchange-unital}.
This, in turn, yields the following result.

\begin{proposition}\label{prop:interchange-comodule-functor}
  The pair \((X \cotens \blank, \chi)\) defines a lax \( \bbcomod \)-module functor.
\end{proposition}

Building on the above proposition,
we may extend this correspondence to morphisms.

\begin{lemma}\label{lem:morphism-of-hopf-modules->module-transformation}
  Let \(f \in \Trimod(X, Y)\) be a morphism of Hopf trimodules.
  Then
  \[
    f \cotens \blank \from X \cotens \blank \nt Y \cotens \blank
  \]
  is a \( \bbcomod \)-module transformation.
\end{lemma}
\begin{proof}
  We have to prove that
  \[
    \big(f \cotens (M \kotimes N)\big) \circ \chi^X_{M, N}
    =
    \chi^Y_{M, N} \circ \big(M \kotimes (f \cotens N)\big).
  \]
  In our graphical language, this means that\\
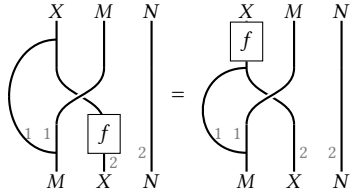
  \[

  \]\\
  which follows immediately from \(f\) being a module morphism.
\end{proof}

Recall that \(\cat{V} \defeq \bbcomod \).
As a result of the previous considerations,
there is a well-defined functor
\begin{equation} \label{eq:hopf-trimod-to-lax-module-functors--functor-definition}
  \Sigma \from \Trimod \to \mathbf{LexfLax}\cat{V}\text{-}\mathbf{Mod}(\cat{V}, \cat{V}), \qquad X \mapsto (X \cotens \blank, \chi)
\end{equation}
To finish the proof of Theorem~\ref{thm:hopf-trimodules-are-lax-monoidal-functors},
we have to show that \(\Sigma\) is monoidal, as well as an equivalence of categories.

\begin{lemma}\label{lem:cotens-of-hopf-modules-is-hopf-module}
  Let \(X\) and \(Y\) be two Hopf trimodules.
  Their cotensor product \(X \cotens Y\) is again a Hopf trimodule,
  where the left action is given diagonally by \(b(x \otimes y) \defeq b_{(1)}x \otimes b_{(2)}y\).\footnote{\,%
    Here, we use \emph{Sweedler notation} and write \(\Delta(b)\) as \(b_{(1)} \otimes b_{(2)}\).%
  }
\end{lemma}
\begin{proof}
  First, we show that the action is well-defined as a map from \(B \kotimes (X \cotens Y) \to X \cotens Y\); i.e.,\\
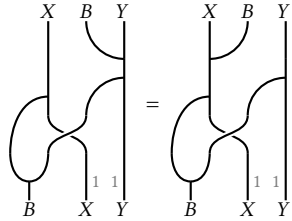
  \[

  \]\\
  This follows by the calculation in Figure~\ref{fig:cotens-of-hopf-modules-is-hopf-module-proof}.
  \begin{figure}[htbp]
    \tikzexternaldisable
    %
    \tikzexternalenable
    \caption{}%
    \label{fig:cotens-of-hopf-modules-is-hopf-module-proof}
  \end{figure}

  It is left to show that the action is a left and right \(B\)-comodule morphism.
  The former case is\\
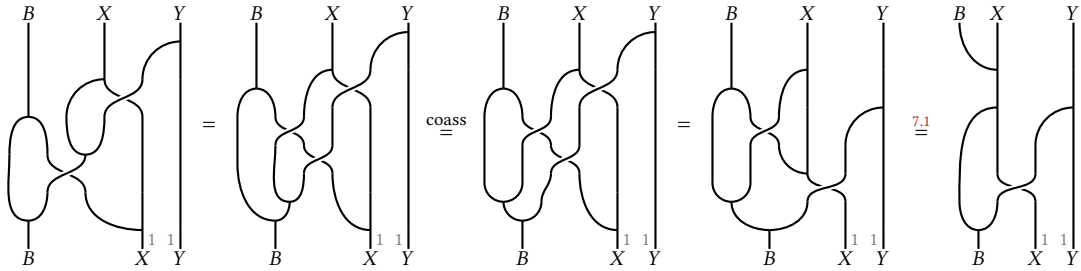
  \[
    \tikzexternaldisable
    %
    \tikzexternalenable
  \]\\
  and in the latter case one calculates\\
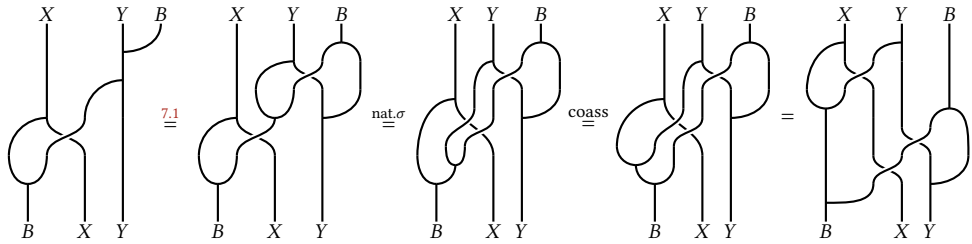
  \[
    \tikzexternaldisable
    %
    \tikzexternalenable
  \]\\
\end{proof}

\begin{proposition}\label{prop:braiding-composes-to-cotens}
  Let \(\chi^X\) and \(\chi^{Y}\) be as in Definition~\ref{def:bicomodule-interchange}.
  Then
  \[
    \chi^X \diamond \chi^Y \cong \chi^{X \cotens Y},
  \]
  where \(\diamond\) is the multiplicative cell for lax module morphisms; i.e.,
  \[
    \big(X \cotens \blank, \chi^X\big) \diamond \big(Y \cotens \blank, \chi^Y\big) = \big(X \cotens Y \cotens \blank, \chi^X \diamond \chi^Y\big).
  \]
  In other words, the functor \(\Sigma\)
  from Equation~\eqref{eq:hopf-trimod-to-lax-module-functors--functor-definition}
  is monoidal.
\end{proposition}
\begin{proof}
  Associativity follows from coassociativity of the coaction and naturality of the braiding.\\
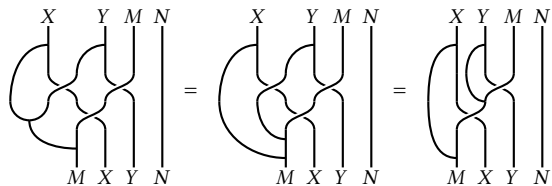
  \[

  \]\\[-0.7em]

  Unitality is immediate.
\end{proof}

\subsubsection{Proof of Theorem~\ref{thm:hopf-trimodules-are-lax-monoidal-functors}}

It is left to show that \(\Sigma\)
from Equation~\eqref{eq:hopf-trimod-to-lax-module-functors--functor-definition}
is an equivalence.

\begin{proposition}\label{prop:sigma-fully-faithful}
  The functor \(\Sigma\) is fully-faithful.
\end{proposition}
\begin{proof}
  To show that \(\Sigma\) is faithful,
  suppose that \(f, g \from X \to Y\) are morphisms of Hopf trimodules,
  such that \(f \cotens \blank = g \cotens \blank\).
  Then
  \[
    \begin{mytikzcd}[ampersand replacement=\&]
      {X \cotens B} \& {Y \cotens B} \\
      X \& Y
      \arrow["{f \cotens B}", shift left=1, from=1-1, to=1-2]
      \arrow["g\cotens B"', shift right=1, from=1-1, to=1-2]
      \arrow["\cong"', from=1-1, to=2-1]
      \arrow["\cong", from=1-2, to=2-2]
      \arrow["f", shift left=1, from=2-1, to=2-2]
      \arrow["g"', shift right=1, from=2-1, to=2-2]
    \end{mytikzcd}
  \]
  commutes, so the result follows.

  For the claim that \(\Sigma\) is full,
  suppose that \(\varphi \from X \cotens \blank \nt Y \cotens \blank\) is a \(\cat{V}\)-module transformation.
  By Proposition~\ref{TakeuchiEilenbergWatts},
  we need only show that the induced arrow
  \begin{equation}\label{eq:sigma-full-induced-morphism}
    \overline{\varphi_B} \from X \cong X \cotens B \xrightarrow{\varphi_B} Y \cotens B \cong Y
  \end{equation}
  is a morphism of Hopf trimodules.
  For \(M \in \Trimod \),
  the morphism \(\varphi_M \from X \cotens M \to Y \cotens M\)
  and the induced morphism~\eqref{eq:sigma-full-induced-morphism}
  can, in string diagrams,
  be depicted like this:\\
  \[

  \]\\

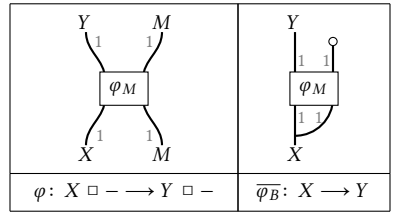
  Figure~\ref{fig:sigma-phi-properties} collects several properties of \(\chi\) and \(\varphi\) in this notation.
  The arrow \(\overline{\varphi_B}\) is a left and right comodule morphism by the calculations in
  Figure~\ref{fig:sigma-phi-left-comodule-morphism}.
  \begin{figure}[htbp]
    \tikzexternaldisable
    %
    \tikzexternalenable
    \caption{}%
    \label{fig:sigma-phi-properties}
  \end{figure}
  \begin{figure}[htbp]
    \tikzexternaldisable
    %
    \tikzexternalenable
    \caption{}%
    \label{fig:sigma-phi-left-comodule-morphism}
  \end{figure}

  Likewise, that \(\varphi_B\) is a left module morphism is straightforward.\\
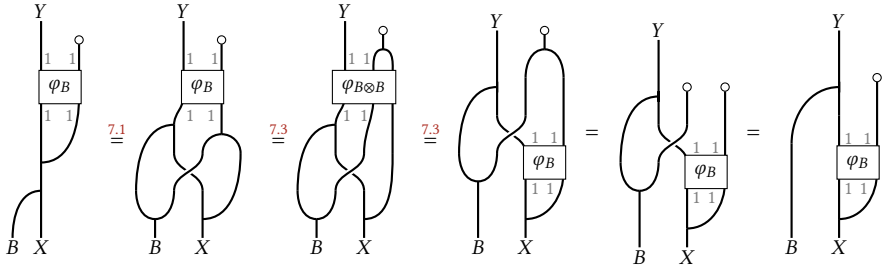
  \[
    \tikzexternaldisable
    %
    \tikzexternalenable
  \]
\end{proof}

\begin{proposition}\label{prop:sigma-essentially-surjective}
  The functor \(\Sigma\) is essentially surjective.
\end{proposition}
\begin{proof}
  Suppose that \(\Phi \in \mathbf{LexfLax}\cat{V}\text{-}\mathbf{Mod}(\cat{V},\cat{V})\).
  Since \(\Phi\) is left exact and finitary,
  by Proposition~\ref{TakeuchiEilenbergWatts},
  there exists a \(B\)-\(B\)-bicomodule \(X\),
  such that \(\Phi \cong X \cotens \blank\) as functors.
  We will show that if \(X \cotens \blank \) is a lax module morphism,
  then \(X\) comes equipped with a left module structure,
  making it into a Hopf trimodule.

  Define a \(B\)-action on \(X\) in the following way:
  \[

  \]
  Above, we have chosen to represent \(\chi_{B,B}\)%
  —see Definition~\ref{def:bicomodule-interchange}—%
  by a single box.
  The claim is that this action turns \(X\) into a Hopf trimodule.

  For ease of notation, we will omit the grey action indicators in our string diagrammatic proofs,
  save when transforming the diagrams using the respective relations.

  To show that \(\alpha\) is a bicomodule morphism,
  it suffices to note that the left comodule morphism \(\chi_{B,B}\) is also a right comodule morphism,
  as \(\alpha\) is then composed of bicomodule morphisms.

  Notice that \(\chi\) is a natural transformation of functors which are left exact finitary in each variable, from
  \(\bbcomod \kotimes \bbcomod\) to \(\bbcomod\).
  Since \(\chi_{B,B}\) is the component of the transformation
  \[
    \chi_{B,\blank}\from B \kotimes (X \cotens \blank) \nt X \cotens (B \kotimes \blank)
  \]
  taken at the injective generator of \( \bbcomod \),
  it follows that it is a right comodule morphism;
  see Propositions~\ref{prop:left exact-is-enough-on-injectives}
  and~\ref{TakeuchiEilenbergWatts}.
  Similarly, \(\chi_{B,B}\) is the component of \( \chi_{\blank,B}\from \blank \kotimes (X \cotens B) \nt X \cotens (\blank \kotimes B) \) taken at the injective generator,
  so \(\chi_{B,B}\) also intertwines this additional comodule structure:
  \[

  \]

  Let us now verify that \(\alpha\) is in fact a \(B\)-module morphism.
  For the multiplicative identity of \(\alpha\), one calculates
  \[
    %
  \]
  where \((i)\) uses the multiplicative identity of \((X \cotens \blank, \chi)\) being a \(\bbcomod\)-module functor.

  The unitality of \(\alpha\) follows from naturality of \(\chi\),
  the unital axiom of \((X \cotens \blank, \chi)\) being a \(\bbcomod\)-module functor,
  and unitality conditions of \(B\) and \(X\):
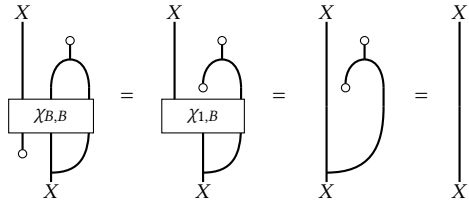
  \[
    %
  \]

  It is left to show that \(\Sigma(\alpha) = \chi_{B,B}\).
  The necessary calculation is given in Figure~\ref{fig:sigma-essentially-surjective-hit},
  where equation \((i)\) holds because \(\chi_{B,B}\) is a right comodule morphism in both variables,
  and \((ii)\) follows by naturality of \(\chi\).
  \begin{figure}[htbp]
    \tikzsetnextfilename{fig-sigma-essentially-surjective-hit}
    %
    \caption{}%
    \label{fig:sigma-essentially-surjective-hit}
  \end{figure}
\end{proof}

\subsection{Reconstruction for bialgebras via contramodules over Hopf trimodule algebras}\label{triconstruction}

\begin{definition}
  A \emph{Hopf trimodule algebra} comprises an algebra object in the monoidal category $\big(\Trimod, \cotens, B\big)$.
  As such it consists of a Hopf trimodule $A \in \Trimod$,
  together with trimodule morphisms $\mu\from A \cotens A \to A$ and $\eta\from B \to A$,
  satisfying associativity and unitality conditions.

  A \emph{left module over a Hopf trimodule algebra} is a left $A$-module object in the $\Trimod$-module category $\tetramod[B]{}$ in the sense of Definition~\ref{ModuleInModule}.
  As such, it comprises a left $B$-comodule $M$ together with a $B$-comodule homomorphism $A \cotens M \to M$, satisfying associativity and unitality conditions with respect to $\mu$ and $\eta$.
  Similarly we define right modules over $A$.
  A module over $A$ is \emph{free} if it is of the form $(A \cotens N, \mu \cotens N)$.
\end{definition}

\begin{definition}
  We refer to objects of the category $\Termod$ as \emph{Hopf termodules}. A \emph{Hopf termodule coalgebra} is a coalgebra object $C$ in the monoidal category $\big(\Termod, \otimes_{B}, B\big)$,
  and a \emph{comodule over $C$} is a left $B$-module $N$ together with a $B$-module homomorphism $N \to C \otimes_{B} N$, satisfying usual coaction axioms.
  Right comodules over $C$ are defined analogously.
\end{definition}

\begin{remark}
  Under the equivalence~\eqref{eq:hopf-trimod-to-lax-module-functors--functor-definition} between $\Trimod$ and $\mathbf{LexfLax}\cat{V}\text{-}\mathbf{Mod}(\cat{V}, \cat{V})$,
  the $\Trimod$-module category $\tetramod[B]{}$ correponds to the evaluation action of $\mathbf{LexfLax}\cat{V}\text{-}\mathbf{Mod}(\cat{V}, \cat{V})$ on $\cat{V}$.
  Thus, the category of modules over a Hopf trimodule algebra $A$ is precisely the Eilenberg--Moore category for the monad $A \cotens \blank$.
  The subcategory of free modules over $A$ is the image of the Kleisli category for this monad under the canonical embedding $\mathbf{Kl}(A \cotens \blank) \xrightarrow{\ \iota\ } \mathbf{EM}(A \cotens \blank)$.
\end{remark}

Let $A \in \Trimod$ be a Hopf trimodule algebra such that ${}^{B}A$ is quasi-finite, so that the corresponding lax $\tetramod[B]{}$-module functor $A \cotens \blank$ has a left adjoint $\on{cohom}_{B\blank}(A,\blank)$.
This left adjoint is then canon\-ically a comonad, by Proposition~\ref{einerKleinerKleisliÄquivalenz}.

\begin{definition}\label{contramodules}
  A \emph{left contramodule over $A$} is an object in \(\mathbf{EM}(\on{cohom}_{B\blank}(A,\blank))\).

  Explicitly, this means that a left contramodule is a left $B$-comodule $M$ together with a left $B$-comodule homomorphism $M \to \on{cohom}_{B\blank}(A,M)$, endowing $M$ with the structure of a comodule in $\tetramod[B]{}$ over the comonad $\on{cohom}_{B\blank}(A,\blank)$.
\end{definition}

We say that a contramodule is \emph{free} if it lies in the image of the canonical embedding
\[
  \mathbf{Kl}(\on{cohom}_{B\blank}(A,\blank)) \xrightarrow{\ \iota\ } \mathbf{EM}(\on{cohom}_{B\blank}(A,\blank)).
\]
In other words, it is of the form $\on{cohom}_{B\blank}(A,M)$, for a left $B$-comodule $M$, and its coaction is of the form
\[
  \on{cohom}_{B\blank}(A,M) \xrightarrow{\on{cohom}_{B\blank}(\mu, M)} \on{cohom}_{B\blank}(A\cotens A,M) \cong \on{cohom}_{B\blank}(A,\on{cohom}_{B\blank})(A,M).
\]
We denote the category of free contramodules by $A\text{\rm-ContramodFree}$.

Similarly, the category $C\text{\rm-Contramod}$ of left contramodules over a Hopf termodule coalgebra $C$ is
the Eilenberg--Moore category $\mathbf{EM}(\on{Hom}_{B\blank}(C,\blank))$ for the monad $\on{Hom}_{B\blank}(C,\blank)$ right adjoint to the comonad $C \otimes_{B} \blank$.
The free contramodules are defined analogously to the case of Hopf trimodules.

\begin{remark}
  Contramodules similar to those of Definition~\ref{contramodules} were studied extensively in more homological settings, see in particular~\cite{Pos}.
  Indeed, forgetting the left $B$-module structure, a Hopf trimodule algebra yields a so-called \emph{semialgebra}, and by forgetting the left $B$-comodule structure, a Hopf termodule coalgebra yields a so-called \emph{bocs}.
\end{remark}

\begin{lemma}\label{freecontramodmodmodmod}
  Given a Hopf trimodule algebra $A$,
  the action $V \triangleright (A \cotens M) \defeq A \cotens (V \otimes M)$
  endows the category $A\text{\rm-ModFree}$ of free $A$-modules with a canonical module category structure over $\tetramod[B]{}$.

  In a similar fashion, $A\text{\rm-ContramodFree}$ is a module category over $\bbcomod$ by means of the action
  $V \triangleright \on{cohom}_{B\blank}(A,N) \defeq \on{cohom}_{B\blank}(A,V \otimes N)$.
  These two $\tetramod[B]{}$-module categories are canonically equivalent.
\end{lemma}
\begin{proof}
  By definition, we have $A\text{\rm-ModFree} = \mathbf{Kl}(A \cotens \blank)$;
  since $A$ is a Hopf trimodule algebra,
  the monad $A \cotens \blank$ is a lax $\tetramod[B]{}$-module monad,
  so $\mathbf{Kl}(A \cotens \blank)$ is a $\tetramod[B]{}$-module category as described in Proposition~\ref{laxKlCorrespondence}.
  It is easy to check that the assignments described in the lemma correspond precisely to those of Proposition~\ref{laxKlCorrespondence}.
  Similar considerations apply to $A\text{\rm-ContramodFree}$, using the equality $A\text{\rm-ContramodFree} = \mathbf{Kl}(\on{cohom}_{B\blank}(A,\blank))$ and the fact that, since $\on{cohom}_{B\blank}(A,\blank)$ is the left adjoint of the lax $\tetramod[B]{}$-module monad $A \cotens \blank$, it is an oplax $\tetramod[B]{}$-module comonad, by Porism~\ref{porism:doctrinal-module-adjunctions}.
  The asserted equivalence between these two $\tetramod[B]{}$-module categories follows directly from Proposition~\ref{KleisliMonadCorrespondence}.
\end{proof}

Recall from~\cite[1.12]{Tak} that $\on{Cohom}_{B-}(A,-)$ is exact if and only if $A$ is an injective left $B$-comodule.

\begin{proposition}
 If\, $\leftidx{^B}{A}{}$ is injective, the $\bbcomod$-module category structure on $A\text{\rm-ContramodFree}$ of Lemma~\ref{freecontramodmodmodmod} extends to a $\bbcomod$-module category structure on $A\text{\rm-Contramod}$, in the sense of Theorem~\ref{extendingalways}.
\end{proposition}

\begin{proof}
 For the comonadic dual of Theorem~\ref{extendingalways}, it suffices that the comonad $\on{Cohom}_{B-}(A,-)$ is left exact, which is the case if and only if ${}^{B}A$ is injective, by~\cite[1.12]{Tak}.
\end{proof}

\begin{theorem}\label{infinitetrihopfreconstruction}
  Let $B$ be a bialgebra and let $\mathcal{M}$ be a locally finite abelian module category over $\bcomod$ which admits an object $X \in \mathbf{Ind}(\mathcal{M})$ such that:
  \begin{itemize}
    \item $X$ is a \(\bbcomod\)-injective \(\bbcomod\)-cogenerator;
    \item $\blank \lact X\from \bcomod \to \mathbf{Ind}(\mathcal{M})$ is exact; and
    \item $\blank \lact X\from \bbcomod \to \mathbf{Ind}(\mathcal{M})$ is quasi-finite.
  \end{itemize}
  Then there is a Hopf trimodule algebra $(A, \mu\from A \cotens A \to A, \eta\from B \to A)$, such that ${}^{B}A$ is quasi-finite and injective, hence finitely cogenerated injective, such that there is an equivalence of $\tetramod[B]{}$-module categories $\mathbf{Ind}(\mathcal{M}) \simeq A\text{\rm-Contramod}$, between $\mathbf{Ind}(\mathcal{M})$ and the category of left $A$-contramodules.
  The $\bbcomod$-module structure on $A\text{\rm-Contramod}$ is extended, in the sense of Definition~\ref{extendable}, from the category $A\text{\rm-ContramodFree}$ of free left $A$-contramodules.

  Further, this equivalence restricts to an equivalence $\mathcal{M} \simeq A\text{\rm-contramod}_{\on{f.d.}}$, between $\mathcal{M}$ and finite-dimensional $A$-contramodules.
\end{theorem}
\begin{proof}
  Since \(\blank \triangleright X\from \mathcal{C} \to \mathbf{Ind}(\mathcal{M})\) is right exact, its extension to \(\mathbf{Ind}(\mathcal{C})\) admits a right adjoint \(\hom{X,\blank}\) by Proposition~\ref{saftrex}.
  Similarly, since \(\blank \triangleright X\from \mathcal{C} \to \mathbf{Ind}(\mathcal{M})\) is left exact, so is its extension to \(\mathbf{Ind}(\mathcal{C})\), which is also quasi-finite by assumption.
  Thus, by Proposition~\ref{saftlex}, the functor \(\blank \triangleright X\from \mathbf{Ind}(\mathcal{C}) \to \mathbf{Ind}(\mathcal{M})\) has a left adjoint, \(\cohom{X, \blank}\).
  We thus obtain an oplax \(\tetramod[B]{}\)-module comonad \(\cohom{X, \blank \triangleright X}\) and and a lax \(\tetramod[B]{}\)-module monad \(\hom{X, \blank \triangleright X}\), right adjoint to \(\cohom{X, \blank \triangleright X}\).

  By Theorem~\ref{thm:hopf-trimodules-are-lax-monoidal-functors},
  we have an isomorphism \(\hom{X, \blank \triangleright X} \simeq \hom{X, B \triangleright X} \cotens \blank\) of lax \(\tetramod[B]{}\)-module monads,
  where \(\hom{X, B \triangleright X}\) is a Hopf trimodule algebra.
  Let us denote \(\hom{X, B \triangleright X}\) by \(A\). Since $A \cotens \blank$ admits a left adjoint, it follows that ${}^{B}A$ is quasi-finite.

  By uniqueness of adjoints, an isomorphism \(\cohom{X, \blank \triangleright X} \cong \on{cohom}_{B-}(A,\blank)\) follows.
  It is moreover an isomorphism of \(\tetramod[B]{}\)-module comonads, since the \(\tetramod[B]{}\)-module comonad structure is lifted from the isomorphic oplax \(\tetramod[B]{}\)-module monad structures on the respective right adjoints.

  By Theorem~\ref{mainnonsenseinjective}, we have an equivalence \(\mathbf{Ind}(\mathcal{M}) \simeq \mathbf{EM}(\cohom{X, \blank \triangleright X})\) of \(\tetramod[B]{}\)-module categories; in particular, the functor $\cohom{X, \blank}$ is comonadic and hence exact, and since by assumption so is its right adjoint $- \triangleright X$, the comonad $\cohom{X, \blank \triangleright X} \simeq \on{cohom}_{B-}(A,\blank)$ also is exact, which implies the injectivity of ${}^{B}A$, by~\cite[1.12]{Tak}.

  The \(\tetramod[B]{}\)-module structure on \(\mathbf{EM}(\cohom{X,\blank \triangleright X})\) is extended from \(\mathbf{Kl}(\cohom{X,\blank \triangleright X})\).
  By the isomorphism of the previous paragraph, we have an equivalence
  \[
    \mathbf{EM}(\cohom{X, \blank \triangleright X}) \simeq \mathbf{EM}(\on{cohom}(A,\blank))
  \]
  of \(\tetramod[B]{}\)-module categories,
  and an equivalence \(\mathbf{Kl}(\cohom{X, \blank \triangleright X}) \simeq \mathbf{Kl}(\on{cohom}(A,\blank))\).
  By definition we have \(\mathbf{EM}(\on{cohom}(A,\blank)) = A\text{\rm-Contramod}\),
  and \(\mathbf{Kl}(\on{cohom}(A,\blank)) = A\text{\rm-ContramodFree}\).
  Composing the established equivalences, we find the one asserted in the theorem.

  The restricted equivalence \(\mathcal{M} \simeq A\text{\rm-contramod}_{\on{f.d.}}\) follows from Proposition~\ref{locomonads}.
  In particular, we use the observation that a contramodule is compact if and only if its underlying $B$-comodule is such, if and only if it is finite-dimensional.
\end{proof}

In the finite-dimensional setting, we find an alternative result using projective objects.

\begin{theorem}\label{finitetrihopfreconstruction}
  Let $B$ be a finite-dimensional bialgebra and let $\mathcal{M}$ be a finite abelian module category over $\bmodule$ which admits an object $X \in \mathcal{M}$ such that:
  \begin{itemize}
    \item $X$ is a $\bmodule$-projective $\bmodule$-generator; and
    \item $\blank \triangleright X\from \bmodule \to \mathcal{M}$ is exact.
  \end{itemize}
  Then there is a finite-dimensional Hopf termodule coalgebra $C \in \termod$, such that ${}_{B}C$ is projective and such that there is an equivalence of $\bmodule$-module categories $\mathcal{M} \simeq C\text{\rm-contramod}_{\on{f.d.}}$,
  between $\mathcal{M}$ and the category of finite-dimensional left $C$-contramodules.
  The $\bmodule$-module structure on $C\text{\rm-contramod}_{\on{f.d.}}$ is extended from the category $C\text{\rm-ContramodFree}$ of free left $A$-contramodules.
\end{theorem}
\begin{proof}
  This follows analogously to Theorem~\ref{infinitetrihopfreconstruction}, using Theorem~\ref{mainnonsenseprojective} instead of Theorem~\ref{mainnonsenseinjective} and Proposition~\ref{finitemonads} instead of Proposition~\ref{locomonads}.
\end{proof}

In the semisimple case, contramodules become equivalent to modules.

\begin{definition}
  We say that a Hopf trimodule algebra $A$ is \emph{semisimple}
  if the monad $A \cotens \blank$ is semisimple in the sense of Definition~\ref{def:semisimple-monad}.
\end{definition}

\begin{proposition}\label{semisimpletrimodulealgebra}
  For a semisimple Hopf trimodule algebra $A \in \Trimod$, there is an equivalence $\lMod{A} \simeq A\text{\rm-Contramod}$ of module categories over $\tetramod[B]{}$.
\end{proposition}
\begin{proof}
  This is Proposition~\ref{ModuleEMSemisimple} applied to the monad $A \cotens \blank$.
\end{proof}

\subsection{Morita equivalence for contramodules over Hopf trimodule algebras}

\begin{definition}~\label{contramoritadef}
 We say that two Hopf trimodule algebras $A, A' \in \trimod $, such that ${}^{B}A, {}^{B}A'$ are quasi-finite, are \emph{contraMorita equivalent} if the comonads $\on{cohom}_{B\blank}(A,\blank)$ and $\on{cohom}_{B\blank}(A',\blank)$ are Morita equivalent. In other words, if there are $M, N \in \tetramod[B][B][][B]$ with coaction homomorphisms
 \[
  \begin{aligned}
   &\mathrm{la}_{M,A}\colon M \to \on{cohom}_{B\blank}(A,M), \qquad\qquad \mathrm{ra}_{N,A}\colon N \to N\cotens_{B} \on{cohom}_{B\blank}(A,B) \\
   &\mathrm{la}_{N,A'}\colon N \to \on{cohom}_{B\blank}(A',N), \qquad\qquad \mathrm{ra}_{M,A'}\colon M \to M\cotens_{B} \on{cohom}_{B\blank}(A',B) \\
  \end{aligned}
 \]
 in $\tetramod[B][B][][B]$, such that

 \begin{enumerate}
  \item
there is an isomorphism in $\tetramod[B][B][][B]$ between the equalizer of
\[\begin{mytikzcd}[ampersand replacement=\&]
	\& {M \cotens_{B}\on{cohom}_{B\blank}(A',B)\cotens_{B} N} \\
	{M \cotens_{B} N} \&\& {M \cotens_{B} \on{cohom}_{B\blank}(A',N)}
	\arrow["{M \cotens_{B} \partial_{A',B,N}}"{pos=0.7}, from=1-2, to=2-3]
	\arrow["\simeq"', from=1-2, to=2-3]
	\arrow["{\mathrm{ra}_{M,A'} \cotens_{B} N}"{pos=0.3}, from=2-1, to=1-2]
	\arrow[from=2-1, to=2-3]
\end{mytikzcd}\]
and $A$, intertwining the coactions $A \to \on{cohom}_{B-}(A,A)$ and $A \to A \cotens_{B} \on{cohom}_{B-}(A,B)$ coming from the algebra structure on $A$, with the coactions on the equalizer induced by $\mathrm{ra}_{N,A}$ and $\mathrm{la}_{M,A}$. Here, $\partial_{A',B,N}$ is the isomorphism of~\cite[1.13]{Tak}, whose invertibility is implied by injectivity of ${}^{B}A'$.
\item a similar isomorphism in $\tetramod[B][B][][B]$ between a similarly defined equalizer reversing the roles of $M$ and $N$, and $A'$ exists.
 \end{enumerate}
\end{definition}

\begin{proposition}\label{contraspberry}
 For $A,A'$ as in Definition~\ref{contramoritadef}, the following are equivalent:
 \begin{enumerate}
  \item\label{contrabanda} $A$ and $A'$ are contraMorita equivalent;
  \item\label{bigequivatti} there is an equivalence $A\text{\rm-Contramod} \simeq A'\text{\rm-Contramod}$ of\, $\tetramod[B]{}$-module categories;
  \item\label{piccoloequivatti} there is an equivalence $A\text{\rm-contramod}_{\on{f.d.}} \simeq A'\text{\rm-contramod}_{\on{f.d.}}$ of\, $\tetramodfd[B]{}$-module categories.
 \end{enumerate}
\end{proposition}

\begin{proof}
 The equivalence of~\eqref{contrabanda} and~\eqref{bigequivatti} is an instance of Proposition~\ref{Moritabybimodules}, for $\mathcal{C} = \tetramod[B]{}$. Since an equivalence of categories preserves compact objects,~\eqref{piccoloequivatti} follows from~\eqref{bigequivatti}. Finally, if $F$ is an equivalence as in~\eqref{piccoloequivatti}, then $\mathbf{Ind}(F)$ is an equivalence as in~\eqref{bigequivatti}.
\end{proof}

Combining Proposition~\ref{contraspberry} with Theorem~\ref{infinitetrihopfreconstruction}, we find the following algebraic realization of Theorem~\ref{injectivebijection}:
\begin{theorem}\label{contramodulereuters}
There is a bijection
\[
  \setj{\,(\mathcal{M},X) \text{ as in Theorem~\ref{infinitetrihopfreconstruction}}\,}/(\mathcal{M} \simeq \mathcal{N})
  \xleftrightarrow{1:1}
  \left\{
    \begin{aligned}
      &\text{Left finitely cogenerated injective} \\
      &\text{Hopf trimodule algebras over $B$}
    \end{aligned}
  \right\}
  \Big/\!\simeq_{\text{ContraMorita}}.
\]
\end{theorem}

\subsection{A semisimple example of non-rigid reconstruction}

We now give a concrete applications of Theorem~\ref{mainnonsenseinjective} and Theorem~\ref{infinitetrihopfreconstruction}, describing module categories for non-rigid categories, where the ordinary reconstruction procedure for rigid categories would not yield the correct result.

\begin{example}\label{coreps}
  Let $S = \setj{e,s}$ be the commutative two-element monoid which is not a group, with identity element $e$. In particular, $s^{2} = s$. The category $\mathbf{Rep}(S)$ of finite-dimensional modules over the monoid algebra $\Bbbk[S]$ is semisimple, with two isoclasses of simple objects, given by $1$-dimensional representations $\Bbbk_{\on{triv}}$ and $\Bbbk_{\on{grp}}$. In $\Bbbk_{\on{triv}}$, the element $s$ acts by $1$, while in $\Bbbk_{\on{grp}}$ it acts by $0$. The standard comultiplication on $\Bbbk[S]$, in which $e,s$ are grouplike, makes it into a bialgebra. The category $\mathbf{CoRep}(S)$ of finite-dimensional comodules over $\Bbbk[S]$ is equivalent to that of $S$-graded finite-dimensional spaces, $\mathbf{vec}^{S}$. Thus, it too is semisimple of rank two---we denote the simple objects by $\delta_{e}, \delta_{s}$.

  The monoidal structure on $\mathbf{Rep}(S)$ satisfies $\Bbbk_{\on{grp}} \otimes \Bbbk_{\on{grp}} \cong \Bbbk_{\on{grp}}$.
  In fact, this presentation of the Grothendieck ring $[\mathbf{Rep}(S)]$ of $\mathbf{Rep}(S)$ determines $\mathbf{Rep}(S)$ uniquely as a semisimple monoidal category, see~\cite[Proposition~3.5]{CSZ}.
  Thus, we find monoidal equivalences
  \[
    \mathbf{Rep}(S) \simeq \mathbf{CoRep}(S) \simeq \mathbf{vec}^{S},
  \]
  since in the latter we have $\delta_{s} \otimes \delta_{s} \cong \delta_{s}$.
  In what follows we focus on the coalgebraic setup of Theorem~\ref{mainnonsenseinjective}, and thus study the module categories over $\mathbf{CoRep}(S)$.

  A semisimple $\mathbf{CoRep}(S)$-module category $(\mathcal{M}, \triangleright)$ is indecomposable (i.e.\ does not split as a direct sum of $\mathbf{CoRep}(S)$-module subcategories) if and only if the matrix ${[\delta_{s}]}_{\mathcal{M}}$ describing the action of $\delta_{s}$ on the Grothendieck ring $[\mathcal{M}]$ in the basis of simple objects for $\mathcal{M}$ is indecomposable.
  But $[\delta_{s}]$ is idempotent, so it must be the $1\times 1$-matrix $(1)$, and, as a category, we must have $\mathcal{M} \simeq \kvect$. We then either have $\delta_{s} \triangleright \blank \cong \on{Id}_{\mathcal{M}}$, or $\delta_{s} \triangleright \blank \cong 0$. One can show that any two module categories satisfying one of these conditions are equivalent, by showing that the second cohomology group $H^{2}(S, \Bbbk^{\times})$ is trivial,
  similarly to~\cite[Example~7.4.10]{EGNO}. Indeed, $\psi \in Z^{2}(S, \Bbbk^{\times})$ must satisfy $\psi(e,s) = \psi(e,e) = \psi(s,e)$, and thus can be written as $\psi(x,y) = \sigma(x) - \sigma(xy) + \sigma(y)$, by setting $\sigma(e) \defeq \psi(e,e)$ and $\sigma(s) \defeq \psi(s,s)$.

  The condition $\delta_{s} \triangleright \blank \cong \on{Id}_{\mathcal{M}}$ is satisfied by the fiber functor $\mathrm{Fr}\from \mathbf{CoRep}(S) \to \mathbf{vec}$.

  Assume instead that we have an indecomposable module category $\mathcal{M}$ satisfying $\delta_{s} \triangleright - \cong 0$. Let $X \in \mathcal{M}$ be a simple object. It is not difficult to verify that $X$ satisfies the assumptions of Theorem~\ref{mainnonsenseinjective}. Given $V \in \mathbf{CoRep}(S)$, we have
  \[
    \lceil X,V \triangleright X \rceil_{s} = \on{Hom}_{\mathbf{CoRep}(S)}(\lceil X, V\triangleright X \rceil, \delta_{s}) \cong \on{Hom}_{\mathbf{CoRep}(S)}(V \triangleright X, \delta_{s} \triangleright X),
  \]
  and similarly for $\delta_{e}$. Using $\delta_{s} \triangleright X \cong 0$ and $\delta_{e} \triangleright X \cong X$ we find that $\lceil X, \delta_{e} \triangleright X \rceil \cong \delta_{e}$ and $\lceil X, \delta_{s} \triangleright X \rceil = 0$. To describe the bicomodule corresponding to the functor $\lceil X, - \triangleright X \rceil$ under the correspondence of Proposition~\ref{TakeuchiEilenbergWatts}, we observe that a bicomodule can be identified with an $S \times S$-graded space, and the cotensor product is given by ${(V \cotens W)}_{x,y} = \bigoplus_{z \in S} V_{x,z} \otimes W_{z,y}$. Using this, we find that $\lceil X, - \triangleright X \rceil \cong \delta_{e,e} \cotens -$, where $\delta_{e,e}$ is the one-dimensional space concentrated in degree $(e,e)$. The unit Hopf trimodule homomorphism $\Bbbk[S] \to \lceil X, \Bbbk[S] \triangleright X \rceil$ must thus have the submodule $\Bbbk\setj{s}$ as its kernel, so the reconstructing Hopf trimodule $\lceil X, \Bbbk[S] \triangleright X \rceil$ is given by $\Bbbk[S]/\Bbbk\setj{s}$. As a bicomodule, it is indeed concentrated in degree $(e,e)$, and as a module it is isomorphic to $\Bbbk_{\on{grp}}$. The composition $\lceil X, \delta_{e} \triangleright X\rceil \otimes \lceil X, \delta_{e} \triangleright X\rceil \to \lceil X, \delta_{e} \triangleright X\rceil$ sending $\on{id}_{X} \otimes \on{id}_{X}$ to $\on{id}_{X}$ corresponds to the algebra structure
  \[
    \begin{aligned}
      \Bbbk[S]/\Bbbk\setj{s} \cotens \Bbbk[S]/\Bbbk\setj{s} \simeq \Bbbk[S]/\Bbbk\setj{s} \otimes \Bbbk[S]/\Bbbk\setj{s} \to \Bbbk[S]/\Bbbk\setj{s} \\
      (e + \Bbbk\setj{s}) \otimes (e + \Bbbk\setj{s}) \mapsto (e + \Bbbk\setj{s}).
    \end{aligned}
  \]
  It is not difficult to see that this multiplication defines a Hopf trimodule algebra. To check that the category $\Bbbk[S]/\Bbbk\setj{s}\!\on{-Contramod}$ satisfies the properties we have assumed of $\mathcal{M}$ previously, note that the free module $\Bbbk[S]/\Bbbk\setj{s} \cotens B$ is one-dimensional, hence semisimple, showing that the category of modules over $\Bbbk[S]/\Bbbk\setj{s}$ is semisimple of rank one. Thus, by~\ref{semisimpletrimodulealgebra} we have
  \[
    \Bbbk[S]/\Bbbk\setj{s}\text{\rm-Contramod} \simeq \Bbbk[S]/\Bbbk\setj{s}\text{\rm-Mod}.
  \]
  Further, $\delta_{s} \triangleright \Bbbk[S]/\Bbbk\setj{s} = \Bbbk[S]/\Bbbk\setj{s} \cotens \delta_{s} = 0$, which corresponds to $\delta_{s} \triangleright X = 0$.

  If we instead follow the reconstruction procedured described in~\cite[Chapter~7]{EGNO}, the reconstructing algebra object would be the object $A$ of $\mathbf{CoRep}(S)$ representing the presheaf $\on{Hom}_{\mathcal{M}}(-\triangleright X,X)$. We see that we must have $A \cong \delta_{e}$, and the algebra structure on $\delta_{e}$ is unique. However, $\delta_{e} \cong \mathbb{1}_{\mathbf{CoRep}(S)}$ and so $\on{Mod}(\delta_{e}) \cong \mathbf{CoRep}(S) \not\simeq \mathcal{M}$, showing that the reconstruction procedure for module categories over rigid monoidal categories fails in this non-rigid case. It also illustrates that this failure can be observed already in the semisimple case.

\end{example}

In Section~\ref{fiberfunctorthings}, we give a general construction of a Hopf trimodule algebra reconstructing the fiber functor as a module category, which covers the remaining indecomposable semisimple module category over the category $\mathbf{CoRep}(S)$ considered in Example~\ref{coreps}.

\section{A Hopf trimodule algebra reconstructing the fiber functor}\label{fiberfunctorthings}

\begin{proposition}\label{prop:B-dot-B-trimodule}
  The \(\Bbbk\)-vector space \(B \kotimes B\) forms a Hopf trimodule
  with coactions
  \[
    b \otimes c \mapsto b_{(1)} \otimes b_{(2)} \otimes c,\qquad\qquad
    b \otimes c \mapsto b \otimes c_{(1)} \otimes c_{(2)},
  \]
  and action
  \[
    x \otimes b \otimes c \mapsto x_{(1)}b \otimes x_{(2)}b.
  \]
  We shall denote it by \(B \bullet B\).
\end{proposition}
\begin{proof}
  It is clear that \(B \bullet B\) is a bicomodule,
  as both of the coactions are given by that of \(B\) on the respective Let.

  Thus, we first verify that the action is a morphism of left and right comodules;
  i.e., that \(B \bullet B\) is in \(\tetramod[][B][][B]\) and \(\tetramod[B][B]\).
  The former follows by Figure~\ref{fig:b-bullet-b-left-mod-right-comod},
  while the latter is due to the calculation in Figure~\ref{fig:b-bullet-b-left-mod-left-comod}.
  \begin{figure}[htbp]
    \tikzsetnextfilename{fig-b-bullet-b-left-mod-right-comod}

    \caption{}%
    \label{fig:b-bullet-b-left-mod-right-comod}
  \end{figure}
  \begin{figure}[htbp]
    \tikzsetnextfilename{fig-b-bullet-b-left-mod-left-comod}
    %
    \caption{}%
    \label{fig:b-bullet-b-left-mod-left-comod}
  \end{figure}
\end{proof}

\begin{proposition}\label{prop:b-dot-b-is-algebra-object}
  The Hopf trimodule \(B \bullet B\) from Proposition~\ref{prop:B-dot-B-trimodule} is an algebra object in \(\tetramod[B][B][][B]\).
\end{proposition}
\begin{proof}
  Figure~\ref{fig:b-bullet-b-algebra-structure-def} defines the multiplication and unit of \(B \bullet B\),
  where we have denote the counit \(\varepsilon\) of \(B\) with a small white dot,
  and have employed the same numbering system for representing the cotensor product as in Section~\ref{sec:bicomodules}.
  \begin{figure}[htbp]
    \tikzsetnextfilename{fig-b-bullet-b-algebra-structure-def}

    \caption{}%
    \label{fig:b-bullet-b-algebra-structure-def}
  \end{figure}

  It is clear that \(\mu\) is associative,
  left unitality follows from Figure~\ref{fig:b-bullet-b-mu-unitality},
  and right unitality is similar.
  \begin{figure}[htbp]
    \tikzsetnextfilename{fig-b-bullet-b-mu-unitality}
    %
    \caption{}%
    \label{fig:b-bullet-b-mu-unitality}
  \end{figure}
  One immediately has that \(\mu\) is a \(\tetramod[B][][][B]\)-morphism,
  and for \(\eta = \Delta\) this follows from coassociativity.
  Thus, to finish the proof we need only show that both are also morphisms of left \(B\)-modules,
  which follows from the calculations of Figures~\ref{fig:b-bullet-b-module-morphisms}
  and~\ref{fig:b-bullet-b-module-morphisms-2}.
  \begin{figure}[htbp]
    \tikzsetnextfilename{fig-b-bullet-b-module-morphisms}
    %
    \caption{}%
    \label{fig:b-bullet-b-module-morphisms}
  \end{figure}
  \begin{figure}[htbp]
    \tikzsetnextfilename{fig-b-bullet-b-module-morphisms-2}
    %
    \caption{}%
    \label{fig:b-bullet-b-module-morphisms-2}
  \end{figure}
\end{proof}

\begin{definition}\label{jequivalence}
  Let $J\from (B \bullet B)\text{\rm-ModFree} \to \kVect$ be the functor given by $J(B \bullet B \cotens M) = M$ and by local maps
  \[
    \begin{aligned}
      J_{M,N}\from (B\bullet B)\text{\rm-ModFree}\big((B\bullet B)\cotens M,  (B\bullet B)\cotens N \big) &\to \on{Hom}_{\Bbbk}(M,N) \\
      \sigma &\mapsto (\varepsilon \otimes \varepsilon \cotens \on{id}_{N})\circ(\sigma(1 \otimes 1 \otimes \blank))
    \end{aligned}
  \]
\end{definition}

\begin{lemma}
  The assignments of Definition~\ref{jequivalence} yield an equivalence of ($\bbcomod$)-module categories.
\end{lemma}
\begin{proof}
  First, assume that \(J\) is a functor;
  we will show it is an equivalence of ($\bbcomod$)-module categories.
  Let $V \in \kVect$.
  Essential surjectivity follows by $V \cong \Bbbk^{\dim{V}} \cong J((B\bullet B)\cotens \Bbbk_{\on{triv}}^{\dim{V}})$.

  Second, $J$ is fully faithful: we have that $J_{M,N}$ is defined as the composite of the isomorphisms
  \begin{align*}
    (B \bullet B)&\text{\rm-ModFree}\big((B\bullet B)\cotens M,  (B\bullet B)\cotens N\big)
             \xrightarrow{\blank \circ \eta} \bbcomod(M,  (B\bullet B)\cotens N) \\
           & \xrightarrow{(B \otimes \varepsilon \cotens N) \circ \blank} \bbcomod(M,  B\otimes N)
             \xrightarrow{(\varepsilon \otimes N) \circ \blank} \kvect(M,N).
  \end{align*}
  The first morphism is an isomorphism of the form $\mathbf{EM}(T)(TX, TY) \xiso \mathcal{A}(X,TY)$, for the monad $T$ on a category $\mathcal{A}$, where $T = (B \bullet B) \cotens \blank$ and $\mathcal{A} = \bbcomod$.
  The second isomorphism is simply postcomposition with the isomorphism $(B \bullet B) \cotens N = B \otimes  B \cotens N \xrightarrow{B \otimes \varepsilon \cotens N} B \otimes N$, since $B \cotens N \xrightarrow{\varepsilon \cotens N} N$ is an isomorphism.
  The third is an isomorphism $\mathbf{EM}(K)(KX,KY) \xiso \mathcal{A}(KX, Y)$, for the comonad $K$ on a category $\mathcal{A}$, where $K = B \otimes \blank$ and $\mathcal{A} = \kVect$.

  Third, $J$ is a ($\bbcomod$)-module functor, since for $W \in \bbcomod$, we have
  \[
    W \triangleright J(B\bullet B \cotens M) = W \otimes M = J(B\bullet B \cotens (W \otimes M)) = J(W \triangleright B \bullet B \cotens M).
  \]
  Thus, if $J$ defines a functor, then it is an equivalence of ($\bbcomod$)-module categories.

  It is left to verify the functoriality of $J$.
  Let \(\sigma \in (B \bullet B)\text{-ModFree}((B \bullet B) \cotens M, (B \bullet B) \cotens N)\)
  and \(\tau \in (B \bullet B)\text{-ModFree}((B \bullet B) \cotens N, (B \bullet B) \cotens W)\).
  Figure~\ref{fig:jequivalence-sigma} shows our graphical representation of these morphisms.
  \begin{figure}[htbp]
    \tikzsetnextfilename{fig-jequivalence-sigma}

    \caption{}%
    \label{fig:jequivalence-sigma}
  \end{figure}

  Functoriality of \(J\) now follows by
  Figure~\ref{fig:jequivalence-functoriality},
  where \((i)\) and \((ii)\) hold due to the following identity:
  \[
    %
  \]
  Further, \((iii)\) is satisfied by \(\tau\) being a \((B \bullet B)\)-module morphism,
  \((iv)\) follows by the following calculation, using that \(B\) is a bialgebra,
  \[
    \varepsilon(b_{(1)}b_{(2)}) \otimes b_{(3)}
    = \varepsilon(b_{(1)})\varepsilon(b_{(2)}) \otimes b_{(3)}
    = \varepsilon(b_{(1)}\varepsilon(b_{(2)})) \otimes b_{(3)}
    = \varepsilon(b_{(1)}) \otimes b_{(2)}
    = b,
  \]
  and \((v)\) holds by an analogous computation.
  \begin{figure}[htbp]
    \tikzsetnextfilename{fig-jequivalence-functoriality}
    %
    \caption{}%
    \label{fig:jequivalence-functoriality}
  \end{figure}
\end{proof}

\begin{corollary}
  The monad $(B \bullet B) \cotens \blank$ is semisimple, and there is an equivalence of $\bbcomod$-module categories $(B \bullet B)\text{\rm-Mod} \simeq \kVect$.
\end{corollary}

\section{The fundamental theorem of Hopf modules}
\subsection{Hopf trimodules and twisted antipodes}\label{sec:Hausser-Nill}

\begin{definition}
  Let $B$ be a bialgebra.
  A \emph{twisted antipode} for $B$ is an antipode for the bialgebra $B^{\on{cop}}$.
  Equivalently, it is one for the bialgebra $B^{\on{op}}$.
\end{definition}

Note that a Hopf algebra admits a twisted antipode \(\overline{S}\) if and only if its own antipode \(S\) is invertible,
in which case we have \(S^{-1} = \overline{S}\).
Notably, a finite-dimensional Hopf algebra always admits a twisted antipode.

Lemma~\ref{antipodeleftrigid} below is a well-known result, see e.g.~\cite[Theorem~5.4.1.(2)]{EGNO}.
We emphasize that while its former part is seemingly opposite to the claim of~\cite[Theorem~5.4.1.(2)]{EGNO},
identifying the presence of an antipode with right, rather than left, rigidity, this is because we consider left, rather than right, comodules.
The latter claim of Lemma~\ref{antipodeleftrigid} follows from the former by using the equivalence ${(\tetramod[B])}^{\on{rev}} \simeq \tetramod[B^{\on{op}}]$.

\begin{lemma}\label{antipodeleftrigid}
  Let $B$ be a bialgebra.
  Then $\tetramodfd[B]$ is right rigid if and only if\, $B$ admits an antipode.
  Similarly, $\tetramodfd[B]$ is left rigid if and only if\, $B$ admits a twisted antipode.
\end{lemma}

The next theorem below was shown
in~\cite[Proposition~3.11]{hausser99:integ-theor-quasi-hopf-algeb}
and~\cite[Theorem~3.9]{saracco17:hopf}
in the setting of quasi-bialgebras,
as a generalization of the structure theorem for Hopf modules of Larson and Sweedler~\cite{larson69:hopf}.

\begin{theorem}\label{HausserNill}
  Let $B$ be a bialgebra.
  The functor $\blank \otimes B\from \tetramod[][B] \to \tetramod[][B][B][B]$ is an equivalence if and only if\, $B$ admits an antipode.
\end{theorem}

The main aim of this section is to show that a variant of Theorem~\ref{HausserNill} can be deduced from the following general statement about module categories and module functors, which can thus be viewed as a weak Hausser--Nill theorem for general monoidal categories.

\begin{proposition}\label{LaxIsStrongThenLeftRigid}
  Let $\mathcal{C}$ be a left closed monoidal category such that every left $\mathcal{C}$-module endofunctor of $\mathcal{C}$ is strong%
  ---in other words, the monoidal embedding \(\mathcal{C}\text{\rm-StrMod}(\mathcal{C},\mathcal{C}) \hookrightarrow \mathcal{C}\text{\rm-LaxMod}(\mathcal{C},\mathcal{C})\) is an equivalence.
  Then $\mathcal{C}$ is left rigid.
\end{proposition}

\begin{proof}
  Since $\mathcal{C}$ is left closed, for every $V \in \mathcal{C}$ there is a right adjoint $\hom{V, \blank}$ to the strong $\mathcal{C}$-module endofunctor $\blank \otimes V$.
  By Porism~\ref{porism:doctrinal-module-adjunctions}, the functor $\hom{V, \blank}$ is a lax $\mathcal{C}$-module endofunctor of $\mathcal{C}$, and thus, by assumption, it is a strong module endofunctor.
  Thus, by Proposition~\ref{prop:bicategorical-yoneda-correspondence-monads-and-algebras},
  there is an object $\ld{V}$ such that there is an isomorphism of $\mathcal{C}$-module functors $\hom{V, \blank} \cong \blank \otimes \ld{V}$.
  Thus we obtain an adjunction of $\mathcal{C}$-module functors $\blank \otimes V \dashv \blank \otimes \ld{V}$,
  which by~\cite[102-103]{kelly72:many} shows that $(\ld{V}, V)$ is a dual pair in $\mathcal{C}$.
  It follows that $\mathcal{C}$ is left rigid.
\end{proof}

We now define the functor $B \otimes \blank$ analogous to the functor $\blank \otimes B$ of Theorem~\ref{HausserNill}.

\begin{definition}\label{bestoftimes}
  Define the functor $B\otimes \blank\from  \tetramod[B][][][] \to \tetramod[B][B][][B]$ by sending a left $B$-comodule $M$ to the trimodule $B \otimes M$ whose left and right coaction, as well as left action, is given by
  \[
    \big(
    b \otimes m \mapsto b_{(1)}m_{(-1)} \otimes b_{(2)} \otimes m_{(0)}, \quad
    b \otimes m \mapsto b_{(1)}\otimes m \otimes b_{(2)}, \quad
    a \otimes b \otimes m \mapsto ab \otimes m
    \big),
  \]
  where we extends to morphisms in the natural way.
  It is easy to verify that $B \otimes M$ defines a Hopf trimodule and also that $B \otimes \blank$ defines a functor.
\end{definition}

Comparing Definition~\ref{bestoftimes} with Theorem~\ref{HausserNill}, we remark on two differences.
First, the domain of the functor we study is the category of comodules over $B$, rather than modules.
This is merely to maintain consistence with Section~\ref{sec:Hopf-trimodules}, and a variant for modules can be formulated without greater difficulty.
Second, the functor we study endows a left comodule with an additional comodule structure on the right, and a module structure on the left, rather than on the right.
We will see that it is related to the functor $\tetramod \to \tetramod[][B][][B]$, which is known to be an equivalence if $B$ admits a \emph{twisted} antipode;
see for example~\cite[Section~1.9.4]{montgomery93:hopf}.

We now give a categorical interpretation of the functor of Definition~\ref{bestoftimes}:

\begin{proposition}\label{EmperorDaibazaal}
  The diagram
  \[
    \begin{mytikzcd}[ampersand replacement=\&]
      {\tetramod[B]\text{\rm-StrMod}(\tetramod[B], \tetramod[B])} \&\& {\tetramod[B]\text{\rm-LexfLaxMod}(\tetramod[B], \tetramod[B])} \\
      {{\tetramod[B]}^{\otimes \on{rev}}} \&\& \tetramod[B][B][][B]
      \arrow[hook, from=1-1, to=1-3]
      \arrow["\simeq", from=1-1, to=2-1]
      \arrow["\simeq", from=1-3, to=2-3]
      \arrow[hook, from=2-1, to=2-3]
    \end{mytikzcd}
  \]
  commutes up to a monoidal natural isomorphism. Its upper horizontal arrow is the natural inclusion functor, its left vertical arrow is the equivalence of Proposition~\ref{prop:bicategorical-yoneda-correspondence-monads-and-algebras}, its right vertical arrows is the equivalence of Theorem~\ref{thm:hopf-trimodules-are-lax-monoidal-functors}, and its lower horizontal arrow is the functor of Definition~\ref{bestoftimes}.
\end{proposition}
\begin{proof}
  Chasing a functor $F$ in the diagram, we find the following:
  \[
    \begin{mytikzcd}[ampersand replacement=\&]
      F \&\&\& F \\
      {F(\Bbbk_{\on{triv}})} \& {B \otimes F(\Bbbk_{\on{triv}})} \& {F(B \otimes \Bbbk_{\on{triv}})} \& {F(B)}
      \arrow[maps to, from=1-1, to=1-4]
      \arrow[maps to, from=1-1, to=2-1]
      \arrow[maps to, from=1-4, to=2-4]
      \arrow[maps to, from=2-1, to=2-2]
      \arrow["\simeq", from=2-2, to=2-3]
      \arrow["\simeq", from=2-3, to=2-4]
    \end{mytikzcd}
  \]
  It is easy to verify that the isomorphism indicated in the bottom row of this chase is natural in $F$, and also monoidal.
\end{proof}

\begin{corollary}
  If the functor $B \otimes \blank$ of Definition~\ref{bestoftimes} is an equivalence, $B$ admits a twisted antipode.
\end{corollary}
\begin{proof}
  By Proposition~\ref{EmperorDaibazaal}, $B \otimes \blank$ is an equivalence if and only if the embedding
  \begin{equation}\label{strongisweak}
    {\tetramod[B]\text{\rm-StrMod}(\tetramod[B], \tetramod[B])} \hookrightarrow {\tetramod[B]\text{\rm-LexfLaxMod}(\tetramod[B], \tetramod[B])}
  \end{equation}
  is an equivalence.
  Thus, we may assume that the embedding of Equation~\ref{strongisweak} is an equivalence. For $M \in \tetramodfd[B]$, the right adjoint of $\blank \kotimes M$ is finitary by Proposition~\ref{saftlex}, since $\blank \kotimes M$ sends finite-dimensional comodules to finite-dimensional comodules. The right adjoint is then a finitary, left exact lax $\tetramod[B]$-module functor, which by assumption is strong, and thus is of the form $\blank \otimes \ld{M}$, for $\ld{M} \in \tetramod[B]$. Since the fiber functor $U\from \tetramod[B] \to \kVect$ is strong monoidal, we have $\ld{M} \simeq M^{\ast}$. This shows that $\tetramodfd[B]$ is left rigid, and hence by Lemma~\ref{antipodeleftrigid} the bialgebra $B$ admits a twisted antipode.
\end{proof}

\begin{proposition}
  If $B$ admits a twisted antipode, then the functor $B \otimes \blank$ of Definition~\ref{bestoftimes} is an equivalence.
\end{proposition}
\begin{proof}
  Since the functor $\tetramod[B] \hookrightarrow {\tetramod[B]\text{\rm-LexfLaxMod}(\tetramod[B], \tetramod[B])}$ is always full and faithful,
  so is the functor $B \otimes \blank$ by Proposition~\ref{EmperorDaibazaal}.
  It remains to show that $B \otimes \blank$ is essentially surjective.

  Following~\cite[Section~1.9.4]{montgomery93:hopf}, we know that for $X \in \tetramod[][B][][B]$, the composite map
  \[
    X \xrightarrow{\Delta_{X}} X \otimes B \xrightarrow{\mathtt{t}\otimes B} X^{\on{co}B} \otimes B \xrightarrow[\sim]{b_{X^{\on{co}B},B}} B \otimes X^{\on{co}B}
  \]
  is an isomorphism in $\tetramod[][B][][B]$, where $X^{\on{co}B} = \{\,m \in X \; | \; \Delta^{r}_{X}(m) = m \otimes 1\,\}$ and $\mathtt{t}$ is the projection diagrammatically represented by
  
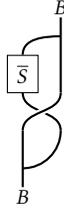
\[
    \begin{tikzpicture}[style=tikzfig]
      \begin{pgfonlayer}{nodelayer}
        \node [style=none] (0) at (-0.5, -2) {};
        \node [style=none] (1) at (-0.5, -1.5) {};
        \node [style=none] (2) at (0.5, -0.5) {};
        \node [style=none] (3) at (-0.5, -0.5) {};
        \node [style=none] (4) at (-0.5, 0.5) {};
        \node [style=none] (5) at (0.5, 0.5) {};
        \node [style=box] (6) at (-0.5, 1) {\scriptsize$\overline{S}$};
        \node [style=none] (7) at (0.5, 2.5) {};
        \node [style=none] (9) at (0.5, 2) {};
        \node [style=none] (10) at (0.5, 2.75) {\scriptsize$B$};
        \node [style=none] (11) at (-0.5, -2.25) {\scriptsize$B$};
      \end{pgfonlayer}
      \begin{pgfonlayer}{edgelayer}
        \draw [style=morphism-edge] (0.center) to (3.center);
        \draw [style=morphism-edge, in=270, out=0] (1.center) to (2.center);
        \draw [style=morphism-edge, in=270, out=90] (2.center) to (4.center);
        \draw [style=morphism-edge] (4.center) to (6);
        \draw [style=morphism-edge, in=180, out=90] (6) to (9.center);
        \draw [style=morphism-edge] (9.center) to (5.center);
        \draw [style=morphism-edge] (9.center) to (7.center);
        \draw [style=braid-over, in=270, out=90] (3.center) to (5.center);
      \end{pgfonlayer}
    \end{tikzpicture}
  \]

  Its inverse is given by the restriction ${\nabla_{X}}_{|\on{co}B}\from B \otimes X^{\on{co}B} \to X$ of the left $B$-action on $X$.

  For $X \in \tetramod[B][B][][B]$, the space $X^{\on{co}B}$ of coinvariants is a left $B$-subcomodule of $X$, since $X$ is a bicomodule. It is easy to check that $B \otimes X^{\on{co}B} \xrightarrow{{\nabla_{X}}_{|\on{co}B}} X$ is a left $B$-comodule morphism, and thus a morphism in $\tetramod[B][B][][B]$, making $B \otimes \blank$ essentially surjective.
\end{proof}

\begin{corollary}
  A bialgebra $B$ admits a twisted antipode if and only if the functor $B \otimes \blank$ of Definition~\ref{bestoftimes} is an equivalence.
\end{corollary}

\begin{remark}
  Here we see that it is crucial
  that Section~\ref{sec:Hopf-trimodules} identifies Hopf trimodules with \emph{finitary} lax module endofunctors.
  Indeed, given an infinite-dimensional comodule $M$ over $B$,
  the functor $\on{Hom}_{\Bbbk}(M, \blank)\from \tetramod[B] \to \tetramod[B]$ is endowed with lax $\tetramod[B]$-structure by virtue of Proposition~\ref{prop:special-doctrinal-module-adjunctions},
  but is not finitary itself.
\end{remark}

\subsection{Fusion operators for Hopf monads}

Similarly to our study of the Hausser--Nill theorem in Section~\ref{sec:Hausser-Nill}, the results of~\cite{BLV} can be used to characterize rigidity of a given monoidal category $\mathcal{C}$.
In this case, we will require $\mathcal{C}$ to admit a fiber functor to a rigid category $\mathcal{W}$.
We begin by recalling the results of \emph{loc cit} that we will need.

\begin{definition}[{\cite[Section~2.6]{BLV}}]\label{fusioninclusion}
  Let $T$ be an opmonoidal monad on a monoidal category $\mathcal{W}$.
  The \emph{right fusion operator} $T_{\mathbf{f}}$ of $T$ is the composite natural transformation
  \[
    {(T_{\mathbf{f}})}_{V,W}\from T(T(V) \otimes W)
    \xrightarrow{{(T_{\mathsf{m}})}_{T(V),W}} T^{2}(V) \otimes T(W)
    \xrightarrow{\mu_{V} \otimes T(W)} T(V) \otimes T(W), \quad\text{ for } V,W \in \mathcal{W}.
  \]
  The monad $T$ is said to be a \emph{right Hopf monad} if its right fusion operator is invertible.
  Left fusion operators and left Hopf monads are defined similarly, and a monad is said to be simply a \emph{Hopf monad} if it is both left and right Hopf.
\end{definition}

\begin{theorem}[{\cite[Lemma~3.4 and Theorem~3.6]{BLV}}]
  Let $\mathcal{W}$ be a right rigid monoidal category and let $T$ be an opmonoidal monad on $\mathcal{W}$.
  The right fusion operators of\, $T$ are invertible if and only if\, $\mathbf{EM}(T)$ is right rigid.
\end{theorem}

A variation of the next result, using \emph{Doi--Koppinen data}, appears in~\cite[Section~5.5]{kowalzig15:cyclic}.

\begin{lemma}\label{delusion}
  Let \(\adj{F}{U}{\cat{W}}{\cat{C}}\) be an opmonoidal adjunction.
  The strong monoidal structure of \(U\) turns \(\mathcal{W}\!\) into a \(\mathcal{C}\)-module category, by defining \(\blank \triangleright \bblank \;\defeq\; U(-) \otimes \bblank\).

  With respect to this $\mathcal{C}$-module structure,
  the opmonoidal monad $T \defeq UF$ becomes an oplax $\mathcal{C}$-module monad.
  For all \(W \in \cat{W}\) and \(X \in \cat{C}\), the coherence morphism is given by
  \begin{equation}\label{confusion}
    {(T_{\mathsf{a}})}_{X,W} \from T(X \triangleright W) \eqdef T(U(X) \otimes W)
    \xrightarrow{{(T_{\mathsf{m}})}_{U(X),W}} TU(X) \otimes T(W)
    \xrightarrow{U\varepsilon_{X} \otimes T(W)} U(X) \otimes T(W).
  \end{equation}
  In particular, for $V \in \mathcal{W}$, the coherence morphism ${(T_{\mathsf{a}})}_{F(V),W}$ is precisely the right fusion operator for $T$ at $(V,W)$.
  In other words, ${(T_{\mathsf{a}})}_{F(V),W} = {(T_{\mathbf{f}})}_{V,W}$.
\end{lemma}
\begin{proof}
  The first statement is a well-known result on transport of structure.
  The $\mathcal{C}$-module structure on $U$ is given by the morphisms
  \[
    U_{\mathsf{m}}\from X \triangleright U(Y) \eqdef U(X) \otimes U(Y) \xiso U(X \otimes Y) \eqdef U(X \triangleright Y).
  \]
  For the second statement, observe that by Porism~\ref{porism:doctrinal-module-adjunctions}, the left adjoint $F$ of $U$ inherits an oplax $\mathcal{C}$-module functor structure from the strong $\mathcal{C}$-module structure on $U$. The resulting oplax $\mathcal{C}$-module functor on the composite $T = UF$ is given by
  \[
    \mathscale{0.83}{%
      \begin{mytikzcd}[ampersand replacement=\&]
        {UF(U(X)\otimes W)} \& {U(FU(X) \otimes F(W))} \& {U(X \otimes F(W))} \& {U(X) \otimes UF(W)} \& {U(X)\otimes T(W)} \\
        {T(X \triangleright W)} \& {UF(X \triangleright W)} \& {U(X \triangleright F(W))} \& {X \triangleright UF(W)} \& {X \triangleright T(W)}
        \arrow["{\raisebox{0.5em}{\(U({(F_{\mathsf{m}})}_{U(X),W})\)}}", from=1-1, to=1-2]
        \arrow["{=}"', from=1-1, to=2-1]
        \arrow["\raisebox{0.5em}{\(U(\varepsilon_{X} \otimes F(W))\)}", from=1-2, to=1-3]
        \arrow["\raisebox{0.5em}{\({(U_{\mathsf{m}})}_{X,F(W)}\)}", from=1-3, to=1-4]
        \arrow["{=}", from=1-3, to=2-3]
        \arrow["{=}", from=1-4, to=1-5]
        \arrow["{=}"{description}, from=1-5, to=2-5]
        \arrow["{=}", from=2-1, to=2-2]
        \arrow[from=2-2, to=2-3]
        \arrow[from=2-3, to=2-4]
        \arrow["{=}", from=2-4, to=2-5]
      \end{mytikzcd}%
    }
  \]
  We remark that the lower row in the above diagram is included only for clarity.

  To see that this morphism equals that defined in Equation~\ref{confusion}, observe that the square
  \[
    \begin{mytikzcd}[ampersand replacement=\&]
      {U(FU(X) \otimes F(W))} \&\& {UFU(X)\otimes UF(W)} \\
      {U(X\otimes F(W))} \&\& {U(X) \otimes UF(W)}
      \arrow["{{(U_{\mathsf{m}})}_{FU(X),F(W)}}", from=1-1, to=1-3]
      \arrow["{U(\varepsilon_{X} \otimes F(W))}", from=1-1, to=2-1]
      \arrow["{U(\varepsilon_{X})\otimes UF(W)}", from=1-3, to=2-3]
      \arrow["{{(U_{\mathsf{m}})}_{X,F(W)}}", from=2-1, to=2-3]
    \end{mytikzcd}
  \]
  commutes, since $U_{\mathsf{m}}$ is natural.

  To see that we have ${(T_{\mathsf{a}})}_{F(V),W} = {(T_{\mathbf{f}})}_{V,W}$,
  it suffices to substitute $X = F(V)$ in Equation~\ref{confusion}, and recall that $\mu = U \varepsilon F$.
\end{proof}

\begin{lemma}\label{sinfusion}
  Let $\mathcal{A}$ be a category, $f,g\from X \to Y$ a pair of parallel morphisms in $\mathcal{A}$ such that the coequalizer $\on{coeq}(f,g)$ exists, $G,H\from \mathcal{A} \to \mathcal{B}$ a pair of functors preserving that coequalizer, and $\tau\from G \nt H$ a natural transformation.
  If $\tau_{X}$ and $\tau_{Y}$ are invertible, then so is $\tau_{\on{coeq}(f,g)}$.
\end{lemma}

\begin{proposition}\label{diffusion}
  The right fusion operators of an opmonoidal monad $T$ on $\mathcal{W}$ are invertible
  if and only if
  the coherence morphisms for the oplax $\mathbf{EM}(T)$-module structure on $T$ of Lemma~\ref{delusion} are invertible.
  In other words, $T$ is a right Hopf monad
  if and only if
  it is a strong $\mathbf{EM}(T)$-module monad, when endowed with the structure of Lemma~\ref{delusion}.
\end{proposition}
\begin{proof}
  Since, by Lemma~\ref{delusion}, the right fusion operators for $T$ are a special case of the coherence morphisms for the $\mathbf{EM}(T)$-module functor structure on $T$,
  it is immediate that $T$ being a strong $\mathbf{EM}(T)$-module monad implies invertibility of the right fusion operators.

  To show the converse,
  assume that the right fusion operators are invertible,
  in particular, by Lemma~\ref{delusion}, the morphisms ${(T_{\mathsf{a}})}_{F(V),W}$ are invertible.
  Recall that for any \(T\)-module $(X, \nabla_X\from T(X) \to X) \in \mathbf{EM}(T)$,
  we have the isomorphism $X \cong \on{coeq}(T\nabla_X, \mu_{X})$ in $\mathbf{EM}(T)$,
  so $X$ is a \(U\)-split coequalizer of morphisms from $F(T(X))$ to $F(X)$, since $F(X)$ is the free $T$-module on $X$.
  Consider the natural transformation ${(T_{\mathsf{a}})}_{\blank,W}\from T(\blank \triangleright W) \to \blank \triangleright T(W)$.
  We have
  \[
    T(\blank \triangleright W) = T(U(\blank) \otimes W) = (\blank \otimes W) \circ U.
  \]
  Thus, $T(\blank \triangleright W)$ preserves the coequalizer $\on{coeq}(T(\nabla_{X}), \mu_{X})$, by first sending it to a split coequalizer under $U$, and then preserving it, since split coequalizers are absolute.
  In a similar fashion, $\blank \triangleright T(W) = U(\blank) \otimes T(W) = (\blank \otimes T(W)) \circ U$ also preserves this coequalizer.
  The invertibility of ${(T_{\mathsf{a}})}_{X,W}$ now follows from
  the assumed invertibility of ${(T_{\mathsf{a}})}_{F(X),W}$ and ${(T_{\mathsf{a}})}_{F(T(X)),W}$, by Lemma~\ref{sinfusion}.
\end{proof}

\def\bibfont{\footnotesize}
\printbibliography%

\end{document}